\documentclass[11pt]{article}

\usepackage{amsmath, amsfonts, amssymb, amsthm,  graphicx, mathtools, enumerate,float}
\usepackage{enumitem}
\usepackage{blindtext}

\usepackage[blocks, affil-it]{authblk}

\usepackage[numbers, square]{natbib}
\usepackage[CJKbookmarks=true,
            bookmarksnumbered=true,
			bookmarksopen=true,
			colorlinks=true,
			citecolor=red,
			linkcolor=blue,
			anchorcolor=red,
			urlcolor=blue]{hyperref}
\usepackage[usenames]{color}

\usepackage[letterpaper, left=1.2truein, right=1.2truein, top = 1.2truein, bottom = 1.2truein]{geometry}
\usepackage[ruled, vlined, lined, commentsnumbered]{algorithm2e}

\usepackage{prettyref,soul}

\newtheorem{lemma}{Lemma}[section]
\newtheorem{proposition}{Proposition}[section]
\newtheorem{thm}{Theorem}[section]
\newtheorem{definition}{Definition}[section]
\newtheorem{corollary}{Corollary}[section]

\newrefformat{eq}{(\ref{#1})}
\newrefformat{chap}{Chapter~\ref{#1}}
\newrefformat{sec}{Section~\ref{#1}}
\newrefformat{algo}{Algorithm~\ref{#1}}
\newrefformat{fig}{Fig.~\ref{#1}}
\newrefformat{tab}{Table~\ref{#1}}
\newrefformat{rmk}{Remark~\ref{#1}}
\newrefformat{clm}{Claim~\ref{#1}}
\newrefformat{def}{Definition~\ref{#1}}
\newrefformat{cor}{Corollary~\ref{#1}}
\newrefformat{lmm}{Lemma~\ref{#1}}
\newrefformat{lemma}{Lemma~\ref{#1}}
\newrefformat{prop}{Proposition~\ref{#1}}
\newrefformat{app}{Appendix~\ref{#1}}
\newrefformat{ex}{Example~\ref{#1}}
\newrefformat{exer}{Exercise~\ref{#1}}
\newrefformat{soln}{Solution~\ref{#1}}
\newrefformat{cond}{Condition~\ref{#1}}

\def\text#1{\mbox{\rm #1}}
\def\overset#1#2{\stackrel{#1}{#2} }

\DeclarePairedDelimiter{\ceil}{\lceil}{\rceil}

\newcommand{\argmin}{\mathop{\rm argmin}}
\newcommand{\argmax}{\mathop{\rm argmax}}

\newcommand{\wh}{\widehat}
\newcommand{\wt}{\widetilde}

\newcommand{\R}{\mathbb{R}}
\newcommand{\E}{\mathbb{E}}

\newcommand{\TV}{{\sf TV}}

\newcommand{\supp}{{\rm supp}}
\newcommand{\iprod}[2]{\left \langle #1, #2 \right \rangle}

\newcommand{\floor}[1]{{\left\lfloor {#1} \right \rfloor}}
\newcommand\defeq{\stackrel{def}{=}}
\newcommand{\1}{\mathbf{1}}

\title{Convergence Rates of Variational Posterior Distributions
}
\author{Fengshuo Zhang}
\author{Chao Gao}
\affil{University of Chicago
}

\allowdisplaybreaks[2]
\begin{document}
\maketitle

\begin{abstract}
We study convergence rates of variational posterior distributions for nonparametric and high-dimensional inference. We formulate general conditions on prior, likelihood, and variational class that characterize the convergence rates. Under similar ``prior mass and testing" conditions considered in the literature, the rate is found to be the sum of two terms. The first term stands for the convergence rate of the true posterior distribution, and the second term is contributed by the variational approximation error. For a class of priors that admit the structure of a mixture of product measures, we propose a novel prior mass condition, under which the variational approximation error of the  mean-field class is dominated by convergence rate of the true posterior. We demonstrate the applicability of our general results for various models, prior distributions and variational classes by deriving convergence rates of the corresponding variational posteriors.

\smallskip

\end{abstract}



\section{Introduction}\label{sec:intro}

Variational Bayes inference is a popular technique to approximate difficult-to-compute probability posterior distributions. Given a posterior distribution $\Pi(\cdot|X^{(n)})$, and a variational family $\mathcal{S}$,  variational Bayes inference seeks a $\wh{Q}\in\mathcal{S}$ that best approximates $\Pi(\cdot|X^{(n)})$ under the Kullback-Leibler divergence. Though it is not exact Bayes inference, the variational class $\mathcal{S}$ often gives computational advantage and leads to algorithms such as coordinate ascent that can be efficiently implemented on large-scale data sets. Researchers in many fields have used variational Bayes inference to solve real problems. Successful examples include statistical genetics \citep{carbonetto2012scalable,raj2014faststructure}, natural language processing \citep{blei2003latent,lafferty2006correlated}, computer vision \citep{sudderth2009shared}, and network analysis \citep{bickel2013asymptotic,zhang2017theoretical}, to name a few. We refer the readers to an excellent recent review \citep{blei2017variational} on this topic.

The goal of this paper is to study the variational posterior distribution $\wh{Q}$ from a theoretic perspective. We propose general conditions on the prior, the likelihood and the variational class to characterize the convergence rate of the variational posterior to the true data generating process.

Before discussing our results, we give a brief review on the theory of convergence rates of the posterior distributions in the literature. In order that the posterior distribution concentrates around the true parameter with some rate, the ``prior mass and testing" framework requires three conditions on the prior and the likelihood:
a) The prior is required to put a minimal amount of mass in a neighborhood of the true parameter;
b) Restricted to a subset of the parameter space, there exists a testing function that can distinguish the truth from the complement of its neighborhood;
c) The prior is essentially supported on the subset described above. Rigorous statements of these three conditions can be found in seminal papers \citep{ghosal2000convergence,shen2001rates,ghosal2007convergence}. Earlier versions of these conditions go back to \cite{schwartz1965bayes,lecam1973convergence,barron1988exponential,barron1999consistency}. We also mention another line of work \cite{zhang2006,walker2007rates,castillo2014bayesian,hoffmann2015adaptive} that established posterior rates of convergence using other approaches.

In this paper, we show that under almost the same three conditions, the variational posterior $\wh{Q}$ also converges to the true parameter, and the rate of convergence is given by
\begin{equation}
\epsilon_n^2+\frac{1}{n}\inf_{Q\in\mathcal{S}}P_0^{(n)}D(Q\|\Pi(\cdot|X^{(n)})).\label{eq:overall-rate}
\end{equation}
The first term $\epsilon_n^2$ is the rate of convergence of the posterior distribution $\Pi(\cdot|X^{(n)})$. The second term is the variational approximation error with respect to the class $\mathcal{S}$ under the data generating process $P_0^{(n)}$. Since we are able to generalize the ``prior mass and testing" theory with the same old conditions, many well-studied problems in the literature can now be revisited under our framework of variational Bayes inference with very similar proof techniques. This will be illustrated with several examples considered in the paper.

Remarkably, for a special class of prior distributions and a corresponding variational class, the second term of (\ref{eq:overall-rate}) will be automatically dominated by $\epsilon_n^2$ under a modified ``prior mass" condition. We illustrate this result by a prior distribution of product measure
$$d\Pi(\theta)=\prod_jd\Pi_j(\theta_j),$$
and a mean-field variational class
$$\mathcal{S}_{\rm MF}=\left\{Q: dQ(\theta)=\prod_jdQ_j(\theta_j)\right\}.$$
As long as there exists a subset $\otimes_j\wt{\Theta}_j\subset\left\{\theta: D_{\rho}\left(P_0^{(n)}\|P_{\theta}^{(n)}\right)\leq C_1n\epsilon_n^2\right\}$, such that the prior mass condition
\begin{equation}
\Pi\left(\otimes_j\wt{\Theta}_j\right) \geq \exp\left(-C_2n\epsilon_n^2\right)\label{eq:super-prior-mass}
\end{equation}
holds together with the testing conditions, then the variational posterior distribution $\wh{Q}$ converges to the true parameter with the rate $\epsilon_n^2$. In other words, the variational approximation error term in (\ref{eq:overall-rate}) is dominated under this stronger prior mass condition (\ref{eq:super-prior-mass}). This is the result of Theorem \ref{thm:mean field}. Here, $D_{\rho}(\cdot\|\cdot)$ stands for a R\'enyi divergence with some $\rho>1$. The implication of the condition (\ref{eq:super-prior-mass}) is important. It says that as long as the prior satisfies a ``prior mass" condition that is coherent with the structure of the variational class, the resulted variational approximation error will always be small compared with the statistical error from the true posterior. Therefore, the condition (\ref{eq:super-prior-mass}) offers a practical guidance on how to choose a good prior for variational Bayes inference. In addition, as a condition only on the prior mass, (\ref{eq:super-prior-mass}) is usually very easy to check. This mathematical simplicity is not just for independent priors and the mean-field class. In Section \ref{sec:var-con}, a more general condition is proposed that includes the setting of (\ref{eq:super-prior-mass}) as a special case.

Besides the general formulation of conditions to ensure convergence of the variational posteriors, several interesting aspects of variational Bayes inference are also discussed in the paper. We show that for a general likelihood with a sieve prior, its mean-field variational approximation of the posterior distribution has an interesting relation to an empirical Bayes procedure. We also show that the empirical Bayes procedure is exactly a variational Bayes procedure using a specially designed variational class. This connection between empirical Bayes and variational Bayes is interesting, and may suggest similar theoretical properties of the two.

Finally, we would like to remark that the general rate (\ref{eq:overall-rate}) for variational posteriors is only an upper bound. It is \textit{not} always true that the variational posterior has a slower convergence rate than the true posterior. Sometimes the variational posterior may not be a good approximation to the true posterior, but it can still contract faster to the true parameter if additional regularity is imposed by the variational class $\mathcal{S}$. We construct examples in Section \ref{sec:better} to illustrate this point.

\paragraph{Related Work}
Statistical properties of variational posterior distributions have also been studied in the literature. A recent work by \cite{wang2017frequentist} established Bernstein-von Mises type of results for parametric models. We refer the readers to \cite{blei2017variational,wang2017frequentist} for other related references on theories for parametric variational Bayes inference. For nonparametric and high-dimensional models, recent work by \cite{alquier2017concentration,yang2017alpha} studied variational approximation to tempered posteriors, where the likelihood $dP_{\theta}^{(n)}/dP_0^{(n)}$ is replaced by $\left(dP_{\theta}^{(n)}/dP_0^{(n)}\right)^{\alpha}$ for some $\alpha\in(0,1)$. Just as the convergence of tempered posteriors \citep{walker2001bayesian}, the convergence of the variational approximation can also be established under generalizations of the prior mass condition. In addition, the paper \cite{alquier2017concentration} also studied convergence rates under model misspecification, and the paper \cite{yang2017alpha} considered a more general setting that can handle latent variables, which is quite useful to analyze mixture models. We would like to point out that these results do not apply to the usual posterior distributions with $\alpha=1$. After the first version of our paper was posted, similar results on $\alpha=1$ have also been obtained independently by \cite{pati2017statistical}\footnote{Some extensions of the results of \cite{pati2017statistical} were later added in the revised version of \cite{yang2017alpha} by the same authors.}.
An early related work on this topic is by \cite{zhang2006}, where the results cover both posterior distributions and their variational approximations. However, the conditions in \cite{zhang2006} are rather abstract and are not easy to check in applications.

\paragraph{Organization}
The rest of the paper is organized as follows. In Section \ref{sec:method}, we formulate the problem and introduce the general conditions that characterize convergence rates of variational posteriors. This section also includes results for the mean-field variational class, where the variational approximation error can be explicitly analyzed. In Section \ref{sec:application}, we apply our general theory to three examples that use three different variational classes. Then, in Section \ref{sec:var-con}, for a general class of prior distributions and a mean-field class under a model selection setting, we propose a new prior mass condition that leads to an automatic control of the variational approximation error. In Section \ref{sec:discussion}, we discuss the relation between variational Bayes and empirical Bayes. We also discuss possible situations where the variational posterior outperforms the true posterior in this section. An extension of the main results under model misspecification is also discussed in Section \ref{sec:discussion}. All the proofs will be given in the Appendix.

\paragraph{Notation}
We close this section by introducing notations that will be used later. For $a,b\in\mathbb{R}$, let $a\vee b=\max(a,b)$ and $a\wedge b=\min(a,b)$. For a positive real number $x$, $\ceil{x}$ is the smallest integer no smaller than $x$ and $\floor{x}$ is the largest integer no larger than $x$. For two positive sequences $\{a_n\}$ and $\{b_n\}$, we write $a_n\lesssim b_n$ or $a_n=O(b_n)$ if $a_n\leq Cb_n$ for all $n$ with some constant $C>0$ that does not depend on $n$. The relation $a_n\asymp b_n$ holds if both $a_n\lesssim b_n$ and $b_n\lesssim a_n$ hold. For an integer $m$, $[m]$ denotes the set $\{1,2,...,m\}$. Given a set $S$, $|S|$ denotes its cardinality, and $\mathbf{1}_S$ is the associated indicator function. The $\ell_p$ norm of a vector $v\in\R^m$ with $1\leq m\leq\infty$ is defined as $\|v\|_p = \left(\sum_{j=1}^m|v_j|^p\right)^{1/p}$ for $1\leq p<\infty$ and $\|v\|_\infty = \sup_{1\leq k\leq m}|v_k|$. Moreover, we use $\|v\|$ to denote the $\ell_2$ norm $\|v\|_2$ by convention. For any function $f$, the $\ell_p$ norm is defined in a similar way, i.e. $\|f\|_p = \left(\int f(x)^pdx\right)^{1/p}$. Specifically, $\|f\|_\infty = \sup_{x}|f(x)|$. We use $\mathbb{P}$ and $\mathbb{E}$ to denote generic probability and expectation whose distribution is determined from the context. The notation $\mathbb{P}f$ also means expectation of $f$ under $\mathbb{P}$ so that $\mathbb{P}f=\int fd\mathbb{P}$. Throughout the paper, $C$, $c$ and their variants denote generic constants that do not depend on $n$. Their values may change from line to line.

\section{Main Results}\label{sec:method}

\subsection{Definitions and Settings}

We start this section by introducing a class of divergence functions.

\begin{definition}[R\'enyi divergence]\label{def:divergence}
Let $\rho>0$ and $\rho\neq 1$. The $\rho$-R\'enyi divergence between two probability measures $P_1$ and $P_2$ is defined as
\[D_\rho(P_1\|P_2) = \left\{\begin{array}{ll}\frac{1}{\rho-1}\log\int\left(\frac{dP_1}{dP_2}\right)^{\rho-1}dP_1,& \mbox{if $ P_1\ll P_2$},\\
+\infty,&\mbox{otherwise}.\\
\end{array}\right.\]
\end{definition}

The relations between the R\'enyi divergence and other divergence functions are summarized below.
\begin{enumerate}
\item When $\rho\rightarrow 1$, the R\'enyi divergence converges to the Kullback-Leibler divergence, defined as
\[D_1(P_1\| P_2) = \left\{\begin{array}{ll}\int\log\left(\frac{dP_1}{dP_2}\right)dP_1,& \mbox{if $ P_1\ll P_2$},\\
+\infty,&\mbox{otherwise}.\\
\end{array}\right.\]
From now on, we use $D(P_1\| P_2)$ without the subscript to denote $D_1(P_1\| P_2)$.
\item When $\rho = 1/2$, the R\'enyi divergence is related to the Hellinger distance by
\[D_{1/2}(P_1\|P_2) = -2\log(1-H(P_1, P_2)^2),\]
and the Hellinger distance is defined as
\[H(P_1, P_2) = \sqrt{\frac{1}{2}\int(\sqrt{dP_1}-\sqrt{dP_2})^2}.\]
\item When $\rho = 2$, the R\'enyi divergence is related to the $\chi^2$-divergence by
\[D_2(P_1\|P_2) = \log(1+\chi^2(P_1\|P_2)),\]
and the $\chi^2$-divergence  is defined as
\[\chi^2(P_1\|P_2) = \int\frac{(dP_1)^2}{dP_2}-1.\]
\end{enumerate}

\begin{definition}[total variation]
The total variation distance between two probability measures $P_1$ and $P_2$ is defined as
$$\TV(P_1,P_2)=\frac{1}{2}\int\left|dP_1-dP_2\right|.$$
\end{definition}

The relation among the divergence functions defined above is given by the following proposition (see \cite{van2014renyi}).
\begin{proposition}\label{prop:div}
With the above definitions, the following inequalities hold,
\begin{eqnarray*}
&&TV(P_1,P_2)^2\leq2H(P_1, P_2)^2\leq D_{1/2}(P_1\|P_2)\\
&& \leq D(P_1\|P_2)\leq D_2(P_1\|P_2)\leq \chi^2(P_1\|P_2).
\end{eqnarray*}
Moreover, the R\'enyi divergence $D_{\rho}(P_1\|P_2)$ is a non-decreasing function of $\rho$.
\end{proposition}

Now we are ready to introduce the variational posterior distribution. Given a statistical model $P_{\theta}^{(n)}$ parametrized by $\theta$, and a prior distribution $\theta\sim \Pi$, the posterior distribution is defined by
\[d\Pi(\theta|X^{(n)}) = \frac{dP_\theta^{(n)}(X^{(n)})d\Pi(\theta)}{\int dP_\theta^{(n)}(X^{(n)})d\Pi(\theta)}.\]
To address possible computational difficulty of the posterior distribution,
variational approximation is a way to find the closest object in a class $\mathcal{S}$ of probability measures to $\Pi(\cdot|X^{(n)})$. The class $\mathcal{S}$ is usually required to be computationally or analytically tractable. The most popular mathematical definition of variational approximation is given through the KL-divergence.

\begin{definition}[variational posterior]\label{def:VB}

Let $\mathcal{S}$ be a family of distributions. The variational approximation of the posterior is defined as
\begin{equation}
\wh Q = \argmin_{Q\in\mathcal{S}}D(Q\|\Pi(\cdot|X^{(n)}))\label{eq:VB}.
\end{equation}
\end{definition}

Just like the posterior distribution $\Pi(\cdot|X^{(n)})$, the variational posterior $\wh{Q}$ is a data-dependent measure that summarizes information from both the prior and the data. For a variational set $\mathcal{S}$, the corresponding variational posterior can be regarded as the projection of the true posterior onto $\mathcal{S}$ under KL-divergence. When $\mathcal{S}$ is the set of all distributions, $\wh{Q}$ turns out to be the true posterior $\Pi(\cdot|X^{(n)})$. The choice of the class $\mathcal{S}$ usually determines the difficulty of the optimization (\ref{eq:VB}). In this paper, our main goal is to study the statistical property of the data-dependent measure $\wh{Q}$ for a general $\mathcal{S}$.

\subsection{Results for General Variational Posteriors}

Assume the observation $X^{(n)}$ is generated from a probability measure $P_0^{(n)}$, and $\wh Q$ is the variational posterior distribution driven by $X^{(n)}$. The goal of this paper is to analyze $\wh Q$ from a frequentist perspective. In other words, we study statistical properties of $\wh{Q}$ under $P_0^{(n)}$. The first theorem gives conditions that guarantee convergence of the variational posterior $\wh Q$.

%
%
%

\begin{thm}\label{thm:convergence}
Suppose $\epsilon_n$ is a sequence that satisfies $n\epsilon_n^2\geq 1$. Consider a loss function $L(\cdot,\cdot)$, such that for any two probability measures $P_1$ and $P_2$, $L(P_1, P_2)\geq 0$. Let $C, C_1, C_2, C_3>0$ be constants such that $C>C_2+C_3+2$. We assume
\begin{itemize}
\item For any $\epsilon >\epsilon_n$, there exists a set $\Theta_n(\epsilon)$ and a testing function $\phi_n$, such that
\begin{equation}\label{eq:C1}
P_0^{(n)}\phi_n+\sup_{\substack{\theta\in\Theta_n(\epsilon)\\ L(P_\theta^{(n)}, P_0^{(n)})\geq C_1n\epsilon^2}}P_\theta^{(n)}(1-\phi_n)\leq \exp(-Cn\epsilon^2).\tag{C1}
\end{equation}
\item For any $\epsilon >\epsilon_n$, the set $\Theta_n(\epsilon)$ above satisfies
\begin{equation}\label{eq:C2}
\Pi(\Theta_n(\epsilon)^c)\leq\exp(-Cn\epsilon^2).\tag{C2}
\end{equation}
\item
For some constant $\rho>1$,
\begin{equation}\label{eq:C3}
\Pi\left(D_{\rho}(P_0^{(n)}\|P_\theta^{(n)})\leq C_3n\epsilon_n^2\right)\geq\exp(-C_2n\epsilon_n^2).\tag{C3}
\end{equation}
\end{itemize}
Then for the variational posterior $\wh Q$ defined in (\ref{eq:VB}), we have
\begin{equation}\label{eq:mean convergence}
P_0^{(n)}\wh QL(P_\theta^{(n)}, P_0^{(n)})\leq M n(\epsilon_n^2+\gamma_n^2),
\end{equation}
for some constant $M$ only depending on $C_1,C$ and $\rho$,
where the quantity $\gamma_n^2$ is defined as
\[\gamma_n^2 = 
\frac{1}{n}\inf_{Q\in\mathcal{S}}P_0^{(n)}D(Q||\Pi(\cdot|X^{(n)})).\]
\end{thm}

Conditions (\ref{eq:C1})-(\ref{eq:C3}) resemble the three conditions of ``prior mass and testing" in \cite{ghosal2000convergence}. Interestingly, Theorem \ref{thm:convergence} shows that with a slight modification, these three conditions also lead to the convergence of the variational posterior. The testing conditions (\ref{eq:C1}) and (\ref{eq:C2}) are required to hold for all $\epsilon>\epsilon_n$. In the prior mass condition (\ref{eq:C3}), the neighborhood of $P_0^{(n)}$ is defined through a R\'enyi divergence with a $\rho>1$, compared with the KL-divergence used in \cite{ghosal2000convergence}. According to Proposition \ref{prop:div}, $D_{\rho}(P_1\|P_2)\geq D(P_1\|P_2)$ for $\rho>1$, so the condition (\ref{eq:C3}) in our paper is slightly stronger than that in \cite{ghosal2000convergence}. This stronger ``prior mass" condition ensures that the loss $L(P_\theta^{(n)}, P_0^{(n)})$ is exponentially integrable under the true posterior $\Pi(\cdot|X^{(n)})$, which is a key step in the proof of Theorem \ref{thm:convergence}. In all the examples considered in this paper, we will check (\ref{eq:C3}) with $D_2(P_0^{(n)}\|P_{\theta}^{(n)})$, which turns out to be a very convenient choice.

The convergence rate is the sum of two terms, $\epsilon_n^2$ and $\gamma_n^2$. The first term $\epsilon_n^2$ is the convergence rate of the true posterior $\Pi(\cdot|X^{(n)})$.
The second term $\gamma_n^2$  characterizes the approximation error given by the variational set $\mathcal{S}$. A larger $\mathcal{S}$ means more expressive power given by the variational approximation, and thus the rate of $\gamma_n^2$ is smaller.

It is worth mentioning that we characterize the convergence of the variational posterior $\wh{Q}$ through the expected loss $P_0^{(n)}\wh QL(P_\theta^{(n)}, P_0^{(n)})$. Bounds for this quantity are also obtained by \cite{pati2017statistical} independently with a stronger testing condition on the entire space. We remark that convergence in $P_0^{(n)}\wh QL(P_\theta^{(n)}, P_0^{(n)})$ automatically implies that the entire variational posterior distribution concentrates in a neighborhood of the true distribution $P_0^{(n)}$ with a radius of the same rate. When the loss function is convex, it also implies the existence of a point estimator that enjoys the same convergence rate. We summarize these results in the next corollary.

\begin{corollary}\label{cor:convergence}
Under the same setting of Theorem \ref{thm:convergence}, for any diverging sequence $M_n\rightarrow\infty$, we have
\[P_0^{(n)}\wh Q\left(L(P_\theta^{(n)}, P_0^{(n)})>M_nn(\epsilon_n^2+\gamma_n^2)\right)\rightarrow 0.\]
Furthermore, if the loss $L(P_\theta^{(n)}, P_0^{(n)})$ is convex respect to $\theta$, then the variational posterior mean $\wh{\theta}=\wh{Q}\theta$ satisfies
\[P_0^{(n)}L(P_{\wh \theta}^{(n)}, P_0^{(n)})\leq Mn(\epsilon_n^2+\gamma_n^2),\]
where $M$ is the same constant in (\ref{eq:mean convergence}).
\end{corollary}
\begin{proof}
The first result is an application of Markov's inequality
$$P_0^{(n)}\wh Q\left(L(P_\theta^{(n)}, P_0^{(n)})>M_nn(\epsilon_n^2+\gamma_n^2)\right)\leq \frac{P_0^{(n)}\wh QL(P_\theta^{(n)}, P_0^{(n)})}{M_nn(\epsilon_n^2+\gamma_n^2)}\leq \frac{M}{M_n}\rightarrow 0.$$
The second result is directly implied by Jensen's inequality that
$$P_0^{(n)}L(P_{\wh Q\theta}^{(n)}, P_0^{(n)})\leq P_0^{(n)}\wh QL(P_\theta^{(n)}, P_0^{(n)})\leq Mn(\epsilon_n^2+\gamma_n^2).$$
\end{proof}

To apply Theorem \ref{thm:convergence} to specific problems, we need to analyze the variational approximation error $\gamma_n^2=\frac{1}{n}\inf_{Q\in\mathcal{S}}P_0^{(n)}D(Q||\Pi(\cdot|X^{(n)}))$ in each individual setting. However, this task may not be trivial for many problems. Now we borrow a technique in \cite{zhang2006} to get a useful upper bound for $\gamma_n^2$. For any $Q\in\mathcal{S}$, we have
\begin{eqnarray*}
n\gamma_n^2&\leq&P_0^{(n)}D(Q\|\Pi(\cdot|X^{(n)})) = D(Q\|\Pi)+Q\left[\int\log\left(\frac{d P_{\Pi}^{(n)}}{dP_\theta}\right)dP_0^{(n)}\right]\\
& = &D(Q\|\Pi)+Q\left[D(P_0^{(n)}\|P_\theta^{(n)})-D(P_0^{(n)}\|P_\Pi^{(n)})\right]\\
&\leq& D(Q\|\Pi)+Q\left[D(P_0^{(n)}\|P_\theta^{(n)})\right],\\
\end{eqnarray*}
where $P_{\Pi}^{(n)} = \int P_\theta^{(n)}d\Pi(\theta)$. Then, we obtain the upper bound
$$\gamma_n^2\leq \inf_{Q\in\mathcal{S}}R(Q),$$
where 
\begin{equation}\label{eq:def risk}
R(Q) = \frac{1}{n}\left(D(Q\|\Pi)+Q\left[D(P_0^{(n)}\|P_\theta^{(n)})\right]\right).
\end{equation}
Now, it is easy to see that a sufficient condition for the variational posterior to converge at the same rate as the true posterior is
\begin{equation}\label{eq:C4}
\inf_{Q\in\mathcal{S}}R(Q)\lesssim \epsilon_n^2.\tag{C4}
\end{equation}
We incorporate this condition into the next theorem.

\begin{thm}\label{thm:general2}
Suppose $\epsilon_n$ is a sequence that satisfies $n\epsilon_n^2\geq 1$, for which the conditions (\ref{eq:C1}), (\ref{eq:C2}), (\ref{eq:C3}), (\ref{eq:C4}) hold. Then, for the variational posterior $\wh Q$ that is defined in (\ref{eq:VB}), we have 
\begin{equation}\label{eq:mean convergence1}
P_0^{(n)}\wh QL(P_\theta^{(n)}, P_0^{(n)})\lesssim n\epsilon_n^2.
\end{equation}
\end{thm}

We would like to remark that the quantity $\inf_{Q\in\mathcal{S}}R(Q)$ is easier to analyze compared with the original definition of $\gamma_n^2$. According to its definition given by (\ref{eq:def risk}), it is sufficient to find a distribution $Q\in\mathcal{S}$, such that
\begin{equation}
D(Q\|\Pi)\lesssim n\epsilon_n^2\quad\text{and}\quad Q\left[D(P_0^{(n)}\| P_\theta^{(n)})\right]\lesssim n\epsilon_n^2. \label{eq:alquier}
\end{equation}
These are exactly the two conditions formulated by \cite{alquier2017concentration} as a natural extension of the prior mass condition. The relation between the prior mass condition and (\ref{eq:alquier}) has also been discussed in \cite{yang2017alpha}.

One way to construct such a distribution $Q$ that satisfies the above two inequalities is to focus on those whose supports are within the set $\mathcal{C} = \{\theta: D(P_0^{(n)}\|P_\theta^{(n)})\leq Cn\epsilon_n^2\}$ for some constant $C>0$. We summarize this method into the following theorem.

\begin{thm}\label{thm:convergence2}
Suppose there exist constants $C_1,C_2 > 0$, such that
\begin{equation}\label{eq:C4*}
\inf_{\substack{Q\in\mathcal{S}\cap\mathcal{E}}} D(Q\|\Pi)\leq C_1n\epsilon_n^2,\tag{C4*}
\end{equation}
where $\mathcal{E} = \{Q: \supp(Q)\subset\mathcal{C}\}$ with $\mathcal{C} = \{\theta: D(P_0^{(n)}\|P_\theta^{(n)})\leq C_2n\epsilon_n^2\}$. Then, we have
\[\inf_{Q\in\mathcal{S}}R(Q)\leq (C_1+C_2)\epsilon_n^2.\]
\end{thm}

\subsection{Results for Mean-Field Variational Posteriors}

A special choice of $\mathcal{S}$ is the mean-field class of distributions. Not only does this class leads to computationally efficient algorithms such as coordinate ascent, but in this section, we will also show that the structure of this class leads to a convenient convergence analysis. We begin with its definition.

\begin{definition}[mean-field class]\label{def:mean field}
For parameters in a product space that can be written as $\theta=(\theta_1,\theta_2,...,\theta_m)$ with some $1\leq m\leq \infty$, the mean-field variational family is defined as 
\[\mathcal{S}_{\rm MF} = \left\{Q: dQ(\mathbf{\theta}) = \prod_{j=1}^m dQ_j(\theta_j)\right\}.\]
\end{definition}

The following theorem can be viewed as an application of Theorem \ref{thm:convergence2} to the mean-field class.

\begin{thm}\label{thm:mean field}
Suppose there exists a $\wt{Q}\in\mathcal{S}_{\rm MF}$ and a subset $\otimes_{j=1}^m\wt{\Theta}_j$, such that
\begin{equation}
\otimes_{j=1}^m\wt{\Theta}_j\subset\left\{\theta: D(P_0^{(n)}\| P_\theta^{(n)})\leq C_1n\epsilon_n^2, \quad\log\frac{d\wt{Q}(\theta)}{d\Pi(\theta)}\leq C_2n\epsilon_n^2\right\},\label{eq:mfc1}
\end{equation}
and
\begin{equation}
-\sum_{j=1}^m\log\wt{Q}_j(\wt{\Theta}_j)\leq C_3n\epsilon_n^2,\label{eq:mfc2}
\end{equation}
for some constants $C_1,C_2,C_3>0$. Then, we have
$$\inf_{Q\in\mathcal{S}_{\rm MF}}R(Q)\leq (C_1+C_2+C_3)\epsilon_n^2.$$
\end{thm}

Note that the condition (\ref{eq:mfc2}) can also be written as
$$\wt{Q}\left(\otimes_{j=1}^m\wt{\Theta}_j\right)\geq \exp\left(-C_3n\epsilon_n^2\right).$$
In other words, Theorem \ref{thm:mean field} gives an interesting ``distribution mass" type of characterization for $\inf_{Q\in\mathcal{S}}R(Q)$. Checking (\ref{eq:mfc2}) is very similar to checking the ``prior mass" condition (\ref{eq:C3}), and is usually not hard in many examples. We only need to make sure that $\wt{Q}$ is not too far away from the prior $\Pi$ in the sense of (\ref{eq:mfc1}). In fact, if the prior $\Pi$ belongs to the class $\mathcal{S}_{\rm MF}$, then one can take $\wt{Q}=\Pi$, and the conditions of Theorem \ref{thm:mean field} simply become a ``prior mass" condition $\Pi\left(\otimes_{j=1}^m\wt{\Theta}_j\right)\geq \exp\left(-C_3n\epsilon_n^2\right)$, with the choice of $\otimes_{j=1}^m\wt{\Theta}_j$ being a subset of the KL-neighborhood $\left\{\theta: D(P_0^{(n)}\| P_\theta^{(n)})\leq C_1n\epsilon_n^2\right\}$. A more general characterization of the variational approximation error under model selection setting through a prior mass condition will be studied in Section \ref{sec:var-con}.

\section{Applications}\label{sec:application}

In this section, we consider several examples to illustrate the theory developed in Section \ref{sec:method}. 

\subsection{Gaussian Sequence Model}\label{sec:gsm}

Consider observations generated by a Gaussian sequence model,
\begin{equation}\label{eq:gauss}
Y_j = \theta_j +\frac{1}{\sqrt{n}}Z_j,\qquad Z_j\overset{i.i.d}{\sim}N(0,1),\qquad j\geq 1.
\end{equation}
We use the notation $P_{\theta}^{(n)}=\otimes_jN(\theta_j,n^{-1})$ for the distribution above. Our goal is to use variational Bayes methods to estimate the true parameter $\theta^*$ that belongs to the following Sobolev ball,
\begin{equation}\label{eq:sobolev}
\Theta_{\alpha}(B) = \left\{\theta=(\theta_j)_{j=1}^\infty: \sum_{j=1}^\infty j^{2\alpha}\theta_j^2\leq B^2\right\}.
\end{equation}
Here, the smoothness $\alpha>0$ and the radius $B>0$ are considered as constants throughout the paper.
The loss function for this problem is $L(P_{\theta}^{(n)}, P_{\theta^*}^{(n)}) = n\|\theta-\theta^*\|^2$, which is a natural choice for the Gaussian sequence model.

The prior distribution $\theta\sim\Pi$ is described through the following sampling process.
\begin{enumerate}
\item Sample $k\sim\pi$;
\item Conditioning on $k$, sample $\theta_j\sim f_j$ for all $j\in[k]$, and set $\theta_j=0$ for all $j>k$.
\end{enumerate}
In other words, the prior on $\theta$ is a mixture of product measures,
\begin{equation}
d\Pi(\theta)=\sum_{k=1}^{\infty}\pi(k)\prod_{j=1}^kf_j(\theta_j)\prod_{j>k}\delta_0(\theta_j)d\theta.\label{eq:refer-later-sieve}
\end{equation}
Priors of similar forms are also considered in \cite{rivoirard2012posterior,gao2016rate,gao2015general,rousseau2017asymptotic}. Direct calculation implies that the posterior is also in the form of a mixture of product measures.

Consider the variational posterior $\wh{Q}$ defined by (\ref{eq:VB}) with $\mathcal{S}=\mathcal{S}_{\rm MF}$. That is, we seek a data-dependent measure in a more tractable form of a product measure. 
In most cases, the variational posterior does not have a closed form and needs to be solved by coordinate ascent  algorithms \citep{blei2017variational}. However, for the Gaussian sequence model (\ref{eq:gauss}) with the prior distribution (\ref{eq:refer-later-sieve}), one can write down the exact form of the mean-field variational posterior distribution.
\begin{thm}\label{thm:VB of gauss}
Consider the variational posterior $\wh{Q}$ induced by the likelihood (\ref{eq:gauss}), the prior (\ref{eq:refer-later-sieve}) and the mean-field variational set $\mathcal{S}_{\rm MF}$.
The distribution $\wh{Q}$ is a product measure with the density of each coordinate specified by
\begin{equation}\label{eq:VB gauss}
q_j = \left\{\begin{array}{ll}\wt f_j, &j<\wt k,\\ \wt p\delta_0+(1-\wt p)\wt f_{\wt k},&j = \wt k,\\ \delta_0,&j>\wt k.\end{array}\right.
\end{equation}
where
$$\wt{f}_j(\theta_j)\propto f_j(\theta_j)\exp\left(-\frac{n}{2}(\theta_j-Y_j)^2\right),$$
$$\wt{p}=\frac{\pi(k-1|Y)}{\pi(k-1|Y)+\pi(k|Y)},$$
and
\begin{equation}
\wt{k}=\argmax_k\left(\pi(k-1|Y)+\pi(k|Y)\right).\label{eq:k-tilde}
\end{equation}
The number $\pi(k|Y)$ is the posterior probability of the model dimension, and according to Bayes formula, it is
$$\pi(k|Y)\propto \pi(k)\prod_{j\leq k}\int f_j(\theta_j)\exp\left(-\frac{n(\theta_j-Y_j)^2}{2}\right)d\theta_j \prod_{j>k}\exp\left(-\frac{nY_j^2}{2}\right).$$
\end{thm}
In other words, the mean-field variational posterior $\wh{Q}$ is nearly equivalent to a thresholding rule. It estimates all $\theta_j^*$ by $0$ after $\wt{k}$ and applies the usual posterior distribution for each coordinate before $\wt{k}$. A mixed strategy is applied to the $\wt{k}$th coordinate. The effective model dimension $\wt{k}$ is found in a data-driven way through (\ref{eq:k-tilde}).

Next, we will show that even though the posterior itself is not a product measure, using $\wh{Q}$ from the mean-field class still gives us a rate-optimal contraction result. 
The conditions on the prior distributions are summarized below.
\begin{itemize}
\item There exist some constants $C_1,C_2>0$ such that
\begin{equation}\label{eq:gauss1-condition1}
\sum_{j=k}^\infty\pi(j)\leq C_1\exp(-C_2k),\text{ for all }k.
\end{equation} 
\item There exist some constants $C_3,C_4>0$ such that for $k_0=\left\lceil\left(\frac{n}{\log n}\right)^{\frac{1}{2\alpha+1}}\right\rceil$,
\begin{equation}\label{eq:gauss1-condition2}
\pi(k_0)\geq C_3\exp(-C_4k_0\log k_0).
\end{equation}
\item For the $k_0$ defined above, there exist some constants $c_0\in\mathbb{R}$ and $c_1>0$ such that
\begin{equation}\label{eq:gauss1-condition3}
-\log f_j(x)\leq c_0+c_1j^{2\alpha+1}x^2,\qquad\text{for all $j\leq k_0$ and $x\in\mathbb{R}$}.
\end{equation}
\end{itemize}
These three conditions on $\Pi$ include a large class of prior distributions. We remark that even though (\ref{eq:gauss1-condition3}) involves $\alpha$, it does not mean that one needs to know $\alpha$ when defining the prior $\Pi$. For example, the choice that $\pi(k)\propto e^{-\tau k}$ and $f_j$ being $N(0,\sigma^2)$ for some constants $\tau,\sigma^2>0$ easily satisfies all the three conditions (\ref{eq:gauss1-condition1})-(\ref{eq:gauss1-condition3}).

Conditions (\ref{eq:gauss1-condition1})-(\ref{eq:gauss1-condition3}) will be used to derive the four conditions in Theorem \ref{thm:general2}. To be specific, (\ref{eq:C1}) and (\ref{eq:C2}) are consequences of (\ref{eq:gauss1-condition1}) (see Lemma \ref{lem:gauss1} in the appendix), and (\ref{eq:C3}) and (\ref{eq:C4}) can be derived from (\ref{eq:gauss1-condition2}) and (\ref{eq:gauss1-condition3}) (see Lemma \ref{lem:gauss1-2} in the appendix).
Then, by Theorem \ref{thm:general2}, we obtain the following result.

\begin{thm}\label{thm:gauss1}
Consider the prior $\Pi$ that satisfies (\ref{eq:gauss1-condition1})-(\ref{eq:gauss1-condition3}).
Then, for any $\theta^*\in\Theta_{\alpha}(B)$, we have 
\[P_{\theta^*}^{(n)}\wh Q\|\theta-\theta^*\|^2\lesssim n^{-\frac{2\alpha}{2\alpha+1}}(\log n)^{\frac{2\alpha}{2\alpha+1}},\]
where $\wh Q$ is the variational posterior defined by (\ref{eq:VB}) with $\mathcal{S}=\mathcal{S}_{\rm MF}$.
\end{thm}

It is well known that the minimax rate of estimating $\theta^*$ in $\Theta_{\alpha}(B)$ is $n^{-\frac{2\alpha}{2\alpha+1}}$ \citep{johnstone2011gaussian}. Using a mean-field variational posterior, we achieve the minimax rate up to a logarithmic factor. In fact, the following proposition demonstrates that this rate cannot be improved for a very general class of priors.
\begin{proposition}\label{prop:lower-GSM}
Consider the prior $\Pi$ specified in (\ref{eq:refer-later-sieve}). Assume that $\max_j\|f_j\|_{\infty}\leq a$ and $\pi(k)$ is nonincreasing over $k$. Then, we have
$$\sup_{\theta^*\in\Theta_{\alpha}(B)}P_{\theta^*}^{(n)}\wh Q\|\theta-\theta^*\|^2\gtrsim n^{-\frac{2\alpha}{2\alpha+1}}(\log n)^{\frac{2\alpha}{2\alpha+1}},$$
where $\wh Q$ is the variational posterior defined by (\ref{eq:VB}) with $\mathcal{S}=\mathcal{S}_{\rm MF}$.
\end{proposition}

On the other hand, the extra logarithmic factor can actually be removed by a rescaling of the prior. Details of this improvement are given in Appendix \ref{sec:remove-log}.

\subsection{Infinite Dimensional Exponential Families}\label{sec:inf-exp}

In this section, we study another interesting variational family. The Gaussian mean-field family is defined as
\begin{equation}\label{eq:gauss VB}
\mathcal{S}_{\rm G} = \left\{Q = \otimes_{j}N(\mu_j,\sigma_j^2): \mu_j\in\R, \sigma_j^2\geq 0\right\}.
\end{equation}
This class offers better interpretability of the results because every distribution in $\mathcal{S}_G$ is fully determined by a sequence of mean and variance parameters. Note that we allow $\sigma_j^2$ to be zero and $N(\mu_j,0)$ is understood as the delta measure $\delta_{\mu_j}$ on $\mu_j$.

The application of $\mathcal{S}_G$ is illustrated by an infinite dimensional exponential family model.
We define the probability measure $P_{\theta}$ by
\begin{equation}\label{eq:density model}
\frac{dP_{\theta}}{d\ell} = \exp\left(\sum_{j=0}^\infty\theta_j\phi_j-c(\theta)\right),
\end{equation}
where $\ell$ denotes the Lebesgue measure on $[0,1]$, $\phi_j$ is the $j$th Fourier basis function of $L^2[0,1]$, and $c(\theta)$ is given by
\[c(\theta) = \log\int_0^1\exp\left(\sum_{j=0}^\infty\theta_j\phi_j(x)\right)dx.\]
Since $\phi_0(x ) = 1$ and $\theta_0$ can take arbitrary values without changing $P_{\theta}$, we simply set $\theta_0=0$. In other words, $P_{\theta}$ is fully parameterized by $\theta=(\theta_1,\theta_2,...)$.
Given i.i.d. observations from $P_{\theta^*}^n$, our goal is to estimate $P_{\theta^*}$, where $\theta^*$ is assumed to belong to the Sobolev ball $\Theta_{\alpha}(B)$ defined in (\ref{eq:sobolev}). The loss function is chosen as $n$ times the squared Hellinger distance $L(P_{\theta}^{n},P_{\theta^*}^{n})=nH^2(P_\theta, P_{\theta^*})$.

We consider a prior distribution $\Pi$ that is similar to the one used in Section \ref{sec:gsm}. Its sampling process is described as follows.
\begin{enumerate}
\item Sample $k\sim\pi$;
\item Conditioning on $k$, sample $\theta_j\sim f_j$ for all $j\in[k]$, and set $\theta_j=0$ for all $j>k$.
\end{enumerate}
We impose the following conditions on the prior $\Pi$.
\begin{itemize}
\item There exist some constants $C_1,C_2>0$ such that
\begin{equation}\label{eq:density-condition1}
\sum_{j=k}^\infty\pi(j)\leq C_1\exp(-C_2k\log k),\text{ for all }k.
\end{equation} 
\item There exist some constants $C_3,C_4>0$ such that for $k_0=\left\lceil\left(\frac{n}{\log n}\right)^{\frac{1}{2\alpha+1}}\right\rceil$
\begin{equation}\label{eq:density-condition2}
\pi(k_0)\geq C_3\exp(-C_4k_0\log k_0).
\end{equation}
\item There exist some constants $c_0\in\mathbb{R}$ and $c_1,\beta>0$ such that
\begin{equation}\label{eq:density-condition3}
-\log f_j(x)\geq c_0+c_1|x|^{\beta},
\end{equation}
for all $x\in\R$ and $j\in[k_0]$ with $k_0$ defined above.
\item For the $k_0$ defined above, there exist some constants $c_0'\in\mathbb{R}$ and $c_1'>0$ such that
\begin{equation}\label{eq:density-condition4}
-\log f_j(x)\leq c_0'+c_1'j^{2\alpha+1}x^2,\qquad\text{for all $j\leq k_0$ and $x\in\mathbb{R}$.}
\end{equation}
\end{itemize}
The conditions (\ref{eq:density-condition1})-(\ref{eq:density-condition4}) are satisfied by a large class of prior distributions. For example, one can choose $k\sim\text{Poisson}(\tau)$ and $f_j$ being the density of $N(0,\sigma^2)$ for some constants $\tau,\sigma^2>0$, and then the four conditions are easily satisfied.

\begin{thm}\label{thm:density1}
Consider the prior $\Pi$ that satisfies (\ref{eq:density-condition1})-(\ref{eq:density-condition4}). Then, for any $\theta^*\in\Theta_{\alpha}(B)$ with some $\alpha>1/2$, we have
\[P_{\theta^*}^{n}\wh Q H^2(P_{\theta},P_{\theta^*})\lesssim n^{-\frac{2\alpha}{2\alpha+1}}(\log n)^{\frac{2\alpha}{2\alpha+1}},\]
where $\wh Q$ is the variational posterior defined by (\ref{eq:VB}) with $\mathcal{S}=\mathcal{S}_{\rm G}$.
\end{thm}

The theorem shows that the Gaussian mean-field variational posterior is able to achieve the minimax rate $n^{-\frac{2\alpha}{2\alpha+1}}$ up to a logarithmic factor. We remark that the same result also holds for the mean-field variational posterior defined with $\mathcal{S}_{\rm MF}$. This is because $\mathcal{S}_{\rm G}\subset\mathcal{S}_{\rm MF}$, and thus $\inf_{Q\in\mathcal{S}_{\rm MF}}R(Q)\leq \inf_{Q\in\mathcal{S}_{\rm G}}R(Q)$. Compared with the class $\mathcal{S}_{\rm MF}$, the objective function using the parametric family $\mathcal{S}_{\rm G}$ can be optimized by algorithms such as stochastic gradient descent over the parameters $(\mu_j,\sigma_j^2)$. The objective function can be greatly simplified according to the general mean-field solution given in Theorem \ref{thm:VB-sieve-general}.

\subsection{Piecewise Constant Model}\label{sec:pcm}

The previous two sections consider examples of the mean-field variational set and its variant. In this section, we use another example to illustrate a situation where the mean-field variational set only gives a trivial rate. On the other hand, we show that alternative variational classes with appropriate dependence structures are able to achieve the optimal rate.

We consider the following piecewise constant model,
\begin{equation}\label{eq:PC-model1}
X_i = \theta_i +\sigma Z_i,\quad i\in[n],
\end{equation}
where $Z_i\sim N(0,1)$ independently for all $i\in[n]$. We assume $n\geq 2$ throughout the section. The true parameter $\theta^*$ is assumed to belong to the class $\Theta_{k^*}(B)=\left\{\theta\in\Theta_{k^*}:\|\theta\|_{\infty}\leq B\right\}$,
where for a general $k\in[n]$,
\begin{eqnarray}
\nonumber\Theta_{k} &=& \Bigg\{\theta\in\R^n:\mbox{ there exist $\{a_j\}_{j=0}^k$ and $\{\mu_j\}_{j=1}^k$ such that}\\
\label{eq:PC-model2}&&0 = a_0\leq a_1\leq\cdots\leq a_k = n, \mbox{ and $\theta_i = \mu_j$ for all $i\in(a_{j-1}:a_j]$}\Bigg\}.
\end{eqnarray}
Here for any two integers $a<b$, we use $(a:b]$ to denote all integers from $a+1$ to $b$. We assume both $B>0$ and $\sigma^2>0$ are constants throughout this section. A vector $\theta^*\in\Theta_{k^*}(B)$ is a piecewise constant signal with at most $k^*$ pieces. We use $P_{\theta}^{(n)}$ to denote the probability distribution of $N(\theta,\sigma^2I_n)$ in this section.

The piecewise constant model is widely studied in the literature of change-point analysis. Recently, the minimax rate of the class $\Theta_{k^*}$ is derived by \cite{gao2017minimax}. When $2< k^*\leq n^{1-\delta}$ for some constant $\delta\in (0,1)$, the minimax rate is $\inf_{\wh{\theta}}\sup_{\theta^*\in\Theta_{k^*}}\mathbb{E}_{\theta^*}^{(n)}\|\wh{\theta}-\theta^*\|^2\asymp k^*\log n$. With an extra constraint on the infinity norm, the minimax rate for $\Theta_{k^*}(B)$ is still $k^*\log n$, with a slight modification of the proof in \cite{gao2017minimax}.
Since $D_\rho(P_\theta^{(n)}, P_{\theta'}^{(n)}) = \frac{\rho}{2\sigma^2}\|\theta-\theta'\|^2$ in this case, it is natural to choose the loss function as $L(P_\theta^{(n)}, P_{\theta^*}^{(n)}) = \|\theta-\theta^*\|^2$.

We put a prior distribution $\Pi$ on the parameter $\theta$. Consider $\Pi$ that has the following sampling process.
\begin{enumerate}
\item Sample $w\sim \text{Beta}(\alpha_0,\beta_0)$;
\item Conditioning on $w$,  sample $z_i\sim \text{Bernoulli}(w)$ for $i=2,3,...,n$;
\item Conditioning on $(z_2,...,z_n)$,  sample $\theta_1\sim g$, and then for $i=2,3,...,n$, sample $\theta_i$ according to $\theta_i\sim g$ if $z_i=1$ and $\theta_i=\theta_{i-1}$ if $z_i=0$.
\end{enumerate}

We first consider variational inference via the mean-field class, defined as
\[\mathcal{S}_{\rm MF} = \left\{Q:dQ(\theta) = \prod_{i=1}^ndQ_i(\theta_i)\right\}.\]
We also define $\mathcal{S} = \mathcal{S}_{\rm MF}^{\rm joint}$ on the joint distribution of $(w,z,\theta)$ by
\begin{eqnarray}
\nonumber\mathcal{S}_{\rm MF}^{\rm joint} &=& \Bigg\{Q: dQ(w,z,\theta)=dQ^{(w)}(w)dQ^{(z)}(z)dQ^{(\theta)}(\theta), \\
\nonumber&& \qquad\qquad dQ^{(z)}(z)=\prod_{i=2}^ndQ_i^{(z)}(z_i), Q^{(\theta)}\in \mathcal{S}_{\rm MF}\Bigg\}.
\end{eqnarray}
The variational posteriors $\wh Q_{\rm MF}$ and $\wh Q_{\rm MF}^{\rm joint}$ are given by (\ref{eq:VB}) with variational classes defined above respectively\footnote{To be rigorous, the posterior distribution $\Pi(\cdot|X^{(n)})$ used in $D(Q\|\Pi(\cdot|X^{(n)}))$ are the marginal posterior of $\theta$ and the joint posterior of $(w,z,\theta)$, respectively.}. Interestingly, for the piecewise constant model, both $\wh{Q}_{\rm MF}$ and $\wh{Q}_{\rm MF}^{\rm joint}$ give a trivial rate.

\begin{thm}\label{thm:PC-MF}
For the prior $\Pi$ specified above with any $g$ absolutely continuous with respect to the Lebesgue measure, we have
\[\sup_{\theta^*\in\Theta_{k^*}(B)}P_{\theta^*}^{(n)}\wh Q_{\rm MF}\|\theta-\theta^*\|^2= \sup_{\theta^*\in\Theta_{k^*}(B)}P_{\theta^*}^{(n)}\wh Q_{\rm MF}^{\rm joint}\|\theta-\theta^*\|^2\gtrsim  n,\]
for any $k^*\in[n]$, where $\wh Q_{\rm MF}$ and $\wh Q_{\rm MF}^{\rm joint}$ are the variational posteriors defined by (\ref{eq:VB}) with $\mathcal{S}=\mathcal{S}_{\rm MF}$ and $\mathcal{S} = \mathcal{S}_{\rm MF}^{\rm joint}$, respectively.
\end{thm}
The result of Theorem \ref{thm:PC-MF} shows that the mean-field variational posteriors $\wh{Q}_{\rm MF}$ and $\wh{Q}_{\rm MF}^{\rm joint}$ are unable to achieve a better rate than simply estimating $\theta^*$ by the naive estimator $\wh{\theta}=X$. The proof, given in Appendix \ref{subsec:PC-MF}, reveals the reason of this phenomenon. Since the independence structure of the two classes fails to capture the underlying dependence structure of the parameter space $\Theta_{k^*}(B)$, the variational posterior distributions are equivalent to the posterior distribution induced by the prior $\Pi = \otimes_{i=1}^n g$, and therefore the condition (\ref{eq:C4}) is violated. Note that this is the first negative result in the literature on the statistical convergence of the mean-field approximation.

In order to achieve the minimax rate of the space $\Theta_{k^*}(B)$, it is necessary to introduce some dependence structure in the variational class. One of the simplest classes of dependent distributions is the class of first-order Markov chains, defined by
$$\mathcal{S}_{\rm MC}=\left\{Q:dQ(\theta)=dQ_1(\theta_1)\prod_{i=2}^n dQ_i(\theta_i|\theta_{i-1})\right\}.$$
The class $\mathcal{S}_{\rm MC}$ introduces a natural dependence structure for the piecewise constant model, and it is compatible with the prior distribution $\Pi$, because conditioning on the change point pattern $z$, the prior distribution of $\theta | z$ belongs to the class $\mathcal{S}_{\rm MC}$.  We also introduce a similar variational class on the joint distribution of $(w,z,\theta)$, defined by
\begin{eqnarray}
\nonumber\mathcal{S}_{\rm MC}^{\rm joint} &=& \Bigg\{Q: dQ(w,z,\theta)=dQ^{(w)}(w)dQ^{(z)}(z)dQ^{(\theta)}(\theta), \\
\nonumber&& \qquad\qquad dQ^{(z)}(z)=\prod_{i=2}^ndQ_i^{(z)}(z_i), Q^{(\theta)}\in \mathcal{S}_{\rm MC}\Bigg\}.
\end{eqnarray}
Besides the distribution of $\theta$ restricted to $\mathcal{S}_{\rm MC}$, the distributions of $w$ and $z$ are both in the mean-field classes.

In order to derive the rates for the variational posterior distributions induced by $\mathcal{S}_{\rm MC}$ and $\mathcal{S}_{\rm MC}^{\rm joint}$, we impose the following conditions on the prior distribution $\Pi$.
\begin{itemize}
\item There exist some constants $C_2> C_1>1$ such that
\begin{equation}\label{eq:MC-condition1}
(n+\alpha_0)n^{C_1}\leq \beta_0 \leq \alpha_0n^{C_2}-n.
\end{equation}
\item There exists a constant $c>0$ such that 
\begin{equation}\label{eq:MC-condition2}
g(x)\geq c,\mbox{ for all $|x|\leq B+1$}.
\end{equation}
\end{itemize}
According to Theorem \ref{thm:general2}, we get the following result. 

\begin{thm}\label{thm:PC-MC}
Consider a prior distribution $\Pi$ that satisfies (\ref{eq:MC-condition1}) and (\ref{eq:MC-condition2}).
Then, for any $\theta^*\in\Theta_{k^*}(B)$, we have
\[P_{\theta^*}^{(n)}\wh Q_{\rm MC}\|\theta-\theta^*\|^2\lesssim k^*\log n,\]
\[P_{\theta^*}^{(n)}\wh Q_{\rm MC}^{\rm joint}\|\theta-\theta^*\|^2\lesssim k^*\log n,\]
where $\wh Q_{\rm MC}$ and $\wh Q_{\rm MC}^{\rm joint}$ are the variational posterior distributions defined by (\ref{eq:VB}) with $\mathcal{S}=\mathcal{S}_{\rm MC}$ and $\mathcal{S}=\mathcal{S}_{\rm MC}^{\rm joint}$, respectively.
\end{thm}
Theorem \ref{thm:PC-MC} shows that both $\wh{Q}_{\rm MC}$ and $\wh Q_{\rm MC}^{\rm joint}$ are able to achieve the minimax rate of the problem. This example illustrates the importance of the choice of the variational class. According to Theorem \ref{thm:convergence}, the  rate of a variational posterior is upper bounded by $\epsilon_n^2$, the rate of the true posterior, plus $\gamma_n^2$, the variational approximation error. The choice of $\mathcal{S}_{\rm MF}$ for the piecewise constant model leads to a very large $\gamma_n^2$, and thus a trivial rate in Theorem \ref{thm:PC-MF}. On the other hand, the variational approximation errors given by the classes $\mathcal{S}_{\rm MC}$ and $\mathcal{S}_{\rm MC}^{\rm joint}$ are small, which are dominated by the minimax rate.

Though the statistical properties of the two classes $\mathcal{S}_{\rm MC}$ and $\mathcal{S}_{\rm MC}^{\rm joint}$ are both satisfactory, the class $\mathcal{S}_{\rm MC}^{\rm joint}$ enjoys a computational advantage, and the solution $\wh Q_{\rm MC}^{\rm joint}$ can be computed exactly via dynamic programming. In order to characterize the solution $\wh Q_{\rm MC}^{\rm joint}$, we consider the following discrete optimization problem:
\begin{eqnarray}\label{eq:opt_MC}
\nonumber && \max_{1\leq k\leq n}\Bigg\{\max_{0=a_0< a_1<\cdots< a_k=n}\sum_{j=1}^k\log\int g(\theta)\exp\left(-\frac{1}{2}\sum_{i\in(a_{j-1}:a_j]}(X_i-\theta)^2\right)d\theta \\
\label{eq:knots} &&\qquad\qquad\qquad + \log\left(\Gamma(k-1+\alpha_0)\Gamma(n-k+\beta_0)\right)\Bigg\}.
\end{eqnarray}
The solution of (\ref{eq:knots}) is denoted as the sequence $0=\wh{a}_0< \wh{a}_1<\cdots <\wh{a}_{\wh{k}}=n$. We remark that under the condition (\ref{eq:MC-condition1}), the penalty term of (\ref{eq:knots}) comes from the fact that
$$-\log\frac{\Gamma(k-1+\alpha_0)\Gamma(n-k+\beta_0)\Gamma(\alpha_0+\beta_0)}{\Gamma(n-1+\alpha_0+\beta_0)\Gamma(\alpha_0)\Gamma(\beta_0)}\asymp k\log n,$$
which coincides with the minimax rate.
\begin{thm}\label{thm:PC-solution}
Let the maximizer of (\ref{eq:knots}) be $(\wh{a}_0,\wh{a}_1,...,\wh{a}_{\wh{k}})$.
For $$d\wh{Q}_{\rm MC}^{\rm joint}(w,z,\theta)=d\wh{Q}^{(w)}(w)d\wh{Q}^{(z)}(z)d\wh{Q}^{(\theta)}(\theta),$$ the distributions $\wh{Q}^{(w)}$, $\wh{Q}^{(z)}$ and $\wh{Q}^{(\theta)}$ are specified as follows.
\begin{enumerate}
\item Under $\wh{Q}^{(z)}$, $z_{\wh{a}_j+1}=1$ for $j=1,...,\wh{k}-1$, and $z_i=0$ elsewhere with probability $1$.
\item We have $\wh{Q}^{(w)}=\text{Beta}(\wh{k}+\alpha_0-1,n-\wh{k}+\beta_0)$.
\item We have $d\wh{Q}^{(\theta)}(\theta)=d\wh{Q}_1^{(\theta)}(\theta_1)\prod_{i=2}^nd\wh{Q}_i^{(\theta)}(\theta_i|\theta_{i-1})$, where
$$\begin{cases}
d\wh{Q}_1^{(\theta)}(\theta_1)\propto g(\theta_1)\exp\left(-\frac{1}{2}\sum_{i\in(\wh{a}_0:\wh{a}_1]}(X_i-\theta_1)^2\right)d\theta_1, \\
d\wh{Q}_i^{(\theta)}(\theta_i|\theta_{i-1}) \propto g(\theta_i)\exp\left(-\frac{1}{2}\sum_{l\in(\wh{a}_{j-1}:\wh{a}_j]}(X_l-\theta_i)^2\right)d\theta_i, & i=\wh{a}_{j-1}+1, j>1,\\
d\wh{Q}_i^{(\theta)}(\theta_i|\theta_{i-1}) = \delta_{\theta_{i-1}}(\theta_i)d\theta_i, & \text{otherwise}.
\end{cases}$$
\end{enumerate}
\end{thm}

By Theorem \ref{thm:PC-solution}, in order to get $\wh Q_{\rm MC}^{\rm joint}$, it is sufficient to solve (\ref{eq:knots}). This can be done through a dynamic programming given in Algorithm \ref{alg:dp}. To simplify the notation, we define
\begin{equation}
S_{(a:b]}=\log \int g(\theta)\exp\left(-\frac{1}{2}\sum_{i\in(a:b]}(X_i-\theta)^2\right)d\theta, \label{eq:partial-sum}
\end{equation}
for any integers $0\leq a< b\leq n$.

\begin{algorithm}[H]
\DontPrintSemicolon
\SetKwInOut{Input}{Input}\SetKwInOut{Output}{Output}
\Input{The data $X_1,...,X_n$.}
\Output{The set of knots $A_{\wh{k},n}=\{\wh a_1,\cdots,\wh a_{\wh k-1}\}$.} 
\nl For $j$ in $1:n$, set $A_{1,j}=\emptyset$, and compute\;
\qquad$B_{1,j}=S_{(0:j]}$.

\nl For $k$ in $2:n$\;
\qquad For $j$ in $k:n$, compute\;
\qquad\qquad $B_{k,j}=\max_{k-1\leq m\leq j-1}\left\{B_{k-1,m} + S_{(m:j]}\right\}$,\;
\qquad\qquad $a_{k,j}=\argmax_{k-1\leq m\leq j-1}\left\{B_{k-1,m} + S_{(m:j]}\right\}$,\;
\qquad\qquad $A_{k,j}=A_{k-1,a_{k,j}}\cup \{a_{k,j}\}$.\;
 
 \nl Compute\;
\qquad $\wh{k}=\argmax_{1\leq k\leq n}\left\{B_{k,n} + \log\left(\Gamma(k-1+\alpha_0)\Gamma(n-k+\beta_0)\right)\right\}$.\;
\caption{Computation of (\ref{eq:knots})}\label{alg:dp}
\end{algorithm}

We note that the computational cost of the dynamic programming above is $O(n^3)$ (see \cite{friedrich2008complexity}), and for any integers $0\leq a< b\leq n$, (\ref{eq:partial-sum}) has a closed form as long as we use a conjugate $g(\cdot)$.

\section{Variational Bayes with Model Selection}\label{sec:var-con}

\subsection{General Settings}\label{sec:gmf-setting}

In this section, we consider a general form of probability models
$$\mathcal{M} = \left\{P_{k,\theta^{(k)}}^{(n)}: k\in\mathcal{K}, \theta^{(k)}\in\Theta^{(k)}\right\}.$$
Here, the probability $P_{k,\theta^{(k)}}^{(n)}$ is determined by an index $k$ and a parameter $\theta^{(k)}$. We assume that the set $\mathcal{K}$ is either countable or finite. For a given $k$, the probability $P_{k,\theta^{(k)}}^{(n)}$ is parametrized by a $\theta^{(k)}$ in a parameter space $\Theta^{(k)}$ that is indexed by this $k$. Without loss of generality, we assume that the parameter $\theta^{(k)}$ can be written in a blockwise structure
$$\theta^{(k)} = (\theta_1^{(k)},\cdots, \theta_{m_k}^{(k)}).$$
Note that the dimension of $\theta^{(k)}$ may vary with $k$.

The model $\mathcal{M}$ is very natural for many applications. One can think of $k$ as a model dimension index, which determines the complexity of the parameter space $\Theta^{(k)}$. A leading example is the mixture density model, where $k$ stands for the number of components.

To model the hierarchical structure of $(k,\theta^{(k)})$, one naturally uses a hierarchical prior distribution, which is specified through the following sampling process:
\begin{enumerate}
\item Firstly, sample $k\sim \pi$ from $\mathcal{K}$;
\item Conditioning on $k$, sample $\theta^{(k)}$ from the probability measure $\Pi^{(k)}$, and $\Pi^{(k)}$ has a product structure
\begin{equation}\label{eq:prod structure}
d\Pi^{(k)}(\theta^{(k)}) = \prod_{j=1}^{m_k}d\Pi_j^{(k)}(\theta_j^{(k)}).
\end{equation}
\end{enumerate}

For variational inference, we consider a mean-field class that naturally takes advantage of the structure of the prior distribution. For a given $k\in\mathcal{K}$, the corresponding mean-field class is defined as
\begin{equation}\label{eq:GMF3}
\mathcal{S}_{\rm MF}^{(k)} =\left\{Q^{(k)}: dQ^{(k)}(\theta^{(k)}) = \prod_{j=1}^{m_k}dQ_j^{(k)}(\theta_j^{(k)})\right\}. 
\end{equation}
In order to select the best model from the data, we consider optimizing the evidence lower bound (ELBO). With the notation $p(X^{(n)}|\theta^{(k)})$ standing for the joint likelihood function, the marginal likelihood given a model $k\in\mathcal{K}$ is defined by
\begin{equation}
p(X^{(n)}|k)=\int p(X^{(n)}|\theta^{(k)})d\Pi^{(k)}(\theta^{(k)}).\label{eq:marginal-k}
\end{equation}
Then, a straightforward model selection procedure is to maximize $\log\left(p(X^{(n)}|k)\pi(k)\right)$ over $k\in\mathcal{K}$. In order to overcome the intractability of the integral (\ref{eq:marginal-k}), we instead optimize a lower bound, which is given by
\begin{eqnarray}
\nonumber && \log\left(p(X^{(n)}|k)\pi(k)\right) \\
\label{eq:ELBO-selection} &\geq& \int\log p(X^{(n)}|\theta^{(k)})dQ^{(k)}(\theta^{(k)}) - D\left(Q^{(k)}\|{\Pi}^{(k)}\right) + \log \pi(k),
\end{eqnarray}
which can be derived by a direct application of Jensen's inequality. Denote the right hand side of (\ref{eq:ELBO-selection}) by $F(Q^{(k)},k)$, and we will solve the following optimization problem,
\begin{equation}
\max_{k\in\mathcal{K}}\max_{Q^{(k)}\in \mathcal{S}_{\rm MF}^{(k)}}F(Q^{(k)},k). \label{eq:ELBO-select-object}
\end{equation}
Finally, the solution to (\ref{eq:ELBO-select-object}) leads to the variational posterior distribution $\wh{Q}=\wh{Q}^{(\wh{k})}$ that we use in a model selection context. A similar variational approximation to the tempered posterior in the model selection setting was studied by \cite{cherief2018consistency}.

\subsection{Convergence Rates}\label{sec:rates-ms}

Assume the observation $X^{(n)}$ is generated from a probability measure $P_0^{(n)}$, and $\wh Q=\wh{Q}^{(\wh{k})}$ is the variational posterior that is a solution to (\ref{eq:ELBO-select-object}). For the general settings described above, we show that the variational approximation error can be automatically controlled by a prior mass condition. Let $\Pi$ be the prior distribution on $P_{k,\theta^{(k)}}$ induced by the sampling process of $(k,\theta^{(k)})$.
\begin{thm}\label{thm:general-MF}
Suppose $\epsilon_n$ is a sequence that satisfies $n\epsilon_n^2\geq 1$. Let $\rho>1$ be a constant and $C_2, C_3>0$ be constants. We assume that there exists a $k_0\in\mathcal{K}$ and a subset $\Theta^{(k_0)} = \otimes_{j=1}^{m_{k_0}}\Theta_j^{(k_0)}\subset\left\{\theta^{(k_0)}: D_{\rho}\left(P_0^{(n)}\|P_{k_0, \theta^{(k_0)}}^{(n)}\right)\leq C_3n\epsilon_n^2\right\}$, such that
\begin{equation}\label{eq:C3*}
-\log\pi(k_0)-\sum_{j=1}^{m_{k_0}}\log\Pi_j^{(k_0)}\left(\Theta_j^{(k_0)}\right)\leq C_2n\epsilon_n^2,\tag{C3*}
\end{equation}
where $\pi(k_0)$ and $\Pi_j^{(k_0)}$ are defined in the prior sampling procedure. Moreover, assume that the conditions (\ref{eq:C1}) and (\ref{eq:C2}) hold for all $\epsilon>\epsilon_n$ with respect to prior procedure $\Pi$ and some constant $C>C_2+C_3+2$. 
Then for the variational posterior $\wh Q^{(\wh k)}$ defined as the solution of (\ref{eq:ELBO-select-object}), we have
\begin{equation}\label{eq:mean convergence1}
P_0^{(n)}\wh Q^{(\wh k)}L(P_{\wh{k},\theta^{(\wh{k})}}^{(n)},P_0^{(n)})\lesssim n\epsilon_n^2.
\end{equation}
\end{thm}

Theorem \ref{thm:general-MF} characterizes the convergence rate of mean-field variational posterior with model selection using the conditions (\ref{eq:C1}), (\ref{eq:C2}) and (\ref{eq:C3*}). Given the structure of the prior distribution, an equivalent way of writing (\ref{eq:C3*}) is
$$\Pi\left(\left\{P_{k,\theta^{(k)}}: k=k_0, \theta^{(k_0)}\in{\Theta}^{(k_0)}\right\}\right)\geq\exp\left(-C_2n\epsilon_n^2\right),$$
for the factorized structure of $\Theta^{(k_0)}$.
Therefore, our three conditions (\ref{eq:C1}), (\ref{eq:C2}) and (\ref{eq:C3*}) still fall into the ``prior mass and testing" framework, and directly correspond to the three conditions in \cite{ghosal2000convergence} for convergence rates of the true posterior.

An interesting special case is when the set $\mathcal{K}$ is a singleton. Then, for a product prior measure and the mean-field variational class, the condition (\ref{eq:C3*}) is reduced to (\ref{eq:super-prior-mass}) discussed in Section \ref{sec:intro}.

\subsection{Density Estimation via Location-Scale Mixtures}\label{sec:mixture-gmf}

In this section, we consider the location-scale mixture model as an application of the theory. The location-scale mixture density is defined as
\begin{equation}\label{eq:mixture}
p(x|k,\theta^{(k)}) = \sum_{j=1}^kw_j\psi_\sigma(x-\mu_j),
\end{equation}
where $k\in\mathbb{N}_+$, $\theta^{(k)} = (\mu, w,\sigma)$ with $\sigma>0$, $\mu = (\mu_1,\cdots,\mu_k)\in\R^k$, $w = (w_1,\cdots, w_k)\in\Delta_k = \left\{w\in\R^k: w_j\geq 0\mbox{ for $1\leq j\leq k$ and }\sum_{j=1}^kw_j = 1\right\}$ and
\begin{equation}\label{eq:kernel}
\psi_\sigma(x) = \frac{1}{2\sigma\Gamma\left(1+\frac{1}{p}\right)}\exp(-(|x|/\sigma)^p),
\end{equation} 
for some positive even integer $p$. The kernel $\psi_{\sigma}(\cdot)$ has a pre-specified form, for example, Gaussian density when $p=2$, while the parameters $k$ and $\theta^{(k)} = (w,\mu,\sigma)$ are to be learned from the data.

The location-scale mixture model (\ref{eq:mixture}) can be written as a special example of the general probability models introduced in Section \ref{sec:gmf-setting}.
In this case, the countable set $\mathcal{K}$ is the positive integer set $\mathbb{N}_+$. The parameter space indexed by $k$ is defined as
\begin{eqnarray}
\label{eq:space-gmf-mixture}\Theta^{(k)} &=& \Big\{\theta^{(k)} = (\mu,w,\sigma): \mu = (\mu_1,\cdots, \mu_k)\in\R^k, \\
\nonumber&& \qquad\qquad w = (w_1,\cdots, w_k)\in\Delta_k, \sigma\in\R_+\Big\}.
\end{eqnarray}

Given i.i.d. observations $X_1,...,X_n$ sampled from some density function $f_0$, our goal is to estimate the density $f_0$ through the location-scale mixture model (\ref{eq:mixture}). We denote the probability distribution of the mixture density $p(x|k,\theta^{(k)})$ as $P_{k,\theta^{(k)}}$ and a probability distribution with a general density $f$ as $P_{f}$. In the paper \cite{kruijer2010adaptive}, a Bayesian procedure is proposed and a nearly minimax optimal convergence rate is derived for the true posterior distribution. We will follow the same setting in \cite{kruijer2010adaptive}, but analyze the variational posterior.

We first specify the prior distribution $\Pi$ through the following sampling process:
\begin{enumerate}
\item Sample the number of mixtures $k\sim\pi$;
\item
Conditioning on $k$, sample the location parameters $\mu_1,\cdots, \mu_k$ independently from $p_\mu$, sample the weights $w = (w_1,\cdots, w_k)$ from $p_w^{(k)}$, and then sample the precision parameter $\tau=\sigma^{-2}$ from $p_{\tau}$.
\end{enumerate}

In order to optimize (\ref{eq:ELBO-select-object}) in the variational Bayes framework, we specify the blockwise structure (\ref{eq:GMF3}) in this case as
\begin{equation}\label{eq:GMF3*}
\mathcal{S}_{\rm MF}^{(k)} = \left\{Q^{(k)}: dQ^{(k)}(\theta^{(k)}) = dQ_\sigma(\sigma)dQ_w^{(k)}(w)\prod_{j=1}^kdQ_{\mu_j}(\mu_j)\right\}.
\end{equation}

Note that we do not factorize $dQ_w^{(k)}(w)$ because of the constraint $\sum_{j=1}^kw_j=1$. The variational posterior distribution is defined as $\wh{Q}=\wh{Q}^{(\wh k)}$ that solves (\ref{eq:ELBO-select-object}). The loss function here is chosen as $n$ times squared Hellinger distance, i.e., $L(P_f^n, P_{f_0}^n) = nH^2(P_f, P_{f_0})$.

In order that $\wh{Q}$ enjoys a good convergence rate, we need conditions on the prior distribution and the true density function $f_0$. We first list the conditions on the prior.
\begin{enumerate}
\item There exist constants $C_1, C_2>0$, such that
\begin{equation}\label{eq:mixture-condition1}
\sum_{m = k}^\infty\pi(m)\leq C_1\exp(-C_2k\log k),
\end{equation}
for all $m>0$. There exist constants $t,C_3, C_4>0$, such that
\begin{equation}\label{eq:mixture-condition4}
\pi(k_0)\geq C_3\exp(-C_4k_0\log k_0),
\end{equation}
for all $n^{\frac{1}{2\alpha+1}}\leq k_0\leq n^{\frac{1}{2\alpha+1}+t}$.
\item There exist constants $c_1, c_2, c_3>0$, such that
\begin{equation}\label{eq:mixture-condition2}
\int_{-\infty}^{-x_0}p_\mu(x)dx +\int_{x_0}^\infty p_\mu(x)dx\leq c_1\exp(-c_2x_0^{c_3}),
\end{equation}
for all $x_0>0$ and constants $c_4, c_5, c_6$, such that
\begin{equation}\label{eq:mixture-condition5}
p_\mu(x)\geq c_4\exp(-c_5|x|^{c_6}),
\end{equation}
for all $x$.
\item There exist constants $t,d_1, d_2, d_3>0$, such that
\begin{equation}\label{eq:mixture-condition6}
\int_{w\in\Delta_{k_0}(w_0, \epsilon)}p_w^{(k_0)}(x)dx\geq d_1\exp\left(-d_2k_0(\log k_0)^{d_3}\log\left(\frac{1}{\epsilon}\right)\right),
\end{equation}
for all $w_0\in\Delta_{k_0}$ and $n^{\frac{1}{2\alpha+1}}\leq k_0\leq n^{\frac{1}{2\alpha+1}+t}$, where $\Delta_{k_0}(w_0,\epsilon)=\{w\in\Delta_{k_0}: \|w-w_0\|_1\leq\epsilon\}$.
\item There exist constants $b_0,b_1,b_2,b_3>0$, such that
\begin{equation}\label{eq:mixture-condition3}
\|p_\tau\|_\infty<b_0,\qquad \int_{\tau_0}^\infty p_\tau(x)dx\leq b_1\exp(-b_2|\tau_0|^{b_3}),
\end{equation}
for all $\tau_0>0$. There exist constants $b_4, b_5>0$ and a constant $b_6\in (0, 1]$ that satisfy
\begin{equation}\label{eq:mixture-condition7}
p_\tau(x)\geq b_4\exp(-b_5|x|^{b_6}),
\end{equation}
for all $x>0$.
\end{enumerate}

The conditions on the prior distribution are quite general. For example, one can choose $k\sim\text{Poisson}(\xi_0)$, $\mu_j\sim N(0,\sigma_0^2)$, $w\sim\rm{Dir}(\alpha_0,\alpha_0,\cdots,\alpha_0)$ and $\tau\sim\Gamma(a_0, b_0)$ for some positive constants $\xi_0,\sigma_0, \alpha_0, a_0, b_0$. Then, the conditions above are all satisfied.

Next, we list the conditions on the true density function $f_0$:
\begin{enumerate}[label=B\arabic*]
\item \label{eq:B1}(Smoothness) The logarithmic density function $\log f_0$ is assumed to be locally $\alpha$-H\"older smooth. In other words, for the derivative $l_j(x) = \frac{d^j}{dx^j}\log f_0(x)$, there exists a polynomial $L(\cdot)$ and a constant $\gamma>0$ such that,
\begin{equation}\label{eq:true-condition1}
|l_{\floor{\alpha}}(x)-l_{\floor{\alpha}}(y)|\leq L(x)|x-y|^{\alpha-\floor{\alpha}},
\end{equation}
for all $x,y$ that satisfies $|x-y|\leq\gamma$. Here, the degree and the coefficients of the polynomial $L(\cdot)$ are all assumed to be constants.
Moreover, the derivative $l_j(x)$ satisfies the bound $\int |l_j(x)|^{\frac{2\alpha+\epsilon}{j}}f_0(x)dx<s_{\max}$ for all $j=1,...,\floor{\alpha}$ with  some constants $\epsilon, s_{\max}>0$.
\item \label{eq:B2}(Tail) There exist positive constants $T$, $\xi_1$, $\xi_2$, $\xi_3$ such that
\begin{equation}\label{eq:true-condition3}
f_0(x)\leq \xi_1e^{-\xi_2|x|^{\xi_3}},
\end{equation}
for all $|x|\geq T$.
\item \label{eq:B3}(Monotonicity) There exist constants $x_m<x_M$ such that $f_0$ is nondecreasing on $(-\infty, x_m)$ and is nonincreasing on $(x_M,\infty)$. Without loss of generality, we assume $f_0(x_m) = f_0(x_M) = c$ and $f_0(x)\geq c$ for all $x_m<x<x_M$ with some constant $c>0$.
\end{enumerate}
These conditions are exactly the same as in \cite{kruijer2010adaptive} and similar conditions are also considered in \cite{maugis2013adaptive}. The conditions allow a well-behaved approximation to the true density by a location-scale mixture.
 There are many density functions that satisfy the conditions (\ref{eq:B1})-(\ref{eq:B3}), for which we refer to \cite{kruijer2010adaptive}.

The convergence rate of the variational posterior is given by the following theorem.
\begin{thm}\label{thm:mixture}
Consider i.i.d. observations generated by $P_{f_0}^n$, and the density function $f_0$ satisfies 
conditions (\ref{eq:B1})-(\ref{eq:B3}). For the prior that satisfies
 (\ref{eq:mixture-condition1})-(\ref{eq:mixture-condition7}), we have
\[P_{f_0}^n\wh{Q} H^2(P_{\wh{k},\theta^{(\wh{k})}}, P_{f_0})\lesssim n^{-\frac{2\alpha}{2\alpha+1}}(\log n)^{\frac{2\alpha r}{2\alpha+1}},\]
where $\wh{Q}=\wh Q^{(\wh k)}$ is the solution of (\ref{eq:ELBO-select-object}), and $r = \frac{p}{\min\{p,\xi_3\}}+\max\{d_3+1, \frac{c_6}{\min\{p,\xi_3\}}\}$, with $p,\xi_3, c_6, d_3$  defined in (\ref{eq:kernel}), (\ref{eq:true-condition3}), (\ref{eq:mixture-condition5}) and (\ref{eq:mixture-condition6}), respectively.
\end{thm}

The proof of Theorem \ref{thm:mixture} largely follows the arguments in  \cite{kruijer2010adaptive} that are used to establish the corresponding result  for the true posterior distribution, thanks to the fact that Theorem \ref{thm:general-MF} requires three very similar ``prior mass and testing" conditions to that of \cite{ghosal2000convergence}. The only difference is that function approximations via location-scale mixtures need to be analyzed under a stronger divergence $D_{\rho}(\cdot\|\cdot)$ for some $\rho>1$. For this reason, the proof of Theorem \ref{thm:mixture} relies on the construction of a surrogate density function $\wt{f}_0$. We first apply Theorem \ref{thm:general-MF} and establish a convergence rate under $\wt{f}_0$. Then, the conclusion is transferred to $f_0$ with a change-of-measure argument. Details of the proof are given in Appendix \ref{sec:pf-mixture}.

\subsection{Dealing with Latent Variables}\label{sec:latent}

For the mixture model considered in Section \ref{sec:mixture-gmf}, we discuss a variation of the variational Bayes approach (\ref{eq:ELBO-select-object}) by including latent variables. This facilitates computation and leads to a simple coordinate ascent algorithm that has closed-form updates. In the setting of mixture model, our approach is adaptive to the unknown number of components, and can be regarded as an extension of \cite{yang2017alpha,pati2017statistical} for variational inference with latent variables.

Since $p(X^{(n)}|k,\theta^{(k)})=\prod_{i=1}^n\sum_{j=1}^kw_j\psi_{\sigma}(X_i-\mu_j)$ with $\theta^{(k)}=(\mu,w,\sigma)$, we can write
$$p(X^{(n)}|\theta^{(k)})=\sum_{z^{(k)}\in[k]^n}p(X^{(n)}|z^{(k)},\theta^{(k)})w^{(k)}(z^{(k)}),$$
where $p(X^{(n)}|z^{(k)},\theta^{(k)})=\prod_{i=1}^n\prod_{j=1}^k\psi_{\sigma}(X_i-\mu_j)^{\mathbf{1}{\left\{z_i^{(k)}=j\right\}}}$, and the probability of $z_i^{(k)}=j$ is $w_j$ under $w^{(k)}(\cdot)$.
We use the notation $\bar{\Pi}^{(k)}$ for the joint distribution of $(z^{(k)},\theta^{(k)})$, and then the marginal likelihood (\ref{eq:marginal-k}) can be written as
$$p(X^{(n)}|k)=\int p(X^{(n)}|z^{(k)},\theta^{(k)})d\bar{\Pi}^{(k)}(z^{(k)},\theta^{(k)}).$$
Similar to (\ref{eq:ELBO-selection}), the evidence lower bound with the latent variables is given by
\begin{eqnarray}
\nonumber && \log\left(p(X^{(n)}|k)\pi(k)\right) \\
\label{eq:ELBO-latent} &\geq& \int\log p(X^{(n)}|z^{(k)},\theta^{(k)})d\bar{Q}^{(k)}(z^{(k)},\theta^{(k)}) - D(\bar{Q}^{(k)}\|\bar{\Pi}^{(k)}) + \log \pi(k).
\end{eqnarray}
The right hand side of (\ref{eq:ELBO-latent}) is shorthanded by $\bar{F}(\bar{Q}^{(k)},k)$. Define
$$\bar{\mathcal{S}}_{\rm MF}^{(k)}=\left\{\bar Q^{(k)}:d\bar Q^{(k)}(z^{(k)},\theta^{(k)})=\prod_{i=1}^ndQ^{(k)}_z(z_i)dQ_{\sigma}(\sigma)dQ_w^{(k)}(w)\prod_{j=1}^kdQ_{\mu_j}(\mu_j)\right\}.$$
Then, we solve the following optimization problem,
\begin{equation}
\max_k\max_{\bar{Q}^{(k)}\in \bar{\mathcal{S}}_{\rm MF}^{(k)}}\bar{F}(\bar{Q}^{(k)},k). \label{eq:objective-mixture}
\end{equation}
The solution to (\ref{eq:objective-mixture}) leads to the variational posterior distribution $\wh{Q}=\wh{Q}^{(\wh{k})}_{\rm latent}$. It is worth noting that even though $\wh{Q}$ is a joint distribution of $(z,\mu,w,\sigma)$, the posterior inference only relies on the marginal of $(\mu,w,\sigma)$, since the parametrization of the density $f(\cdot)$ in (\ref{eq:mixture}) does not depend on the latent variables. The existence of the latent variables only facilitates computation.
\begin{thm}\label{thm:mixture-latent}
Consider i.i.d. observations generated by $P_{f_0}^n$, and the density function $f_0$ satisfies 
conditions (\ref{eq:B1})-(\ref{eq:B3}). For the prior that satisfies
 (\ref{eq:mixture-condition1})-(\ref{eq:mixture-condition7}), we have
\[P_{f_0}^n\wh{Q} H^2(P_{\wh{k},\theta^{(\wh{k})}}, P_{f_0})\lesssim n^{-\frac{2\alpha}{2\alpha+1}}(\log n)^{\frac{2\alpha r}{2\alpha+1}},\]
where $\wh Q=\wh{Q}^{(\wh{k})}_{\rm latent}$ is the solution to (\ref{eq:objective-mixture}), and $r = \frac{p}{\min\{p,\xi_3\}}+\max\{d_3+1, \frac{c_6}{\min\{p,\xi_3\}}\}$, with $p,\xi_3, c_6, d_3$  defined in (\ref{eq:kernel}), (\ref{eq:true-condition3}), (\ref{eq:mixture-condition5}) and (\ref{eq:mixture-condition6}), respectively.
\end{thm}

Theorem \ref{thm:mixture-latent} shows that the variational posterior with latent variables achieves the same contraction rate as in Theorem \ref{thm:mixture}. In fact, the two variational lower bounds (\ref{eq:ELBO-selection}) and (\ref{eq:ELBO-latent}) satisfy the following relation,
$$\log\left(p(X^{(n)}|k)\pi(k)\right)\geq \max_{Q^{(k)}\in \mathcal{S}_{\rm MF}^{(k)}}F(Q^{(k)},k) \geq \max_{\bar{Q}^{(k)}\in \bar{\mathcal{S}}_{\rm MF}^{(k)}}\bar{F}(\bar{Q}^{(k)},k),$$
which implies that the introduction of latent variables makes the variational approximation looser. On the other hand, Theorem \ref{thm:mixture-latent} shows that the worse variational approximation does not compromise the statistical convergence rate. Moreover, with the help of latent variables, $\wh{Q}^{(\wh{k})}_{\rm latent}$ can be computed via standard variational inference algorithms. Details of the computational issues are given in Appendix \ref{sec:CAVI-mixture}.

\section{Discussion}\label{sec:discussion}

\subsection{Variational Bayes and Empirical Bayes}\label{sec:EB}

In this section, we discuss an intriguing relation between variational Bayes and empirical Bayes in the context of sieve priors. We consider a nonparametric model $P_{\theta}^{(n)}$ with an infinite dimensional parameter $\theta=(\theta_j)\in\otimes_{j=1}^{\infty}\Theta_j\subset\mathbb{R}^{\infty}$. This includes the Gaussian sequence model and the infinite dimensional exponential family discussed in Section \ref{sec:application}, as well as nonparametric regression and spectral density estimation. For each dimension, we assume $\Theta_j=\Theta_{j1}\cup \Theta_{j2}$ and $\Theta_{j1}\cap \Theta_{j2}=\emptyset$. Then, a sieve prior $\theta\sim\Pi$ is specified by the following sampling process.
\begin{enumerate}
\item Sample $k\sim \pi$;
\item Conditioning on $k$, sample $\theta_j\sim f_{j1}$ for all $j\in[k]$, and sample $\theta_j\sim f_{j2}$ for all $j>k$.
\end{enumerate}
We assume that the densities $f_{j1}$ and $f_{j2}$ satisfy $\int_{\Theta_{j1}}f_{j1}=1$ and $\int_{\Theta_{j2}}f_{j2}=1$. A leading example of the sieve prior is case of $\Theta_{j1}=\mathbb{R}\backslash\{0\}$ and $\Theta_{j2}=\{0\}$, as is used in Section \ref{sec:gsm} and Section \ref{sec:inf-exp}.

An empirical Bayes procedure maximizes $e^{m_k(X^{(n)})}\pi(k)$\footnote{The canonical form of empirical Bayes has a flat prior on $k$.}, where
$$m_k(X^{(n)})=\log\int p(X^{(n)}|\theta)\prod_{j\leq k}f_{j1}(\theta_j)\prod_{j>k}f_{j2}(\theta_j)d\theta$$
is the logarithm of marginal likelihood. With the maximizer $\wh{k}$, the empirical Bayes posterior is defined as
\begin{equation}
d\wh{Q}_{\rm EB}(\theta) \propto p(X^{(n)}|\theta)\prod_{j\leq \wh{k}}f_{j1}(\theta_j)\prod_{j>\wh{k}}f_{j2}(\theta_j)d\theta.\label{eq:EB-post-def}
\end{equation}
Compared with a hierarchical Bayes approach, the empirical Bayes procedure does not need to evaluate the posterior distribution of $k$, and thus in many cases is easier to implement.

We also study mean-field approximation of the posterior distribution. In order to characterize its form, we need a few definitions. 
For any $g=(g_j)_{j=1}^{\infty}$, define
$$m_k(X^{(n)};g)= \int \prod_{j=1}^{\infty}g_j(\theta_j)\log p(X^{(n)}|\theta)d\theta - \sum_{j\leq k}D(g_j\|f_{j1}) - \sum_{j>k}D(g_j\|f_{j2}).$$
By Jensen's inequality, we observe that
\begin{equation}
m_k(X^{(n)}) \geq m_k(X^{(n)},g), \label{eq:EB-VB-ELBO}
\end{equation}
for any $g$. We also define the density classes $\mathcal{G}_{j1}=\left\{g\geq 0:\int g=\int_{\Theta_{j1}}g=1\right\}$ and $\mathcal{G}_{j2}=\left\{g\geq 0:\int g=\int_{\Theta_{j2}}g=1\right\}$.
The next theorem gives the exact form of the mean-field variational posterior.
\begin{thm}\label{thm:VB-sieve-general}
Consider the variational posterior $\wh{Q}_{\rm VB}$ induced by the sieve prior and  the mean-field variational set $\mathcal{S}_{\rm MF}$. The distribution $\wh{Q}_{\rm VB}$ is a product measure with the density of each coordinate specified by
$$q_j=\begin{cases}
\wt{g}_{j1}^{(\wt{k})}, & j<\wt{k}, \\
(1-\wt{p})\wt{g}_{j1}^{(\wt{k})} + \wt{p}\wt{g}_{j2}^{(\wt{k})}, & j=\wt{k}, \\
\wt{g}_{j2}^{(\wt{k})}, & j>\wt{k},
\end{cases}$$
where for each given $k$, $(\wt{g}_{j1}^{(k)})_{j=1}^{k}$ and $(\wt{g}_{j2}^{(k)})_{j=k}^{\infty}$ maximize the following objective function,
\begin{equation}
\pi(k-1)e^{m_{k-1}\left(X^{(n)},(g_{j1})_{j=1}^{k-1}\cup (g_{j2})_{j=k}^{\infty}\right)}+\pi(k)e^{m_k\left(X^{(n)},(g_{j1})_{j=1}^{k}\cup (g_{j2})_{j=k+1}^{\infty}\right)},\label{eq:general-VB-objective-k}
\end{equation}
under the constraints that $g_{j1}\in\mathcal{G}_{j1}$ and $g_{j2}\in\mathcal{G}_{j2}$ for all $j$,
$\wt{k}$ maximizes
\begin{equation}
\pi(k-1)e^{m_{k-1}\left(X^{(n)},(\wt g_{j1}^{(k)})_{j=1}^{k-1}\cup (\wt g_{j2}^{(k)})_{j=k}^{\infty}\right)}+\pi(k)e^{m_k\left(X^{(n)},(\wt g_{j1}^{(k)})_{j=1}^{k}\cup (\wt g_{j2}^{(k)})_{j=k+1}^{\infty}\right)},\label{eq:VB-k-obj}
\end{equation}
and finally,
$$\wt{p}=\frac{\pi(\wt k-1)e^{m_{\wt k-1}\left(X^{(n)},(\wt g_{j1}^{(\wt k)})_{j=1}^{\wt k-1}\cup (\wt g_{j2}^{(\wt k)})_{j=\wt k}^{\infty}\right)}}{\pi(\wt k-1)e^{m_{\wt k-1}\left(X^{(n)},(\wt g_{j1}^{(\wt k)})_{j=1}^{\wt k-1}\cup (\wt g_{j2}^{(\wt k)})_{j=\wt k}^{\infty}\right)}+\pi(\wt k)e^{m_{\wt k}\left(X^{(n)},(\wt g_{j1}^{(\wt k)})_{j=1}^{\wt k}\cup (\wt g_{j2}^{(\wt k)})_{j=\wt k+1}^{\infty}\right)}}.$$
\end{thm}

The result of Theorem \ref{thm:VB-sieve-general} also applies to the class $\mathcal{S}_{\rm G}$ discussed in Section \ref{sec:inf-exp} with $\mathcal{G}_{j1}$ replaced by the Gaussian class.
We note that Theorem \ref{thm:VB-sieve-general} can be viewed as an extension of Theorem \ref{thm:VB of gauss}. In fact, if the likelihood function can be factorized over each coordinate of $\theta$, the form of $\wh{Q}_{\rm VB}$ can be greatly simplified.
\begin{corollary}\label{cor:VB-sieve-factorize}
Under the same setting of Theorem \ref{thm:VB-sieve-general}, if we further assume that $p(X^{(n)}|\theta)=\prod_{j=1}^{\infty}p(X_j^{(n)}|\theta_j)$, then we will have
$$\wt{g}_{j1}^{(\wt{k})}(\theta_j)\propto f_{j1}(\theta_j)p(X_j^{(n)}|\theta_j)\mathbf{1}_{\{\theta_j\in\Theta_{j1}\}},$$
$$\wt{g}_{j2}^{(\wt{k})}(\theta_j)\propto f_{j2}(\theta_j)p(X_j^{(n)}|\theta_j)\mathbf{1}_{\{\theta_j\in\Theta_{j2}\}},$$
\begin{equation}
\wt{k}=\argmax_k\left(\pi(k-1|X^{(n)})+\pi(k|X^{(n)})\right), \label{eq:VB-k-obj-ind}
\end{equation}
and
$$\wt{p}=\frac{\pi(k-1|X^{(n)})}{\pi(k-1|X^{(n)})+\pi(k|X^{(n)})},$$
where
\[\pi(k|X^{(n)})\propto\pi(k)\prod_{j=1}^k\int_{\Theta_{j1}}f_{j1}(\theta_j)p(X^{(n)})d\theta_j\prod_{j = k+1}^{\infty}\int_{\Theta_{j2}}f_{j2}(\theta_j)p(X^{(n)}|\theta_j)d\theta_j.\]
\end{corollary}
In light of Theorem \ref{thm:VB-sieve-general}, we can compare the variational Bayes approach and the empirical Bayes approach, especially the definitions of $\wt{k}$ and $\wh{k}$. The empirical Bayes chooses the best model by maximizing $e^{m_k(X^{(n)})}\pi(k)$, or equivalently $\pi(k|X^{(n)})$, while the variational Bayes maximizes (\ref{eq:VB-k-obj}). There are two major differences. The first difference is that empirical Bayes uses the exact marginal likelihood function $m_k(X^{(n)})$ and variational Bayes uses a mean-field approximation of $m_k(X^{(n)})$. We remark that in the case of likelihood that can be factorized, the mean-field approximation is exact, which leads to (\ref{eq:VB-k-obj-ind}). The second difference is that empirical Bayes maximizes the posterior probability of the $k$th model, but the variational Bayes maximizes the sum of the posterior probabilities (or their mean-field approximations) of the $(k-1)$th and the $k$th models.

Despite the two differences, the empirical Bayes approach and the variational Bayes approach have a lot in common. Both are random probability distributions that summarize the information in data and prior. Both select a sub-model according to very similar criteria.
To close this section, we show that with a special variational class, the induced variational posterior is exactly the empirical Bayes posterior.
\begin{thm}\label{thm:EB as VB}
Define the following set
\begin{eqnarray*}
\mathcal{S}_{\rm EB} &=& \Bigg\{Q: Q\left(\left(\otimes_{j\leq k}\Theta_{j1}\right)\bigotimes\left(\otimes_{j>k}\Theta_{j2}\right)\right) = 1\text{ for some integer }k
\Bigg\}.
\end{eqnarray*}
Then, the empirical Bayes posterior $\wh{Q}_{\rm EB}$ defined by (\ref{eq:EB-post-def}) is the variational posterior induced by the sieve prior and the variational class $\mathcal{S}_{\rm EB}$.
\end{thm}
The result of Theorem \ref{thm:EB as VB} shows that for sieve priors, one can view the empirical Bayes approach as a variational Bayes approach, which suggests that it may be possible to unify the theoretical analysis in this paper and the analysis of empirical Bayes procedures in \cite{rousseau2017asymptotic}.

\subsection{Variational Approximation as Regularization}\label{sec:better}

According to Theorem \ref{thm:convergence}, the convergence rate of the posterior is determined by the sum of $\epsilon_n^2$, the rate of the true posterior, and $\gamma_n^2$, the variational approximation error. Since $\epsilon_n^2+\gamma_n^2\geq \epsilon_n^2$, it seems that the convergence rate of variational posterior is always no faster than that of the true posterior. However, Theorem \ref{thm:convergence} just gives an upper bound. In this section, we give two examples, and we show that it is possible for a variational posterior to have a faster convergence rate than that of the true posterior.

\paragraph{Example 1}

We consider the setting of Gaussian sequence model (\ref{eq:gauss}). The true signal $\theta^*$ that generates the data is assumed to belong to the Sobolev ball $\Theta_{\alpha}(B)$. The prior distribution is specified as
$$\theta\sim d\Pi=\prod_{j\leq n} dN(0,j^{-2\beta-1})\prod_{j>n}\delta_0.$$
Note that a similar Gaussian process prior is well studied in the literature \citep{vanzanten2008rates,castillo2008lower}. We force all the coordinates after $n$ to be zero, so that the variational approximation through Kullback-Leibler divergence will not explode. For the specified prior, the posterior contraction rate is $n^{-\frac{2(\alpha\wedge\beta)}{2\beta+1}}$, and when $\beta=\alpha$, the optimal minimax rate $n^{-\frac{2\alpha}{2\alpha+1}}$ is achieved.

Consider the following variational class
$$\mathcal{S}_{[k]}=\left\{Q: dQ=\prod_{j\leq k}dQ_j\prod_{j=k+1}^ndN(0,e^{-jn})\prod_{j>n}\delta_0\right\},$$
for a given integer $k$.
It is easy to see that the variational posterior $\wh{Q}_{[k]}$ defined by (\ref{eq:VB}) with $\mathcal{S}=\mathcal{S}_{[k]}$ can be written as
$$d\wh{Q}_{[k]}=\prod_{j\leq k}dN\left(\frac{n}{n+j^{2\beta+1}}Y_j, \frac{1}{n+j^{2\beta+1}}\right)\prod_{j=k+1}^ndN(0,e^{-jn})\prod_{j>n}\delta_0.$$
In other words, the class $\mathcal{S}_{[k]}$ does not put any constraint on the first $k$ coordinates and shrink all the coordinates after $k$ to zero. Ideally, one would like to use $\delta_0$ for the coordinates after $k$. However, that would lead to $D\left(Q\|\Pi(\cdot|Y)\right)=\infty$ for all $Q\in\mathcal{S}_{[k]}$ given that the support of $\delta_0$ is a singleton. That is why we use $N(0,e^{-jn})$ instead.
The rate of $\wh{Q}_{[k]}$ for each $k$ is given by the following theorem.
\begin{thm}\label{thm:UB for S_k}
For the variational posterior $\wh{Q}_{[k]}$, we have
$$\sup_{\theta^*\in\Theta_{\alpha}(B)}\mathbb{P}_{\theta^*}^{(n)}\wh{Q}_{[k]}\|\theta-\theta^*\|^2\asymp \begin{cases}
\frac{k}{n} + k^{-2\alpha}, & k\leq n^{\frac{1}{2\beta+1}}, \\
n^{-\frac{2(\alpha\wedge \beta)}{2\beta+1}}, & k>n^{\frac{1}{2\beta+1}},
\end{cases}$$
where $\wh Q_{[k]}$ is the variational posterior defined by (\ref{eq:VB}) with $\mathcal{S}=\mathcal{S}_{[k]}$.
\end{thm}
Note that Theorem \ref{thm:UB for S_k} gives both upper and lower bounds for $\wh{Q}_{[k]}$. This makes the comparison between variational posterior and true posterior possible. Observe that when $k=\infty$, we have $\wh{Q}_{[\infty]}=\Pi(\cdot|Y)$, and the result is reduced to the posterior contraction rate $n^{-\frac{2(\alpha\wedge\beta)}{2\beta+1}}$ in \cite{castillo2008lower}.

Depending on the values of $\alpha,\beta$ and $k$, the rate for $\wh{Q}_{[k]}$ can be better than that of the true posterior. For example, when $\beta<\alpha$, the choice $k=n^{\frac{1}{2\alpha+1}}$ leads to the minimax rate $n^{-\frac{2\alpha}{2\alpha+1}}$, which is always faster than $n^{-\frac{2(\alpha\wedge\beta)}{2\beta+1}}$. This is because for a $\beta<\alpha$, the true posterior distribution undersmooths the data, but the variational class $\mathcal{S}_{[k]}$ with $k=n^{\frac{1}{2\alpha+1}}$ helps to reduce the extra variance resulted from undersmoothing by thresholding all the coordinates after $k$. On the other hand, when $\beta\geq \alpha$, an improvement through the variational class $\mathcal{S}_{[k]}$ is not possible. In this case, the true posterior has already overly smoothed the data, and the information loss cannot be recovered by the variational class.

\paragraph{Example 2} Consider the problem of sparse linear regression $y\sim N(X\beta^*,I_n)$, where $X$ is a design matrix of size $n\times p$ and $\beta^*$ belongs to the sparse set $\mathcal{B}(s)=\{\beta\in\mathbb{R}^p:\sum_{j=1}^p\mathbf{1}_{\{\beta_j\neq 0\}}\leq s\}$ for some $s\in[p]$.
The prior distribution on $\beta$ is specified by the Laplace density
$$\frac{d\Pi(\beta)}{d\beta}=\prod_{j=1}^p\left(\frac{\lambda}{2}e^{-\lambda|\beta_j|}\right).$$
Though the posterior distribution has a close connection to LASSO, it is proved in \cite{castillo2015bayesian} that the posterior distribution cannot adapt to the sparsity of $\beta^*$. In particular, the common choice of $\lambda$ in the theoretical analysis of LASSO only leads to a dense posterior.

In fact, it is known in the literature (e.g. \cite{bickel2009simultaneous}) that the LASSO, which is the posterior mode, achieves a nearly optimal rate over the class $\mathcal{B}(s)$.
We show that the posterior mode can be well approximated by applying a simple variational class. Consider the variational class
$$\mathcal{S}_{\tau^2}=\left\{N(\beta,\tau^2I_p):\beta\in\mathbb{R}^p\right\}.$$
Define $\wh{Q}_{\tau^2}$ to be the minimizer of $\min_{Q\in\mathcal{S}_{\tau^2}}D(Q\|\Pi(\cdot|y))$.
\begin{thm}\label{thm:lasso}
For any $\lambda>0$ and $\tau>0$, we have $\wh{Q}_{\tau^2}=N(\wh{\beta},\tau^2I_p)$, where
\begin{equation}
\wh{\beta} = \argmin_{\beta}\left\{\frac{1}{2}\|y-X\beta\|^2 + \lambda \sum_{j=1}^p \tau h(\beta_j/\tau)\right\}.\label{eq:lasso}
\end{equation}
The function $h$ is defined by $h(x)=2\phi(x)+x\left(\Phi(x)-\Phi(-x)\right)$ with $\Phi(x)=\mathbb{P}(N(0,1)\leq x)$ and $\phi(x)=\Phi'(x)$.
\end{thm}
Theorem \ref{thm:lasso} shows that the variational approximation is characterized by the penalized least-squares estimator (\ref{eq:lasso}). Observe that $h$ is a convex function, and it satisfies $\sup_{x\in\mathbb{R}}\Big|\tau h(x/\tau) - |x|\Big|=\tau\sqrt{\frac{2}{\pi}}$ (see Figure \ref{fig:fz}), and thus $\wh{\beta}$ will get arbitrarily close to the LASSO estimator as $\tau\rightarrow 0$. Therefore, even though the posterior does not have a good frequentist property, its variational approximation can recover a sparse signal.
\begin{figure}[H]
\begin{center}
\includegraphics[width=0.8\textwidth]{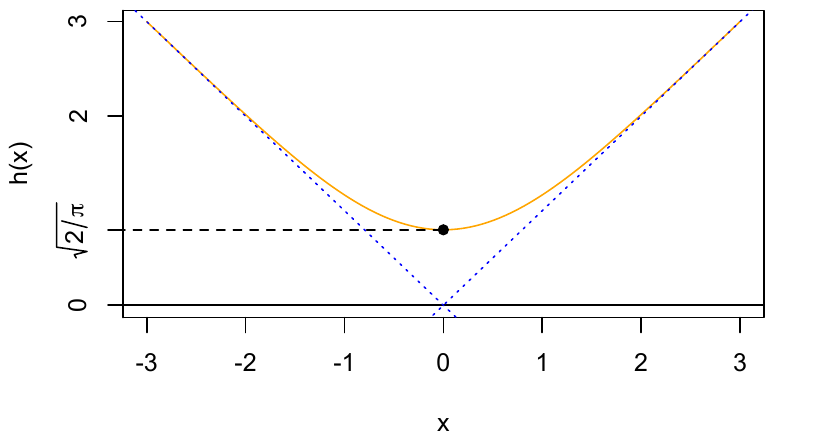}
\caption{The functions $h(x)$ (orange) and $|x|$ (blue).}
\label{fig:fz}
\end{center}
\end{figure}
By the fact that $\wh{Q}_{\tau^2}=N(\wh{\beta},\tau^2I_p)$, we have
\begin{equation}
\wh{Q}_{\tau^2}\|\beta-\beta^*\|^2 = \|\wh{\beta}-\beta^*\|^2 + p\tau^2.\label{eq:Pyth}
\end{equation}
Hence, a risk bound for the penalized least-squares estimator (\ref{eq:lasso}) directly leads to the convergence of the variational posterior. To present a bound for $\|\wh{\beta}-\beta^*\|^2$, we need to introduce some new notation. Let $S=\{j\in[p]:\beta_j^*\neq 0\}$ be the support of $\beta^*$. Define the restricted eigenvalue by
\begin{equation}
\kappa=\inf_{\{\Delta \neq 0: \|\Delta_{S^c}\|_1\leq 3\|\Delta_S\|_1\}}\frac{\frac{1}{\sqrt{n}}\|X\Delta\|}{\|\Delta\|}, \label{eq:RE}
\end{equation}
where $\|\Delta_S\|_1=\sum_{j\in S}|\Delta_j|$ and $\|\Delta_{S^c}\|_1$ is defined similarly. The same quantity (\ref{eq:RE}) also appears in the risk bound of LASSO \citep{bickel2009simultaneous}.
\begin{thm}\label{thm:VB-LASSO}
Assume $\|X_{*j}\|/\sqrt{n}\leq L$ for all $j\in[p]$ and $\kappa\leq L$ with some constant $L>0$.
Choose $\lambda=C\sqrt{n\log p}$ and $\tau=O\left(\frac{1}{np}\right)$ for some sufficiently large constant $C>0$. The solution to (\ref{eq:lasso}) satisfies
$$\|\wh{\beta}-\beta^*\|^2 \lesssim \frac{s\log p}{n\kappa^4},$$
with probability at least $1-p^{-C'}$ uniformly over $\|\beta^*\|_0\leq s$ for some constant $C'>0$.
As a consequence of (\ref{eq:Pyth}), we also have
$$\wh{Q}_{\tau^2}\|\beta-\beta^*\|^2\lesssim \frac{s\log p}{n\kappa^4},$$
with probability at least $1-p^{-C'}$.
\end{thm}
We note that $\frac{s\log p}{n\kappa^4}$ is the same rate of convergence of LASSO \citep{bickel2009simultaneous}. With $\tau$ chosen as small as $O\left(\frac{1}{np}\right)$, the statistical property of the variational posterior is very similar to that of the LASSO, and thus improves the original dense posterior distribution that is not suitable for sparse recovery.

\subsection{Model Misspecification}\label{sec:missp}

In this section, we present an extension of Theorem \ref{thm:convergence} in the context of model misspecification. We consider a data generating process $X^{(n)}\sim P_*^{(n)}$ that may not satisfies the conditions (\ref{eq:C1})-(\ref{eq:C3}). The following theorem shows that the convergence rate of the variational posterior will then have an extra term that characterizes the deviation of $P_*^{(n)}$ to the model specified by the likelihood.
\begin{thm}\label{thm:misspecification}
Suppose $\epsilon_n$ is a sequence that satisfies $n\epsilon_n^2\geq 1$. Assume that the conditions (\ref{eq:C1})-(\ref{eq:C3}) hold with $P_0^{(n)}$ replaced by $P_{\theta_0}^{(n)}$. 
Then for the variational posterior $\wh Q$ defined in (\ref{eq:VB}), we have
\begin{equation}
P_*^{(n)}\wh QL(P_\theta^{(n)}, P_{\theta_0}^{(n)})\leq M \left(n\left(\epsilon_n^2+\gamma_n^2\right) + D_2\big(P_*^{(n)}\|P_{\theta_0}^{(n)}\big)\right), \label{eq:misp-rate-3}
\end{equation}
for some constant $M$ only depending on $C_1,C$ and $\rho$ in (\ref{eq:C1})-(\ref{eq:C3}),
where the quantity $\gamma_n^2$ is defined as
\[\gamma_n^2 = 
\frac{1}{n}\inf_{Q\in\mathcal{S}}P_*^{(n)}D(Q||\Pi(\cdot|X^{(n)})).\]
\end{thm}

We note that here $\gamma_n^2$ is defined with respect to $P_*^{(n)}$ instead of $P_0^{(n)}$ in Theorem \ref{thm:convergence}. Theorem \ref{thm:convergence} can be viewed as a special case of Theorem \ref{thm:misspecification} with $P_0^{(n)}=P_*^{(n)}=P_{\theta_0}^{(n)}$. The extra term in the convergence rate that characterizes model misspecification is given by $D_2\big(P_*^{(n)}\|P_{\theta_0}^{(n)}\big)$. In fact, it can be replaced by any $\rho$-R\'enyi divergence with $\rho>1$.

Convergence rates of variational approximation to tempered posterior distributions under model misspecification have been studied by \cite{alquier2017concentration} (See their Theorem 2.7). Our results complement theirs by considering variational approximation to the ordinary posterior.

The next theorem gives sufficient conditions so that the variational approximation error $\gamma_n^2$ is dominated by the sum of the other two terms in (\ref{eq:misp-rate-3}). It can be viewed as an extension of Theorem \ref{thm:convergence2}.
\begin{thm}\label{thm:mis_MF}
Suppose there are constants $C_1,C_2>0$, such that
\begin{equation}\label{eq:C4**}
\inf_{Q\in\mathcal{S}\cap\mathcal{E}}D(Q\|\Pi)\leq C_1\left(n\epsilon_n^2+D_2\big(P_*^{(n)}\|P_{\theta_0}^{(n)}\big)\right),\tag{C4**}
\end{equation}
where $\mathcal{E}=\{Q:\supp(Q)\subset\mathcal{C}\}$ with
$$\mathcal{C}=\left\{\theta: D(P_*^{(n)}\|P_{\theta}^{(n)})\leq C_2\left(n\epsilon_n^2+D_2\big(P_*^{(n)}\|P_{\theta_0}^{(n)}\big)\right)\right\}.$$
Then, we have
$$n\gamma_n^2\leq (C_1+C_2)\left(n\epsilon_n^2+D_2\big(P_*^{(n)}\|P_{\theta_0}^{(n)}\big)\right).$$
\end{thm}

To end this section, we apply Theorem \ref{thm:misspecification} and Theorem \ref{thm:mis_MF} to the piecewise constant model discussed in Section \ref{sec:pcm} and derive oracle inequalities for the variational posterior distributions.
\begin{thm}\label{thm:PC-MC-misp}
Consider a prior distribution $\Pi$ that satisfies (\ref{eq:MC-condition1}) and (\ref{eq:MC-condition2}).
Then, for any $\theta^*\in \mathbb{R}^n$, we have
\[P_{\theta^*}^{(n)}\wh Q_{\rm MC}\|\theta-\theta^*\|^2\lesssim \min_{1\leq k\leq n}\left\{\inf_{\theta_0\in\Theta_k(B)}\|\theta^*-\theta_0\|^2+k\log n\right\},\]
\[P_{\theta^*}^{(n)}\wh Q_{\rm MC}^{\rm joint}\|\theta-\theta^*\|^2\lesssim \min_{1\leq k\leq n}\left\{\inf_{\theta_0\in\Theta_k(B)}\|\theta^*-\theta_0\|^2+k\log n\right\},\]
where the definitions of $\wh Q_{\rm MC}$ and $\wh{Q}_{\rm MC}^{\rm joint}$ are given in Theorem \ref{thm:PC-MC}.
\end{thm}

\section*{Acknowledgements}

The authors are grateful to an associate editor and two referees who give very insightful feedbacks that lead to the improvement of the paper.

\newpage

\appendix

\section{Additional Results}

\subsection{Sharp Convergence Rates for Gaussian Sequence Model}\label{sec:remove-log}

In this section, we consider a prior so that the logarithmic term in the convergence rate of Theorem \ref{thm:gauss1} can be removed. The sampling process of the prior is specified as follows.

\begin{enumerate}
\item Sample $k\sim\pi$;
\item Conditioning on $k$, sample $\sqrt{n}\theta_j\sim g_j$ for all $j\in[k]$, and set $\theta_j=0$ for all $j>k$.
\end{enumerate}
Obviously, this prior is the same as the previous one when $f_j(x)=\sqrt{n}g_j(\sqrt{n}x)$. However, the $\sqrt{n}$-scaling allows us to formulate conditions that help remove the logarithmic factor in Theorem \ref{thm:gauss1}. The same rescaling is also used in \cite{gao2016rate,gao2015general} to achieve sharp minimax rates. The following two conditions will be used to replace (\ref{eq:gauss1-condition2}) and (\ref{eq:gauss1-condition3}).
\begin{itemize}
\item There exist some constants $C_3,C_4>0$ such that for $k_0=\ceil{n^{\frac{1}{2\alpha+1}}}$,
\begin{equation}\label{eq:gauss2-condition2}
\pi(k_0)\geq C_3\exp(-C_4k_0).
\end{equation}
\item For the $k_0$ defined above, there exist constants $c_0\in\mathbb{R},c_1>0$ and $0<\beta<\frac{2}{2\alpha+1}$, such that
\begin{equation}\label{eq:gauss2-condition3}
-\log g_j(x)\leq c_0+c_1|x|^\beta,\qquad\text{for all $j\leq k_0$ and $x\in \mathbb{R}$}.
\end{equation}
\end{itemize}
The condition (\ref{eq:gauss2-condition2}) is similar to (\ref{eq:gauss1-condition2}), while (\ref{eq:gauss2-condition3}) is stronger compared with (\ref{eq:gauss1-condition3}). In general, one can choose $g_j$ to be a density with a heavy tail. As an example, one can easily check that $\pi(k)\propto e^{-\tau k}$ and $g_j(x)=\frac{1}{\pi\sigma\left(1+(x/\sigma)^2\right)}$ with constants $\tau,\sigma^2>0$ satisfy the two conditions.
Conditions (\ref{eq:C3}) and (\ref{eq:C4}) can be derived from (\ref{eq:gauss2-condition2}) and (\ref{eq:gauss2-condition3}) (see Lemma \ref{lem:gauss2-2}). This leads to the following result.

\begin{thm}\label{thm:gauss2}
Consider the prior $\Pi$ that satisfies (\ref{eq:gauss1-condition1}), (\ref{eq:gauss2-condition2}) and (\ref{eq:gauss2-condition3}).
Then, for any $\theta^*\in\Theta_{\alpha}(B)$, we have 
\[P_{\theta^*}^{(n)}\wh Q\|\theta-\theta^*\|^2\lesssim n^{-\frac{2\alpha}{2\alpha+1}},\]
where $\wh Q$ is the variational posterior defined by (\ref{eq:VB}) with $\mathcal{S}=\mathcal{S}_{\rm MF}$.
\end{thm}

\subsection{An Algorithm for (\ref{eq:objective-mixture})}\label{sec:CAVI-mixture}

In this section, we discuss how to optimize (\ref{eq:objective-mixture}). We first consider the problem $\max_{\bar{Q}^{(k)}\in \bar{\mathcal{S}}_{\rm MF}^{(k)}}\bar{F}(\bar{Q}^{(k)},k)$ for a fixed $k$. To solve this problem, a traditional method is to apply the coordinate ascent variational inference (CAVI). In order to obtain closed-form updates, we restrict ourselves to conjugate priors. In particular, we choose the kernel to be $\psi_{\sigma}(x)\propto e^{-\frac{x^2}{\sigma^2}}$, and priors to be 
\[p_{\mu_j}(\mu_j)\propto\exp\left(-\frac{\mu_j^2}{2\sigma_0^2}\right),\qquad p_{w}^{(k)}(w)\propto\prod_{j=1}^kw_j^{\alpha_0-1},\qquad p_\tau(\tau)\propto\tau^{a_0-1}\exp\left(-b_0\tau\right),\]
where $\tau = \sigma^{-2}$. By conjugacy, we can assume the variational posterior density for $(\mu, w,\tau, z)$ as $q(\mu, w, \tau, z) = \prod_{j=1}^kq_{\mu_j}(\mu_j)q_w(w)q_\tau(\tau)\prod_{i=1}^nq_{z_i}(z_i)$ with
\[q_{\mu_j}(\mu_j)\propto\exp\left(-\frac{(\mu_j-\wt\kappa_j)^2}{2\wt\sigma_j^2}\right),\qquad q_w(w)\propto\prod_{j=1}^kw_{j}^{\wt\alpha_j},\qquad q_\tau(\tau)\propto\tau^{\wt a-1}\exp(-\wt b\tau)\]
and
\[q_{z_i}(z_i) = \prod_{j=1}^kq_{ij}^{\1\{z_i = j\}},\]
where $\sum_{j=1}^kq_{ij} = 1$. Then, we only need to iteratively update the parameters as below.

\begin{enumerate}
\item Update $\wt\kappa_j, \wt\sigma_j^2$ by
$$\wt{\kappa}_j=\frac{\sum_{i=1}^nq_{ij}X_i}{\sum_{i=1}q_{ij} + \frac{\wt{b}}{2\wt{a}\sigma_0^2}}\text{ and }\wt{\sigma}_j^2=\frac{1}{\frac{2\wt a}{\wt b}\sum_{i=1}q_{ij} + \frac{1}{\sigma_0^2}}.$$
\item Update $\wt{\alpha}_1,...,\wt{\alpha}_k$ by
$$\wt{\alpha}_j=\alpha_0+\sum_{i=1}^nq_{ij}.$$
\item Update $\wt{a},\wt{b}$ by
$$\wt{a}=a_0+\frac{n}{2}\text{ and }\wt{b}=b_0+\sum_{i=1}^n\sum_{j=1}^kq_{ij}\left[(X_i-\wt{\kappa}_j)^2+\wt{\sigma}_j^2\right].$$
\item Update $q_{ij}$ by
$$q_{ij}\propto\exp\left(-\frac{\wt{a}}{\wt{b}}\left[(X_i-\wt{\kappa}_j)^2+\wt{\sigma}_j^2\right]+\text{digamma}(\wt{\alpha}_j)\right).$$
\end{enumerate}

The above iterations will approximately solve $\max_{\bar{Q}^{(k)}\in \bar{\mathcal{S}}_{\rm MF}^{(k)}}\bar{F}(\bar{Q}^{(k)},k)$ for a fixed $k$. The solution is parametrized by $\wt{\kappa}_j^{(k)},(\wt{\sigma}_j^2)^{(k)},\wt{\alpha}_j^{(k)},\wt{a}^{(k)},\wt{b}^{(k)},q_{ij}^{(k)}$. To select the best $k$, we then need to evaluate the objective function (\ref{eq:objective-mixture}), which is equivalent to plugging the values of $\wt{\kappa}_j^{(k)},(\wt{\sigma}_j^2)^{(k)},\wt{\alpha}_j^{(k)},\wt{a}^{(k)},\wt{b}^{(k)},q_{ij}^{(k)}$ into the right hand side of (\ref{eq:ELBO-latent}). This leads to the objective function
\begin{equation}
\bar{F}\left(\left(\wt{\kappa}_j^{(k)},(\wt{\sigma}_j^2)^{(k)},\wt{\alpha}_j^{(k)},\wt{a}^{(k)},\wt{b}^{(k)},q_{ij}^{(k)}\right),k\right),\label{eq:lazy}
\end{equation}
which can be calculated with a closed form by conjugacy. Finally, we choose $\wh{k}$ that maximize (\ref{eq:lazy}).

Note for each fixed $k$, computing (\ref{eq:lazy}) is straightforward and efficient by CAVI. The bottleneck of the algorithm is that one needs to evaluate (\ref{eq:lazy}) for every $k$. However, in terms of achieving the same statistical convergence rate given by Theorem \ref{thm:mixture-latent}, this is not necessary. Even if the variational posterior selects the best $k$ from a much smaller set $\mathcal{K}=\{1,2,4,...,2^{\ceil{\log_2n}}\}$ according to (\ref{eq:objective-mixture}), the same rate in Theorem \ref{thm:mixture-latent} can still be achieved with a slight modification of the proof. Therefore, one only needs to compute (\ref{eq:lazy}) for all $k\in\{1,2,4,...,2^{\ceil{\log_2n}}\}$.

\subsection{Beyond the Kullback-Leibler Approximation}

Modern variational approximation methods are not limited to the approximation by Kullback-Leibler divergence. For example, \cite{li2016renyi} proposed a generalized variational inference method using R\'enyi divergence and derived a corresponding evidence lower bound.  Though alternative divergences may
be hard to optimize, they may give better approximations \citep{minka2005divergence,opper2005expectation}.

It is possible to generalize our results to variational approximation using other criterions. We first introduce a $D_*$-variational posterior.

\begin{definition}[$D_*$-variational posterior]
Let $\mathcal{S}$ be a family of distributions. The $D_*$-variational posterior is defined as
\begin{equation}\label{eq:general_variational}
\wh Q_* = \argmin_{Q\in\mathcal{S}}D_*(Q\|\Pi(\cdot|X^{(n)})).
\end{equation}
\end{definition}

Then we state a result that extends Theorem \ref{thm:convergence} to the $D_*$-variational posterior distribution.
\begin{thm}\label{thm:convergence-general}
Suppose $D_*$ is a divergence such that $D_*(P_1\|P_2)\geq 0$ for all probability measures $P_1$ and $P_2$. Assume $D_*(P_1\|P_2)\geq D(P_1\|P_2)$ for any $P_1\in\mathcal{S}$ and any $P_2$, and the conditions (\ref{eq:C1})-(\ref{eq:C3}) in Theorem \ref{thm:convergence} hold. Then for the $D_*$-variational posterior $\wh Q_*$ defined in (\ref{eq:general_variational}), we have
\begin{equation}\label{eq:mean convergence-general}
P_0^{(n)}\wh Q_*L(P_\theta^{(n)},P_0^{(n)})\leq M n(\epsilon_n^2+\gamma_{n}^{*2}),
\end{equation}
for some constant $M>0$,
where the quantity $\gamma_n^2$ is defined as
\[\gamma_n^{*2} = 
\frac{1}{n}\inf_{Q\in\mathcal{S}}P_0^{(n)}D_*(Q||\Pi(\cdot|X^{(n)})).\]
\end{thm}
Theorem \ref{thm:convergence-general} is a generalization of Theorem \ref{thm:convergence} for a divergence $D_*$ that is not smaller than the Kullback-Leibler divergence. Examples of applications include all R\'enyi divergence with $\rho\geq 1$ and the $\chi^2$-divergence. Divergence functions that are not necessarily larger than the Kullback-Leibler require new techniques to analyze, and will be considered as an interesting future project.

\section{Proofs}\label{sec:proof}

\subsection{Proof of Theorem \ref{thm:convergence}}

This section gives the proof of Theorem \ref{thm:convergence}, which is divided into several lemmas. We first give an inequality that uses the basic property of the KL-divergence.

\begin{lemma}\label{lem:basic ineq}
For any function $f\geq 0$ and two probability measure $P$ and $Q$, we have
\[\int f(x)dQ(x)\leq D(Q\|P)+\log\int\exp(f(x))dP(x).\]
\end{lemma}
\begin{proof}
By the definition of KL-divergence, we have
\begin{eqnarray*}
&& D(Q\|P)+\log\int\exp(f(x))dP(x) \\
& = &\int \log\left(\frac{dQ(x)\int \exp(f(y))dP(y)}{dP(x)}\right)dQ(x)\\
& = &\int \log\left(\frac{dQ(x)\int \exp(f(y))dP(y)}{\exp(f(x))dP(x)}\right)dQ(x)+\int f(x)dQ(x)\\
& = &D(Q\|\wt P)+\int f(x)dQ(x)\\
&\geq& \int f(x)dQ(x),
\end{eqnarray*}
where $\wt P$ is a probability measure given by
\[d\wt P(x) = \frac{\exp(f(x))dP(x)}{\int\exp(f(y))dP(y)}.\]
\end{proof}
Then, we can use the inequality in Lemma \ref{lem:basic ineq} to derive a useful bound for $P_0^{(n)}\wh QL(P_\theta^{(n)}, P_0^{(n)})$.
\begin{lemma}\label{lem:general1}
For the $\wh Q$ defined in (\ref{eq:VB}), we have
\begin{eqnarray*}
&& P_0^{(n)}\wh QL(P_\theta^{(n)}, P_0^{(n)}) \\
&\leq& \inf_{a>0}\frac{1}{a}\left(\inf_{Q\in\mathcal{S}}P_0^{(n)}D( Q\|\Pi(\cdot|X^{(n)}))+\log P_0^{(n)}\Pi(\exp(aL(P_\theta^{(n)}, P_0^{(n)}))|X^{(n)})\right).
\end{eqnarray*}
\end{lemma}
\begin{proof}
By Lemma \ref{lem:basic ineq}, we have
$$a\wh QL(P_\theta^{(n)}, P_0^{(n)})\leq D( \wh{Q}\|\Pi(\cdot|X^{(n)}))+\log\Pi(\exp(aL(P_\theta^{(n)}, P_0^{(n)}))|X^{(n)}),$$
for all $a>0$. By the definition of $\wh Q$, we have
\[D(\wh Q\|\Pi(\cdot|X^{(n)}))\leq D(Q\|\Pi(\cdot|X^{(n)})),\]
for all $Q\in\mathcal{S}$. Taking expectation on both sides, we have
$$aP_0^{(n)}\wh QL(P_\theta^{(n)}, P_0^{(n)})\leq P_0^{(n)}D( Q\|\Pi(\cdot|X^{(n)}))+P_0^{(n)}\log\Pi(\exp(aL(P_\theta^{(n)}, P_0^{(n)}))|X^{(n)}).$$
Using Jensen's inequality, we get
$$P_0^{(n)}\log\Pi(\exp(aL(P_\theta^{(n)}, P_0^{(n)}))|X^{(n)})\leq \log P_0^{(n)}\Pi(\exp(aL(P_\theta^{(n)}, P_0^{(n)}))|X^{(n)}).$$
Therefore,
$$P_0^{(n)}\wh QL(P_\theta^{(n)}, P_0^{(n)})\leq \frac{1}{a}\left(P_0^{(n)}D( Q\|\Pi(\cdot|X^{(n)}))+\log P_0^{(n)}\Pi(\exp(aL(P_\theta^{(n)}, P_0^{(n)}))|X^{(n)})\right).$$
The proof is complete by taking minimum over $a>0$ and $Q\in\mathcal{S}$.
\end{proof}

In order to bound $P_0^{(n)}\Pi(\exp(aL(P_\theta^{(n)}, P_0^{(n)}))|X^{(n)})$, we need the following lemma on the posterior tail probability. Its proof is similar to the one used in \cite{ghosal2000convergence}.

\begin{lemma}\label{lem:strong convergence}
Under the conditions of Theorem \ref{thm:convergence}, we have
\[P_0^{(n)}\Pi\left(L(P_\theta^{(n)}, P_0^{(n)})>C_1n\epsilon^2|X^{(n)}\right)\leq\exp(-Cn\epsilon^2)+\exp(-\lambda n\epsilon^2)+2\exp(-n\epsilon^2),\]
for all $\epsilon\geq \epsilon_n$, where $\lambda = \rho-1$ for $\rho$ in (\ref{eq:C3}).
\end{lemma}
\begin{proof}
We first define the sets
\[U_n = \left\{\theta: L(P_\theta^{(n)}, P_0^{(n)})>C_1n\epsilon^2\right\},\qquad K_n = \left\{\theta:D_{1+\lambda}(P_0^{(n)}\|P_\theta^{(n)})\leq C_3n\epsilon_n^2\right\}.\]
We also define the event
\[A_n = \left\{X^{(n)}: \int\frac{dP_{\theta}^{(n)}}{dP_0^{(n)}}(X^{(n)})d\wt\Pi(\theta)\leq\exp(-(C_3+1)n\epsilon^2)\right\},\]
where the probability measure $\wt{\Pi}$ is defined as $\wt\Pi(B) = \frac{\Pi(B\cap K_n)}{\Pi(K_n)}$. Let $\Theta_n(\epsilon)$ and $\phi_n$ be the set and the testing function in (\ref{eq:C1}). Then, we bound $P_0^{(n)}\Pi(U_n|X^{(n)})$ by
\begin{eqnarray*}
&&P_0^{(n)}\Pi(U_n|X^{(n)})\\
&\leq&P_0^{(n)}\phi_n + P_0^{(n)}(A_n) +P_0^{(n)}(1-\phi_n)\Pi(U_n|X^{(n)})\mathbf{1}_{A_n^c}\\
& = &P_0^{(n)}\phi_n + P_0^{(n)}(A_n) +
P_0^{(n)}\frac{\int_{U_n}\frac{dP_{\theta}^{(n)}}{dP_0^{(n)}}(X^{(n)})d\Pi(\theta)}{\int \frac{dP_{\theta}^{(n)}}{dP_0^{(n)}}(X^{(n)})d\Pi(\theta)}(1-\phi_n)\mathbf{1}_{A_n^c}.
\end{eqnarray*}
We will give bounds for the three terms above respectively.
By (\ref{eq:C1}),
\begin{equation}\label{part1}
P_0^{(n)}\phi_n\leq \exp(-Cn\epsilon^2).
\end{equation}
Using the definitions of $A_n$, we have
\begin{eqnarray}\label{part2}
P_0^{(n)}(A_n)& = & P_0^{(n)}\left(\left(\int\frac{dP_{\theta}^{(n)}}{dP_0^{(n)}}(X^{(n)})d\wt\Pi(\theta)\right)^{-\lambda}>\exp(\lambda(C_3+1)n\epsilon^2)\right)\nonumber\\
&\leq&\exp(-\lambda (C_3+1)n\epsilon^2)P_0^{(n)}\left(\int\frac{dP_{\theta}^{(n)}}{dP_0^{(n)}}(X^{(n)})d\wt\Pi(\theta)\right)^{-\lambda }
\nonumber\\
&\leq&\exp(-\lambda (C_3+1)n\epsilon^2)\int\left(\int \frac{(dP_0^{(n)})^{1+\lambda}}{(dP_\theta^{(n)})^{\lambda }}\right)d\wt\Pi(\theta)\nonumber\\
& = &\exp(-\lambda (C_3+1)n\epsilon^2)\int\exp(\lambda  D_{1+\lambda}(P_0^{(n)}\|P_\theta^{(n)}))d\wt\Pi(\theta)\nonumber\\
& \leq & \exp(-\lambda (C_3+1)n\epsilon^2+\lambda  C_3n\epsilon_n^2)\nonumber\\
& \leq &\exp(-\lambda n\epsilon^2).
\end{eqnarray}
Now we analyze the third term.
On the event $A_n^c$, we have
\[\int\frac{dP_{\theta}^{(n)}}{dP_0^{(n)}}(X^{(n)})d\Pi(\theta)\geq\Pi(K_n)\int\frac{dP_{\theta}^{(n)}}{dP_0^{(n)}}(X^{(n)})d\wt\Pi(\theta)\geq\exp(-(C_2+C_3+1)n\epsilon^2),\]
where the last inequality is by (\ref{eq:C3}).
Then,
it follows that
\begin{eqnarray*}
&&P_0^{(n)}\frac{\int_{U_n}\frac{dP_{\theta}^{(n)}}{dP_0^{(n)}}(X^{(n)})d\Pi(\theta)}{\int \frac{dP_{\theta}^{(n)}}{dP_0^{(n)}}(X^{(n)})d\Pi(\theta)}(1-\phi_n)\mathbf{1}_{A_n^c}\\
&\leq&\exp((C_3+C_2+1)n\epsilon^2)P_0^{(n)}\int_{U_n}\frac{dP_{\theta}^{(n)}}{dP_0^{(n)}}(X^{(n)})(1-\phi_n)d\Pi(\theta)\\
& \leq &\exp((C_3+C_2+1)n\epsilon^2)\left[\int_{U_n\cap \Theta_n(\epsilon)}P_\theta^{(n)}(1-\phi_n)d\Pi(\theta)+\Pi(\Theta_n(\epsilon)^c)\right]\\
&\leq&\exp((C_3+C_2+1)n\epsilon^2)(\exp(-Cn\epsilon^2)+\exp(-Cn\epsilon^2)),
\end{eqnarray*}
where the last inequality is by (\ref{eq:C1}) and (\ref{eq:C2}).
Since $C> C_3+C_2+2$, we obtain the bound
\begin{equation}\label{part3}
P_0^{(n)}\frac{\int_{U_n}\frac{dP_{\theta}^{(n)}}{dP_0^{(n)}}(X_i)d\Pi(\theta)}{\int \frac{dP_{\theta}^{(n)}}{dP_0^{(n)}}(X_i)d\Pi(\theta)}(1-\phi_n)\mathbf{1}_{A_n^c}\leq 2\exp(-n\epsilon^2).
\end{equation}
Combining the bounds (\ref{part1}), (\ref{part2}) and (\ref{part3}), we have
\[P_0^{(n)}\Pi(U_n|X^{(n)})\leq\exp(-Cn\epsilon^2)+\exp(-\lambda n\epsilon^2)+2\exp(-n\epsilon^2).\]
 \end{proof}

Next, we derive a moment generating function bound for a sub-exponential random variable.
\begin{lemma}\label{lem:sub-exp}
Suppose the random variable $X$ satisfies
\[\mathbb{P}(X\geq t)\leq c_1\exp(-c_2t),\]
for all $t\geq t_0>0$.
Then, for any $0<a\leq \frac{1}{2}c_2$,
\[\mathbb{E}\exp(aX)\leq \exp(at_0)+c_1.\]
\end{lemma}
\begin{proof}
Set $Y = \exp(aX)$ for some $0<a\leq\frac{1}{2}c_2$. Then, for any $M_0>0$.
\begin{eqnarray*}
&&\mathbb{E} Y\leq M_0+\int_{M_0}^\infty \mathbb{P}(Y\geq y)dy \\
&=& M_0 +\int_{M_0}^\infty \mathbb{P}\left(X\geq\frac{1}{a}\log y\right)dy\leq M_0+c_1\int_{M_0}^{\infty}y^{-c_2/a}dy.
\end{eqnarray*}
Choose $M_0 = \exp(at_0)$, and then since $a\leq\frac{1}{2}c_2$, we have
\[\mathbb{E} Y\leq\exp(at_0)+\frac{c_1a}{c_2-a}\exp((a-c_2)t_0)\leq \exp(at_0)+c_1\exp(-at_0)\leq\exp(at_0)+c_1.\]
\end{proof}

Now we are ready to prove Theorem \ref{thm:convergence}.

\begin{proof}[Proof of Theorem \ref{thm:convergence}]
By Lemma \ref{lem:strong convergence}, we have
$$P_0^{(n)}\Pi\left(L(P_{\theta}^{(n)},P_0^{(n)})>t|X^{(n)}\right)\leq c_1\exp(-c_2t),$$
for all $t\geq t_0$. Here, $c_1=4$, $c_2=\min\left\{\lambda,1\right\}/C_1$ as $C>C_1+C_2+2>1$ and $t_0=C_1n\epsilon_n^2$. Then, by Lemma \ref{lem:sub-exp}, we have
$$P_0^{(n)}\Pi\left(\exp\left(aL(P_{\theta}^{(n)},P_0^{(n)})\right)|X^{(n)}\right)\leq \exp\left(aC_1n\epsilon_n^2\right)+4,$$
for all $a\leq \min\left\{\lambda,1\right\}/(2C_1)$. Taking $a=\min\left\{\lambda,1\right\}/(2C_1)$ and using Lemma \ref{lem:general1}, we get
\begin{eqnarray*}
P_0^{(n)}\wh{Q}L(P_{\theta}^{(n)},P_0^{(n)})&\leq &\frac{n\gamma_n^2+\log(4+e^{aC_1n\epsilon_n^2})}{a}\leq \frac{n\gamma_n^2}{a}+C_1n\epsilon_n^2+\frac{4e^{-aC_1n\epsilon_n^2}}{a}\\
&\leq& Mn(\gamma_n^2+\epsilon_n^2),
\end{eqnarray*}
with some $M>0$ that only depends on $C,C_1,\lambda$.
\end{proof}

\subsection{Proofs of Theorem \ref{thm:convergence2}, Theorem \ref{thm:mean field} and Theorem \ref{thm:general-MF}}

\begin{proof}[Proof of Theorem \ref{thm:convergence2}]
For any $Q\in\mathcal{S}\cap\mathcal{E}$, we have $\supp(Q)\subset\mathcal{C}$, and thus $QD(P_0^{(n)}\|P_{\theta}^{(n)})\leq C_2n\epsilon_n^2$. By (\ref{eq:C4*}), we have $D(Q\|\Pi)\leq C_1n\epsilon_n^2$. Therefore, $R(Q)\leq (C_1+C_2)n\epsilon_n^2$, and the proof is complete.
\end{proof}

\begin{proof}[Proof of Theorem \ref{thm:mean field}]
It is sufficient to find a $Q\in\mathcal{S}_{\rm MF}$ and bound
$$R(Q)=\frac{1}{n}\left(D(Q\|\Pi)+QD(P_0^{(n)}\|P_{\theta}^{(n)})\right).$$
We choose $Q$ to be the product measure $dQ(\theta)=\prod_{j=1}^mdQ_j(\theta_j)$, with
$$Q_j(B_j)=\frac{\wt{Q}_j(B_j\cap \wt{\Theta}_j)}{\wt{Q}_j(\wt{\Theta}_j)}.$$
Then, it is easy to see that $Q\in\mathcal{S}_{\rm MC}$ and $\supp(Q)\subset \otimes_{j=1}^m\wt{\Theta}_j$. By (\ref{eq:mfc1}), we have
$$QD(P_0^{(n)}\|P_{\theta}^{(n)})\leq C_1n\epsilon_n^2.$$
Moreover, we can write $D(Q\|\Pi)$ as below
\[D(Q\|\Pi) = Q\log\frac{dQ}{d\wt{Q}}+Q\log\frac{d\wt{Q}}{d\Pi},\]
where 
$$Q\log\frac{dQ}{d\wt{Q}}=-\sum_{j=1}^m\log\wt{Q}_j(\wt{\Theta}_j)\leq C_3n\epsilon_n^2,$$
by (\ref{eq:mfc2}), and
$$Q\log\frac{d\wt{Q}}{d\Pi}\leq C_2n\epsilon_n^2,$$
by (\ref{eq:mfc1}). Hence, we obtain the desired bound.
\end{proof}

To show Theorem \ref{thm:general-MF}, we need a model selection version of Lemma \ref{lem:general1}:

\begin{lemma}\label{lem:general1_model_selection}
For $\wh Q^{(\wh k)}$ defined as the solution of (\ref{eq:ELBO-select-object}),
\begin{eqnarray*} 
&&P_0^{(n)}\left[\wh Q^{(\wh k)}L\left(P_{\wh k,\theta^{(\wh k)}}^{(n)}, P_0^{(n)}\right)\right]\\
&\leq& \inf_{a>0}\frac{1}{a}\left[\min_{k\in\mathcal{K}}\min_{Q^{(k)}\in\mathcal{S}_{\rm MF}^{(k)}}\left\{D\left(Q^{(k)}\|\Pi^{(k)}\right)+Q^{(k)}D\left(P_0^{(n)}\|P_{k,\theta^{(k)}}^{(n)}\right)-\log\pi(k)\right\}\right.\\
&&\left.+P_0^{(n)}\log \Pi\left(\exp\left(aL(P_{k,\theta^{(k)}}^{(n)}, P_0^{(n)})\right)\Big|X^{(n)}\right)\right],
\end{eqnarray*}
where $\Pi$ is the prior distribution on $P_{k,\theta^{(k)}}$ induced by the sampling process
of $(k,\theta^{(k)})$.
\end{lemma}

\begin{proof}
We use $p_0^{(n)}$, $p_{k,\theta^{(k)}}^{(n)}$ to denote the densities of $P_0^{(n)}$, $P_{k,\theta^{(k)}}^{(n)}$. A lower bound can be directly derived from the right hand side minus the left hand side. For any $a>0$, any $k\in\mathcal{K}$, and any $Q^{(k)}\in\mathcal{S}_{\rm MF}^{(k)}$, we have
\begin{eqnarray*}
&&D\left(Q^{(k)}\|\Pi^{(k)}\right)+Q^{(k)}D\left(P_0^{(n)}\|P_{k,\theta^{(k)}}^{(n)}\right)-\log\pi(k)\\
&&-aP_0^{(n)}\left[\wh Q^{(\wh k)}L\left(P_{\wh k,\theta^{(\wh k)}}^{(n)}, P_0^{(n)}\right)\right]\\
&=& P_0^{(n)}\left(-F(Q^{(k)},k)+\log p_0^{(n)}(X^{(n)})\right)-aP_0^{(n)}\left[\wh Q^{(\wh k)}L\left(P_{\wh k,\theta^{(\wh k)}}^{(n)}, P_0^{(n)}\right)\right]\\
&\geq& P_0^{(n)}\left(-F(\wh Q^{(\wh k)},\wh k)+\log p_0^{(n)}(X^{(n)})\right)-aP_0^{(n)}\left[\wh Q^{(\wh k)}L\left(P_{\wh k,\theta^{(\wh k)}}^{(n)}, P_0^{(n)}\right)\right]\\
&=&P_0^{(n)}D\left(\wh Q^{(\wh k)}\|\Pi^{(\wh k)}\right)+P_0^{(n)}\wh Q^{(\wh k)}\log\frac{p_0^{(n)}(X^{(n)})}{p_{\wh k,\theta^{(\wh k)}}^{(n)}(X^{(n)})}-P_0^{(n)}\log\pi(\wh k)\\
&&-aP_0^{(n)}\left[\wh Q^{(\wh k)}L\left(P_{\wh k,\theta^{(\wh k)}}^{(n)}, P_0^{(n)}\right)\right]\\
&=&P_0^{(n)}\left[\wh Q^{(\wh k)}\log\frac{d\wh Q^{(\wh k)}(\theta^{(\wh k)})p_0^{(n)}(X^{(n)})}
{\pi(\wh k)d\Pi^{(\wh k)}(\theta^{(\wh k)})p_{\wh k,\theta^{(\wh k)}}^{(n)}(X^{(n)})\exp\left(aL\left(P_{\wh k,\theta^{(\wh k)}}^{(n)}, P_0^{(n)}\right)\right)}\right]\\
&=&D\left(P_0^{(n)}\|P_\Pi^{(n)}\right)\\
&&+P_0^{(n)}\left[\wh Q^{(\wh k)}\log\frac{d\wh Q^{(\wh k)}(\theta^{(\wh k)})p_\Pi^{(n)}(X^{(n)})}
{\pi(\wh k)d\Pi^{(\wh k)}(\theta^{(\wh k)})p_{\wh k,\theta^{(\wh k)}}^{(n)}(X^{(n)})\exp\left(aL\left(P_{\wh k,\theta^{(\wh k)}}^{(n)}, P_0^{(n)}\right)\right)}\right]\\
&=& D\left(P_0^{(n)}\|P_\Pi^{(n)}\right) + P_0^{(n)}D\left(\wh Q^{(\wh k)}\|\wt\Pi^{(\wh k)}\right)\\
&&-P_0^{(n)}\log\frac{\int\pi(\wh k)p_{\wh k,\theta^{(\wh k)}}^{(n)}(X^{(n)})\exp\left(aL\left(P_{\wh k,\theta^{(\wh k)}}^{(n)}, P_0^{(n)}\right)\right)d\Pi^{(\wh k)}(\theta^{(\wh k)})}{p_\Pi^{(n)}(X^{(n)})}\\
&\geq&-P_0^{(n)}\log\frac{\sum_{k\in\mathcal{K}}\int\pi( k)p_{ k,\theta^{(k)}}^{(n)}(X^{(n)})\exp\left(aL\left(P_{ k,\theta^{( k)}}^{(n)},P_0^{(n)}\right)\right)d\Pi^{(k)}(\theta^{(k)})}{p_\Pi^{(n)}(X^{(n)})}\\
&=&-P_0^{(n)}\log\Pi\left(\exp\left(aL\left(P_{k,\theta^{(k)}}^{(n)},P_0^{(n)}\right)\right)\Big|X^{(n)}\right),
\end{eqnarray*}
where $P_\Pi^{(n)}$ is the probability measure with the density $p_{\Pi}^{(n)}$ with
$$
p_{\Pi}^{(n)}(X^{(n)}) = \sum_{k\in\mathcal{K}}\pi(k)\int p_{k,\theta^{(k)}}^{(n)}(X^{(n)})d\Pi^{(k)}(\theta^{(k)})=\int p_{k,\theta^{(k)}}^{(n)}d\Pi\left(P_{k,\theta^{(k)}}^{(n)}\right),
$$
and
\[d\wt\Pi^{(k)}(\theta^{(k)}) = \frac{d\Pi^{(k)}(\theta^{(k)})p_{k,\theta^{(k)}}^{(n)}(X^{(n)})\exp\left(aL\left(P_{k,\theta^{(k)}}^{(n)}, P_0^{(n)}\right)\right)}{\int p_{k,\theta^{(k)}}^{(n)}(X^{(n)})\exp\left(aL\left(P_{k,\theta^{(k)}}^{(n)}, P_0^{(n)}\right)\right)d\Pi^{(k)}(\theta^{(k)})}.\]
The proof is complete.
\end{proof}

\begin{proof}[Proof of Theorem \ref{thm:general-MF}]
By Lemma \ref{lem:general1_model_selection}, we have
\begin{eqnarray*}
&&P_0^{(n)}\left[\wh Q^{(\wh k)}L\left(P_{\wh k,\theta^{(\wh k)}}^{(n)},P_0^{(n)}\right)\right]\\
&\leq& \inf_{a>0}\frac{1}{a}\left[\min_{k\in\mathcal{K}}\min_{Q^{(k)}\in\mathcal{S}_{\rm MF}^{(k)}}\left\{D\left(Q^{(k)}\|\Pi^{(k)}\right)+Q^{(k)}D\left(P_0^{(n)}\|P_{k,\theta^{(k)}}^{(n)}\right)-\log\pi(k)\right\}\right.\\
&&\left.+P_0^{(n)}\log \Pi\left(\exp\left(aL\left(P_{k,\theta^{(k)}}^{(n)}, P_0^{(n)}\right)\right)\Big|X^{(n)}\right)\right].
\end{eqnarray*}
Now we analyze each term on the right hand side. 
By Jensen's Inequality together with Lemma \ref{lem:strong convergence} and Lemma \ref{lem:sub-exp}, we have
\begin{eqnarray*}
&&P_0^{(n)}\log \Pi\left(\exp\left(aL\left(P_{k,\theta^{(k)}}^{(n)},P_0^{(n)}\right)\right)\Big|X^{(n)}\right)\\
&\leq&\log P_0^{(n)}\Pi\left(\exp\left(aL\left(P_{k,\theta^{(k)}}^{(n)},P_0^{(n)}\right)\right)\Big|X^{(n)}\right)\lesssim n\epsilon_n^2,
\end{eqnarray*}
with some small constant $a>0$.
This is because the conditions (\ref{eq:C1}) and (\ref{eq:C2}) with respect to prior $\Pi$ hold by assumption, and (\ref{eq:C3}) is implied by (\ref{eq:C3*}) with the argument
\begin{eqnarray*}
&&\Pi\left(\left\{P_{k,\theta^{(k)}}^{(n)}: D_{\rho}\Big(P_0^{(n)}\|P_{k,\theta^{(k)}}^{(n)}\Big)\leq C_3n\epsilon_n^2\right\}\right) \\
&\geq& \Pi\left(\left\{P_{k,\theta^{(k)}}: k=k_0, \theta^{(k_0)}\in{\Theta}^{(k_0)}\right\}\right)\\ 
&\geq& \pi(k_0)\Pi^{(k_0)}(\Theta^{(k_0)})\geq\exp\left(-C_2n\epsilon_n^2\right).
\end{eqnarray*}
For the remaining terms, we choose $k = k_0$ and $dQ^{(k_0)} = \frac{d\Pi^{(k_0)}\1_{\Theta^{(k_0)}}}{\Pi^{(k_0)}(\Theta^{(k_0)})}$. According to prior structure, $Q^{(k_0)}\in\mathcal{S}_{\rm MF}^{(k_0)}$, and
\begin{eqnarray*}
&&Q^{(k_0)}D\left(P_0^{(n)}\|P_{k_0,\theta^{(k_0)}}^{(n)}\right)\leq\max_{\theta^{(k)}\in\Theta^{(k_0)}}D\left(P_0^{(n)}\|P_{k_0,\theta^{(k_0)}}^{(n)}\right)\\
&\leq&\max_{\theta^{(k)}\in\Theta^{(k_0)}}D_\rho\left(P_0^{(n)}\|P_{k_0,\theta^{(k_0)}}^{(n)}\right)\lesssim n\epsilon_n^2.
\end{eqnarray*}
We also have
\[D\left(Q^{(k_0)}\|\Pi^{(k_0)}\right)-\log\pi(k_0) = -\sum_{j=1}^{m_{k_0}}\Pi_j^{(k_0)}(\Theta_j^{(k_0)})-\log\pi(k_0)\lesssim n\epsilon_n^2.\]
Hence, we obtain the desired result.
\end{proof}

\subsection{Proofs of Theorem \ref{thm:VB of gauss}, Proposition \ref{prop:lower-GSM}, Theorem \ref{thm:gauss1} and Theorem \ref{thm:gauss2}}

\begin{proof}[Proof of Theorem \ref{thm:VB of gauss}]
Theorem \ref{thm:VB of gauss} can be regarded as a simple application of Corollary \ref{cor:VB-sieve-factorize} with $\Theta_{j1} = \R\backslash\{0\}$, $\Theta_{j2} = \{0\}$ and $p(X_j^{(n)}|\theta_j)\propto\exp\left(-\frac{n}{2}(X_j-\theta_j)^2\right)$. The proof of Corollary \ref{cor:VB-sieve-factorize} will be given in Section \ref{sec:pf-EB}.
\end{proof}

To show Proposition \ref{prop:lower-GSM}, the following lemma is needed.

\begin{lemma}\label{lem:tilde k}
For the prior distribution $\Pi$ defined in (\ref{eq:refer-later-sieve}), we assume that $\max_j\|f_j\|_{\infty}\leq a$ and $\pi(k)$ is nonincreasing over $k$. Then, we have
$$P_{\theta^*}^{(n)}\widetilde{k}\lesssim \left(\frac{n}{\log n}\right)^{\frac{1}{2\alpha+1}},$$
for any $\theta^*\in\Theta_{\alpha}(B)$, where $\mathbb{P}_{\theta} = \otimes_{j=1}^\infty N(\theta_j,n^{-1/2})$.
\end{lemma}

\begin{proof}
We use the notation
$$W_j=\int f_j(\theta_j)\exp\left(-\frac{n(\theta_j-Y_j)^2}{2}\right)d\theta_j.$$
By the condition $\|f_j\|_{\infty}\leq a$, we have $W_j\leq a\sqrt{\frac{2\pi}{n}}\leq 1$. Define the objective function
$$L(k)=\sum_{j<k}\log\frac{1}{W_j}+\sum_{j>k}\frac{nY_j^2}{2}-\log\left(\pi(k-1)\exp\left(-\frac{nY_k^2}{2}\right)+\pi(k)Z_k\right).$$
It is easy to check that
$$\widetilde{k}=\argmax_k\left(\pi(k-1|Y)+\pi(k|Y)\right)=\argmin_kL(k).$$
To give a bound for $\widetilde{k}$, we first study the difference $L(k_1)-L(k_2)$ for any $k_1<k_2$. We use the inequalities
\[\log\left(\frac{\pi(k-1)\exp(-\frac{n}{2}Y_k^2)+\pi(k)W_k}{\pi(k-1)+\pi(k)}\right)\leq\max\left\{-\frac{n}{2}Y_k^2, \log W_k\right\}\leq 0,\]
and
\[\log\left(\frac{\pi(k-1)\exp(-\frac{n}{2}Y_k^2)+\pi(k)W_k}{\pi(k-1)+\pi(k)}\right)\geq\min\left\{-\frac{n}{2}Y_k^2, \log W_k\right\}\geq -\frac{n}{2}Y_k^2+\log W_k.\]
Then, we have
\begin{eqnarray*}
L(k_1) - L(k_2) &\leq& \sum_{j=k_1}^{k_2}\frac{nY_j^2}{2} + \sum_{j=k_1+1}^{k_2-1}\log W_j  + \log\left(\frac{\pi(k_2-1)+\pi(k_2)}{\pi(k_1-1)+\pi(k_1)}\right) \\
&\leq& \sum_{j=k_1}^{k_2}\frac{nY_j^2}{2} - (k_2-k_1-1)\left(\frac{1}{2}\log n-\log(a\sqrt{2\pi})\right) \\
&\leq& n\sum_{j=k_1}^{k_2}\theta_j^{*2} + \sum_{j=k_1}^{k_2}Z_j^2 - (k_2-k_1-1)\left(\frac{1}{2}\log n-\log(a\sqrt{2\pi})\right) \\
&\leq& nB^2k_1^{-2\alpha} + \sum_{j=k_1}^{k_2}Z_j^2 - (k_2-k_1-1)\left(\frac{1}{2}\log n-\log(a\sqrt{2\pi})\right),
\end{eqnarray*}
where $Z_j\sim N(0,1)$.
Now we bound $P_{\theta^*}^{(n)}\widetilde{k}$ by
\begin{equation}
P_{\theta^*}^{(n)}\widetilde{k}\leq Ck_0 + \sum_{l>Ck_0}lP_{\theta^*}^{(n)}(\widetilde{k}=l),\label{eq:exp-tail}
\end{equation}
where $k_0=\ceil{\left(\frac{n}{\log n}\right)^{\frac{1}{2\alpha+1}}}$, and $C$ is some  large constant. For each $l>Ck_0$,
\begin{eqnarray*}
{P}_{\theta^*}^{(n)}(\widetilde{k}=l) &\leq& {P}_{\theta^*}^{(n)}\left(L(l)\leq L(k_0)\right) \\
&\leq& \mathbb{P}\left(nB^2k_0^{-2\alpha} + \sum_{j=k_0}^{l}Z_j^2 - (l-k_0-1)\left(\frac{1}{2}\log n-\log(a\sqrt{2\pi})\right)\geq 0\right) \\
&\leq& \mathbb{P}\left(\sum_{j=k_0}^{l}Z_j^2\geq (l-k_0-1)\left(\frac{1}{2}\log n-\log(a\sqrt{2\pi})\right)-C_1\left(\frac{n}{\log n}\right)^{\frac{2\alpha}{2\alpha+1}}\right)  \\
&\leq& \mathbb{P}\left(\sum_{j=k_0}^{l}Z_j^2\geq c(l-k_0-1)\log n\right),
\end{eqnarray*}
where the last inequality is by the fact that $C_1\left(\frac{n}{\log n}\right)^{\frac{2\alpha}{2\alpha+1}}$ is of a smaller order than $(l-k_0-1)\log n$. Finally, a standard chi-squared tail bound gives
$$P_{\theta^*}^{(n)}(\widetilde{k}=l)\lesssim \exp\left(-C'(l-k_0)\log n\right).$$
Using (\ref{eq:exp-tail}) and summing over $l$, we get $P_{\theta^*}^{(n)}\widetilde{k}\lesssim k_0$, and the proof is complete.
\end{proof}

\begin{proof}[Proof of Proposition \ref{prop:lower-GSM}]
According to Theorem \ref{thm:VB of gauss}, the variational posterior $\wh{Q}$ is a product measure, and for any coordinate after a $\widetilde{k}$, the component is $\delta_0$.
By Theorem \ref{lem:tilde k}, we know that $P_{\theta^*}^{(n)}\widetilde{k}\leq C\left(\frac{n}{\log n}\right)^{\frac{1}{2\alpha+1}}$. Use the notation $\bar{k}=C\left(\frac{n}{\log n}\right)^{\frac{1}{2\alpha+1}}$. Then, we have $P_{\theta^*}^{(n)}\left(\widetilde{k}>2\bar{k}\right)\leq 1/2$ by Markov inequality. Consider a $\theta^*$ with every entry zero except that $\theta_{\ceil{2\bar{k}}}^*=B\ceil{2\bar{k}}^{-\alpha}$. It is easy to check that $\theta^*\in\Theta_{\alpha}(B)$. For this $\theta^*$, we have
\begin{eqnarray*}
P_{\theta^*}^{(n)}\wh{Q}\|\theta-\theta^*\|^2 &\geq& P_{\theta^*}^{(n)}\wh{Q}(\theta_{\ceil{2\bar{k}}}-\theta_{\ceil{2\bar{k}}}^*)^2\1_{\{\widetilde{k}\leq 2\bar{k}\}} \\
&=& \theta_{\ceil{2\bar{k}}}^{*2}P_{\theta^*}^{(n)}\left(\widetilde{k}\leq 2\bar{k}\right) \\
&\geq& \frac{1}{2}\theta_{\ceil{2\bar{k}}}^{*2} \\
&\asymp&n^{-\frac{2\alpha}{2\alpha+1}}(\log n)^{\frac{2\alpha}{2\alpha+1}}.
\end{eqnarray*}
Thus, the proof is complete.
\end{proof}

The proofs of Theorem \ref{thm:gauss1} and Theorem \ref{thm:gauss2} will be split into the following three lemmas. Recall that we use the loss $L(P_{\theta}^{(n)}, P_{\theta^*}^{(n)})=n\|\theta-\theta^*\|^2$ for this model.
\begin{lemma}\label{lem:gauss1}
For the prior $\Pi$ that satisfies (\ref{eq:gauss1-condition1}), the conditions (\ref{eq:C1}) and (\ref{eq:C2}) hold for all $\epsilon\geq n^{-1/2}$.
\end{lemma}
\begin{proof}
Given any $\epsilon\geq n^{-1/2}$ and any $C>0$, we define
$$\Theta_n(\epsilon)=\left\{\theta=(\theta_j): \theta_j=0,\text{ for all }j>Cn\epsilon^2/C_2\right\}.$$
Then, by (\ref{eq:gauss1-condition1}), we have
$$\Pi(\Theta_n(\epsilon)^c)\leq \Pi(k>Cn\epsilon^2/C_2)\lesssim \exp\left(-Cn\epsilon^2\right).$$
This proves (\ref{eq:C2}). To show (\ref{eq:C1}), we consider the following testing problem,
$$H_0:\theta=\theta^*,\quad H_1:\theta\in\Theta_n(\epsilon)\text{ and }\|\theta-\theta^*\|^2\geq \wt{C}\epsilon^2.$$
Define $N(\delta,S,d)$ as the $\delta$-covering number of a set $S$ under a metric $d$. Then, according to Lemma 5 in \cite{ghosal2007convergence} and Theorem 7.1 in \cite{ghosal2000convergence}, it is sufficient to establish the bound
$$\log N\left(\epsilon/8,\{\theta\in\Theta_n(\epsilon):\|\theta-\theta^*\|\leq\epsilon\},\|\cdot\|\right)\lesssim n\epsilon^2.$$
This is obviously true given a standard volume ratio calculation in a Euclidean space of dimension $\ceil{Cn\epsilon^2/C_2}$. Then, by Theorem 7.1 in \cite{ghosal2000convergence}, there exists a testing procedure $\phi_n$ such that (\ref{eq:C1}) holds. Note that the testing error can be arbitrarily small given a sufficiently large $\wt{C}>0$.
\end{proof}

\begin{lemma}\label{lem:gauss1-2}
Assume $\theta^*\in\Theta_\alpha(B)$.
For the prior $\Pi$ that satisfies (\ref{eq:gauss1-condition2}) and (\ref{eq:gauss1-condition3}), the conditions (\ref{eq:C3}) and (\ref{eq:C4}) hold for $\epsilon_n=n^{-\frac{\alpha}{2\alpha+1}}(\log n)^{\frac{\alpha}{2\alpha+1}}$.
\end{lemma}
\begin{proof}
We first show (\ref{eq:C4}). We will apply Theorem \ref{thm:mean field} by constructing a $\wt{Q}\in\mathcal{S}_{\rm MF}$ and $\otimes_j\wt{\Theta}_j$ that satisfy the conditions (\ref{eq:mfc1}) and (\ref{eq:mfc2}).
Define $\wt\Theta_j = [\theta_j^*-n^{-1/2},\theta_j^*+n^{-1/2}]$ for all $j\leq k_0$ and $\wt{\Theta}_j=\{0\}$ for all $j>k_0$, where $k_0=\left\lceil\left(\frac{n}{\log n}\right)^{\frac{1}{2\alpha+1}}\right\rceil$ is the same as defined in \ref{eq:gauss1-condition2}. We also define the measure $\wt{Q}$ by
$$d\wt{Q}(\theta)=\prod_{j=1}^{k_0}f_j(\theta_j)\prod_{j>k_0}\delta_0(\theta_j)d\theta.$$
It is easy to see that $\wt{Q}\in\mathcal{S}_{\rm MF}$.
For any $\theta\in\otimes_j\wt\Theta_j$, we have
\begin{equation}\label{eq:gauss1-part1}
D_2(P_{\theta^*}^{(n)}\| P_\theta^{(n)}) =  2D(P_{\theta^*}^{(n)}\| P_\theta^{(n)})= n\|\theta-\theta^*\|^2 \leq k_0\lesssim n\epsilon_n^2,
\end{equation}
and
\begin{equation}\label{eq:gauss1-part2}
\log\frac{d\wt Q(\theta)}{d\Pi(\theta)} \leq \log\frac{1}{\pi(k_0)}\leq -\log C_3+C_4k_0\log k_0\lesssim n\epsilon_n^2.
\end{equation}
Therefore, the condition (\ref{eq:mfc1}) holds.
To check the condition (\ref{eq:mfc2}), we use the bound
\begin{eqnarray*}
&&-\sum_{j=1}^\infty\log\wt Q_j(\wt\Theta_j) = -\sum_{j=1}^{k_0}\log\wt Q_j(\wt\Theta_j) = -\sum_{j=1}^{k_0}\log\int_{\theta_j^*-n^{-1/2}}^{\theta_j^*+n^{-1/2}}f_j(x)dx\\
&\leq&-k_0\log(2n^{-1/2})-\frac{1}{2n^{-1/2}}\sum_{j=1}^{k_0}\int_{\theta_j^*-n^{-1/2}}^{\theta_j^*+n^{-1/2}}\log f_j(x)dx,
\end{eqnarray*}
where we have used Jensen's inequality above. We are going to bound each of the integral above using (\ref{eq:gauss1-condition3}). For any $j\leq k_0$, we have
\begin{eqnarray*}
&& -\frac{1}{2n^{-1/2}}\int_{\theta_j^*-n^{-1/2}}^{\theta_j^*+n^{-1/2}}\log f_j(x)dx \leq c_0+c_1j^{2\alpha+1}(3\theta_j^{*2}+n^{-1}) \\
&\leq&c_0+3c_1k_0j^{2\alpha}\theta_j^{*2}+c_1k_0^{2\alpha+1}n^{-1}\leq c_0+c_1+3c_1k_0j^{2\alpha}\theta_j^{*2}.
\end{eqnarray*}
Hence, we get
\begin{equation}
-\sum_{j=1}^\infty\log\wt Q_j(\wt\Theta_j)\leq \frac{1}{2}k_0\log n+(c_0+c_1-\log 2)k_0+3c_1k_0\sum_{j}j^{2\alpha}\theta_j^{*2}\lesssim n\epsilon_n^2,\label{eq:gauss1-part3}
\end{equation}
which implies that (\ref{eq:mfc2}) holds. The condition (\ref{eq:C4}) is thus proved by applying Theorem \ref{thm:mean field}.

Finally, we derive the condition (\ref{eq:C3}). In view of (\ref{eq:gauss1-part1}), there is a constant $C>0$, such that
\begin{eqnarray*}
&& -\log\Pi\left(D_2(P_{\theta^*}^{(n)}\| P_\theta^{(n)})\leq Cn\epsilon_n^2\right) \\
&\leq& -\log\pi(k_0) - \log\wt{Q}\left(D_2(P_{\theta^*}^{(n)}\| P_\theta^{(n)})\leq Cn\epsilon_n^2\right) \\
&\leq& -\log\pi(k_0) -\sum_{j=1}^\infty\log\wt Q_j(\wt\Theta_j) \lesssim n\epsilon_n^2.
\end{eqnarray*}
The last inequality above is by (\ref{eq:gauss1-part2}) and (\ref{eq:gauss1-part3}). Hence, the proof is complete.
\end{proof}

\begin{lemma}\label{lem:gauss2-2}
Assume $\theta^*\in\Theta_\alpha(B)$.
For the prior $\Pi$ that satisfies (\ref{eq:gauss2-condition2}) and (\ref{eq:gauss2-condition3}), the conditions (\ref{eq:C3}) and (\ref{eq:C4}) hold for $\epsilon_n=n^{\frac{2\alpha}{2\alpha+1}}$.
\end{lemma}

\begin{proof}
The proof is essentially the same as that of Lemma \ref{lem:gauss1-2}. We define $\wt{Q}\in\mathcal{S}_{\rm MF}$ and $\otimes_j\wt{\Theta}_j$ in the same way except that $k_0=\ceil{n^{\frac{1}{2\alpha+1}}}$. Then, by the same calculation, we have for any $\theta\in\otimes_j\wt{\Theta}_j$,
$$D_2(P_{\theta^*}^{(n)}\| P_\theta^{(n)}) =  2D(P_{\theta^*}^{(n)}\| P_\theta^{(n)})\lesssim n\epsilon_n^2,\quad\text{and}\quad\log\frac{d\wt Q(\theta)}{d\Pi(\theta)} \lesssim n\epsilon_n^2.$$
Therefore, the condition (\ref{eq:mfc1}) holds. For any $j\leq k_0$,
\begin{eqnarray*}
-\log\wt Q_j(\wt{\Theta}_j) &=& -\log\int_{\sqrt{n}\theta_j^*-1}^{\sqrt{n}\theta_j^*+1}g_j(x)dx \\
&\leq& c_0-\log 2 - \log\int_{\sqrt{n}\theta_j^*-1}^{\sqrt{n}\theta_j^*+1}\frac{1}{2}\exp(-c_1|x|^{\beta})dx \\
&\leq& c_0-\log 2 + \frac{c_1}{2}\int_{\sqrt{n}\theta_j^*-1}^{\sqrt{n}\theta_j^*+1}|x|^{\beta}dx
\end{eqnarray*}
By H\"{o}lder's inequality,
\begin{eqnarray*}
\int_{\sqrt{n}\theta_j^*-1}^{\sqrt{n}\theta_j^*+1}|x|^{\beta}dx&\leq&\left(\int_{\sqrt{n}\theta_j^*-1}^{\sqrt{n}\theta_j^*+1}x^2\right)^{\beta/2}\left(\int_{\sqrt{n}\theta_j^*-1}^{\sqrt{n}\theta_j^*+1}1\right)^{(2-\beta)/2}\\
&= &[2(n\theta_j^{*2}+1)]^{\frac{\beta}{2}}\cdot 2^{\frac{2-\beta}{2}} = 2(n\theta_j^{*2}+1)^{\beta/2}\\
&\leq&4(n\theta_j^{*2})^{\beta/2}+4.
\end{eqnarray*}
Therefore,
$$-\sum_{j=1}^{k_0}\log\wt Q_j(\wt{\Theta}_j)\leq (4+c_0-\log 2)k_0 + 4n^{\beta/2}\sum_{j=1}^{k_0}|\theta_j^*|^{\beta}.$$
Using H\"{o}lder's inequality again, we get
\[\sum_{j=1}^{k_0}|\theta_j^*|^\beta\leq\left(\sum_{j=1}^{k_0}j^{2\alpha}\theta_j^{*2}\right)^{\beta/2}\left(\sum_{j=1}^{k_0}j^{-\frac{2\alpha\beta}{2-\beta}}\right)^{1-\beta/2}.\]
Set $t = \frac{2\alpha\beta}{2-\beta}$. As $0<\beta<\frac{2}{2\alpha+1}$, we have $t\in (0,1)$. Then
\[\sum_{j=1}^{k_0}j^{-t}\leq 1+\int_1^{k_0}x^{-t}dx = \frac{k_0^{1-t}+t}{1-t}<c_2k_0^{1-t}.\]
Thus,
$$
\sum_{j=1}^{k_0}|\theta_j^*|^\beta\leq B^\beta c_2k_0^{(1-t)\frac{2-\beta}{2}} \lesssim n^{1/(2\alpha+1)-\beta/2}.
$$
This leads to the desired bound $-\sum_{j=1}^{k_0}\log\wt Q_j(\wt{\Theta}_j)\lesssim n\epsilon_n^2$ in (\ref{eq:mfc2}). The condition (\ref{eq:C4}) is thus proved by applying Theorem \ref{thm:mean field}.

The condition (\ref{eq:C3}) can be derived in the same way as in the proof of Lemma \ref{lem:gauss1-2}.
\end{proof}

\begin{proof}[Proofs of Theorem \ref{thm:gauss1} and Theorem \ref{thm:gauss2}]
The results are directly implied by Lemma \ref{lem:gauss1}, Lemma \ref{lem:gauss1-2} and Lemma \ref{lem:gauss2-2}. 
\end{proof}

\subsection{Proof of Theorem \ref{thm:density1}}

For Theorem \ref{thm:density1}, the loss function is $L(P_{\theta}^n,P_{\theta^*}^n)=nH^2(P_{\theta},P_{\theta^*})$.
We split the proof of Theorem \ref{thm:density1} into following two lemmas.

\begin{lemma}\label{lem:density1}
Assume $\theta^*\in\Theta_\alpha(B)$ for $\alpha>1/2$. For the prior $\Pi$ that satisfies (\ref{eq:density-condition1}) and (\ref{eq:density-condition3}), the conditions (\ref{eq:C1}) and (\ref{eq:C2}) hold for all $\epsilon\geq \left(\frac{\log n}{n}\right)^{\frac{\alpha}{2\alpha+1}}$.
\end{lemma}

\begin{lemma}\label{lem:density2}
Assume $\theta^*\in\Theta_\alpha(B)$ for $\alpha>1/2$. For the prior $\Pi$ that satisfies (\ref{eq:density-condition2}) and (\ref{eq:density-condition4}), the conditions (\ref{eq:C3}) and (\ref{eq:C4}) hold for $\epsilon_n^2=\left(\frac{\log n}{n}\right)^{\frac{2\alpha}{2\alpha+1}}$.
\end{lemma}

Before proving these two lemmas, we need the following two results that establish relations between different divergence functions for the exponential family model.

\begin{lemma}\label{lem:equiv1}
If $\|\theta-\theta'\|_1\leq \frac{1}{\sqrt{2}}$, then
\[H(P_\theta, P_{\theta'})\leq 2\sqrt{2}\|\theta-\theta'\|_1.\]
\end{lemma}

\begin{proof}
We first give some uniform bounds that are well known for exponential family density functions (see \cite{rivoirard2012posterior}). For any $\theta,\theta'$, we have
\begin{equation}
\left\|\log\frac{dP_{\theta}}{dP_{\theta'}}\right\|_{\infty} \leq 2\sqrt{2}\|\theta-\theta'\|_1.\label{eq:inf-log-ratio}
\end{equation}
We start from the left hand side of the inequality:
\begin{eqnarray*}
&&H^2(P_\theta, P_{\theta'})^2 = \frac{1}{2}\int\left(\sqrt{\frac{dP_{\theta}}{dP_{\theta'}}}-1\right)^2dP_{\theta'}\\
&\leq&\frac{1}{2}\int\left(\exp(\sqrt{2}\|\theta-\theta'\|_1)-1\right)^2dP_{\theta'}+\frac{1}{2}\int\left(\exp(-\sqrt{2}\|\theta-\theta'\|_1)-1\right)^2dP_{\theta'}\\
&\leq&\frac{1}{2}\int 8\|\theta-\theta'\|_1^2dP_{\theta'}+\frac{1}{2}\int 8\|\theta-\theta'\|_1^2dP_{\theta'}\\
&=&8\|\theta-\theta'\|_1^2,
\end{eqnarray*}
where we have applied the property that $\frac{e^x-1}{x}$ is monotonically increasing for all $x$. Then it follows that
$H(P_\theta, P_{\theta'})\leq 2\sqrt{2}\|\theta-\theta'\|_1$.
\end{proof}

\begin{lemma}\label{lem:equiv2}
For any $\theta$ and any $\theta^*\in\Theta_{\alpha}(B)$ with $\alpha>1/2$, we have
\begin{eqnarray*}
&& C_0^{-1}\exp\left(-3\sqrt{2}\|\theta^*-\theta\|_1\right)\|\theta^*-\theta\|^2\leq 2H^2(P_{\theta^*},P_{\theta})\leq D(P_{\theta^*}\|P_{\theta})\\
&& \leq D_2(P_{\theta^*}\|P_{\theta})\leq C_0\exp\left(3\sqrt{2}\|\theta^*-\theta\|_1\right)\|\theta^*-\theta\|^2,
\end{eqnarray*}
where the constant $C_0>0$ only depends on $\alpha$ and $B$.
\end{lemma}

\begin{proof}
For any $\theta^*\in\Theta_{\alpha}(B)$, we have $\left\|\log\frac{dP_{\theta^*}}{d\ell}\right\|_{\infty}\leq 2\sqrt{2}\|\theta^*\|_1$. Since
\begin{equation}
\|\theta^*\|_1^2\leq \left(\sum_{j=1}^\infty j^{-2\alpha}\right)\left(\sum_{j=1}^\infty j^{2\alpha}\theta_j^{*2}\right)\leq B^2\gamma_{\alpha},\label{eq:l1-bound}
\end{equation}
where $\gamma_{\alpha}=\sum_{j=1}^\infty j^{-2\alpha}=O(1)$ for $\alpha>1/2$. This gives
\begin{equation}
\left\|\log\frac{dP_{\theta^*}}{d\ell}\right\|_{\infty}\leq 2\sqrt{2}\gamma_{\alpha}^{1/2}B.\label{eq:inf-log-den}
\end{equation}

Now we proceed to show Lemma \ref{lem:equiv2}. Given the result of Proposition \ref{prop:div}, it is sufficient to prove the first and the last inequalities. Define
$$V(P_{\theta^*},P_{\theta})=\int\left(\log\frac{dP_{\theta^*}}{dP_{\theta}}-D(P_{\theta^*}\|P_{\theta})\right)^2dP_{\theta^*}.$$
Following the argument in the proof of Lemma 3.2 in \cite{gao2016rate}, we have
$$e^{-\left\|\log\frac{dP_{\theta^*}}{d\ell}\right\|_{\infty}}\|\theta^*-\theta\|^2\leq V(P_{\theta^*},P_{\theta})\leq 4H^2(P_{\theta^*},P_{\theta})e^{3/2\|\log\frac{dP_{\theta}}{dP_{\theta^*}}\|_{\infty}}.$$
By (\ref{eq:inf-log-ratio}) and (\ref{eq:inf-log-den}), we have
$$C_0^{-1}\|\theta-\theta^*\|^2\leq 2H^2(P_{\theta^*},P_{\theta})\exp\left(3\sqrt{2}\|\theta-\theta^*\|_1\right),$$
for $C_0 = 2\exp(2\sqrt{2}\gamma_\alpha^{1/2}B)$, which implies the first inequality.

For the last inequality, we have
\begin{eqnarray*}
D_2(P_{\theta^*}\|P_{\theta})
&=&\log\left(\int dP_{\theta^*}\exp\left(\log\frac{dP_{\theta^*}}{dP_\theta}\right)\right)\\
& = &\log\left(1+\sum_{l=1}^\infty\frac{1}{l!}\int dP_{\theta^*}\left(\log\frac{dP_{\theta^*}}{dP_\theta}\right)^l\right)\\
&\leq&\log\left(1+D(P_{\theta^*}\| P_{\theta})\sum_{l=1}^\infty\frac{1}{l!}\left\|\log\frac{dP_{\theta^*}}{dP_\theta}\right\|_\infty^{l-1}\right)\\
&\leq& D(P_{\theta^*}\| P_\theta)\exp\left(\left\|\log\frac{dP_{\theta^*}}{dP_\theta}\right\|_\infty\right) \\
&\leq& D(P_{\theta^*}\| P_\theta)e^{2\sqrt{2}\|\theta-\theta^*\|_1},
\end{eqnarray*}
where we have used the inequality that $\frac{e^x-1}{x}\leq e^x$ for all $x>0$ and the last inequality is by (\ref{eq:inf-log-ratio}). By the same argument in the proof of Lemma 3.2 in \cite{gao2016rate}, we have
$$D(P_{\theta^*}\| P_\theta)\leq e^{\sqrt{2}\|\theta-\theta^*\|_1+2\sqrt{2}\|\theta^*\|_1}\|\theta-\theta^*\|^2.$$
Therefore, we obtain the bound
$$D_2(P_{\theta^*}\|P_{\theta})\leq e^{3\sqrt{2}\|\theta-\theta^*\|_1+2\sqrt{2}\|\theta^*\|_1}\|\theta-\theta^*\|^2,$$
which implies the desired result by (\ref{eq:l1-bound}).
\end{proof}

Now we are ready to prove Lemma \ref{lem:density1} and Lemma \ref{lem:density2}.
\begin{proof}[Proof of Lemma \ref{lem:density1}]
Given any $\epsilon\geq \left(\frac{\log n}{n}\right)^{\frac{2\alpha}{2\alpha+1}}$, we define the set
$$\Theta_n(\epsilon)=\left\{\theta=(\theta_j): \theta_j\in[-w_n,w_n]\text{ for }1\leq j\leq k_n, \theta_j=0\text{ for }j>k_n\right\},$$
where $w_n=(\wt{C}n\epsilon^2)^{1/\beta}$ and $k_n=\left\lceil\frac{\wt{C}n\epsilon^2}{\log(n\epsilon^2)}\right\rceil$. We bound $\Pi(\Theta_n(\epsilon)^c)$ by
\begin{eqnarray*}
\Pi(\Theta_n(\epsilon)^c) &\leq& \Pi(k>k_n) + \sum_{j=1}^{k_n}\Pi(k = j)\sum_{i=1}^j\Pi(|\theta_i|>w_n|k = j) \\
&\leq& \Pi(k>k_n) + \sum_{j=1}^{k_n}\Pi(k = j)\sum_{i=1}^j\int_{|x|>w_n}f_i(x)dx \\
&\leq& \Pi(k>k_n) + \sum_{j=1}^{k_n}\int_{|x|>w_n}f_j(x)dx \\
&\leq& \Pi(k>k_n) + \sum_{j=1}^{k_n}e^{-c_1w_n^{\beta}/2}\int e^{c_1|x|^{\beta}/2}f_j(x)dx \\
&\leq& \Pi(k>k_n) + \sum_{j=1}^{k_n}e^{-c_1w_n^{\beta}/2}\int e^{c_1|x|^{\beta}/2-c_0-c_1|x|^{\beta}}dx \\
&\lesssim& \exp(-C_2k_n\log k_n) + k_n\exp\left(-c_1\wt{C}n\epsilon^2/2\right),
\end{eqnarray*}
where we have used the conditions (\ref{eq:density-condition1}) and (\ref{eq:density-condition3}). Therefore, for any $C>0$, we can choose a sufficiently large $\wt{C}$, such that $\Pi(\Theta_n(\epsilon)^c) \lesssim\exp(-Cn\epsilon^2)$, which proves (\ref{eq:C2}).

To prove (\ref{eq:C1}), we consider the following testing problem,
$$H_0:\theta=\theta^*,\quad H_1:\theta\in\Theta_n(\epsilon)\text{ and }H(P_{\theta}, P_{\theta^*})\geq C'\epsilon.$$
By Theorem 7.1 in \cite{ghosal2000convergence}, it is sufficient to establish the bound
$$\log N(\epsilon,\{P_{\theta}:\theta\in\Theta_n(\epsilon)\},H)\lesssim n\epsilon^2.$$
Note that for any $\theta,\theta'\in\Theta_n(\epsilon)$, we have $\|\theta-\theta'\|_1\leq \sqrt{k_n}\|\theta-\theta'\|$. Therefore, by Lemma \ref{lem:equiv1},
$$H(P_{\theta},P_{\theta'})\lesssim \|\theta-\theta'\|_1\leq \sqrt{k_n}\|\theta-\theta'\|,$$
when $\|\theta-\theta'\|_1\leq \frac{1}{\sqrt{2}}$. This means as long as $\|\theta-\theta'\|\leq k_n^{-1/2}(\epsilon\wedge 2^{-1/2})$, we have $H(P_{\theta},P_{\theta'})\lesssim \epsilon$. Thus, there exists a constant $c'$, such that
\begin{eqnarray*}
&& \log N(\epsilon,\{P_{\theta}:\theta\in\Theta_n(\epsilon)\},H) \\
&\leq& \log N\left(c'k_n^{-1/2}(\epsilon\wedge 2^{-1/2}),\{\theta\in\mathbb{R}^{k_n}:\|\theta\|^2\leq k_nw_n^2\},\|\cdot\|\right)\\
&\lesssim&k_n\log\left(\frac{k_nw_n}{c'(\epsilon\wedge 2^{-1/2})}\right)\\
&\lesssim&k_n\log(n\epsilon^2)\asymp n\epsilon^2,
\end{eqnarray*}
where we have used the condition $\epsilon\geq \left(\frac{\log n}{n}\right)^{\frac{\alpha}{2\alpha+1}}$ in the last two steps above. 

It implies the existence of a testing function that satisfies (\ref{eq:C1}). The testing error can be made arbitrarily small by choosing a sufficiently large $C'$. Hence, the proof is complete.
\end{proof}

\begin{proof}[Proof of Lemma \ref{lem:density2}]
In the first part of the proof, we derive (\ref{eq:C3}).
We take $k_0=\ceil{\left(n/\log n\right)^{\frac{1}{2\alpha+1}}}$.
Define $\wt{\Theta}=\otimes_j\wt{\Theta}_j$, where $\wt\Theta_j = [\theta_j^*-n^{-1/2},\theta_j^*+n^{-1/2}]$ for all $j\leq k_0$ and $\wt{\Theta}_j=\{0\}$ for all $j>k_0$.
Then, by Lemma \ref{lem:equiv1}, for all $\theta\in\wt{\Theta}$,
\begin{eqnarray*}
D_2(P_{\theta^*}\|P_{\theta})&\leq&  C_0\exp(3\sqrt{2}\|\theta^*-\theta\|_1)\|\theta-\theta^*\|^2\\
& = &C_0\exp\left(3\sqrt{2}\left(\frac{k_0}{\sqrt{n}}+\sum_{j>k_0}|\theta_j^*|\right)\right)\left(\frac{k_0}{n}+\sum_{j>k_0}\theta_j^{*2}\right)\\
&\leq&C_0\exp\left(3\sqrt{2}\left(n^{\frac{1-2\alpha}{2+4\alpha}}+B\gamma_\alpha^{1/2}\right)\right)\left(\frac{k_0}{n}+k_0^{-2\alpha}B^2\right)\\
&\lesssim&n\epsilon_n^2.
\end{eqnarray*}
where we have use the condition $\alpha>1/2$.

Therefore, it is sufficient to lower bound $\Pi(\wt{\Theta})$, which has been done in the proof of Lemma \ref{lem:gauss1-2}.

Now we will derive (\ref{eq:C4}).
Rather than using the results of Theorem \ref{thm:convergence2} or Theorem \ref{thm:mean field}, we will construct a $Q\in\mathcal{S}_{\rm G}$ and bound $R(Q)$ directly.
Note that in the current setting, we have
 \[R(Q) = \frac{1}{n}D(Q\|\Pi)+QD(P_{\theta^*}\| P_\theta).\]
 For $k_0=\ceil{\left(n/\log n\right)^{\frac{1}{2\alpha+1}}}$, define $Q=\otimes_jQ_j$, where $Q_j=N(\theta_j^*,n^{-1})$ for $j\leq k_0$ and $Q_j=N(0,0)$ for $j>k_0$. Then, it is easy to see that $Q\in\mathcal{S}_{\rm G}$.
 
We first give a bound for $D(Q\|\Pi)$. Let $F_j$ denote the probability distribution with density function $f_j$. Then, we have
$$D(Q\|\Pi)  \leq \log\frac{1}{\pi(k_0)}+\sum_{j=1}^{k_0}D\left(N(\theta_j^*,n^{-1})\|F_j\right),$$
where the first term on the right hand side above can be bounded as
$$\log\frac{1}{\pi(k_0)}\lesssim k_0\log k_0\lesssim n\epsilon_n^2,$$
according to the condition (\ref{eq:density-condition2}). For any $j\leq k_0$, we use $\psi_j$ to denote the density function of $N(\theta_j^*,n^{-1})$. Then, by (\ref{eq:density-condition4}), we have
\begin{eqnarray*}
D\left(N(\theta_j^*,n^{-1})\|F_j\right) &=& \int\psi_j\log\psi_j-\int\psi_j\log f_j \\
&\leq& \int\psi_j\log\psi_j + c_0' + c_1'j^{2\alpha+1}\int\phi_j(x)x^2dx \\
&=& \frac{1}{2}\log\left(\frac{n}{2\pi e}\right) + c_0' + c_1'j^{2\alpha+1}(n^{-1}+\theta_j^{*2}).
\end{eqnarray*}

Since $\theta^*\in\Theta_{\alpha}(B)$, we have
$$\sum_{j=1}^{k_0}D\left(N(\theta_j^*,n^{-1})\|F_j\right)\lesssim k_0\log n\lesssim n\epsilon_n^2.$$
Therefore, we have obtained $D(Q\|\Pi)\lesssim n\epsilon_n^2$.

We then derive a bound for $QD(P_{\theta^*}\| P_\theta)$.
For $j\leq k_0$, we write $\theta_j = \theta_j^*+\frac{1}{\sqrt{n}}Z_j$ where $Z_j\sim N(0,1)$. Then according to Lemma \ref{lem:equiv2}, it follows that
\begin{eqnarray}\label{eq:toshow}
\nonumber QD(P_{\theta^*}\| P_\theta) &\lesssim& Q\exp\left(3\sqrt{2}\|\theta-\theta^*\|_1\right)\|\theta-\theta^*\|^2 \\
\nonumber &=&Q\left[e^{3\sqrt{2}\sum_{j=1}^{k_0}|\theta_j-\theta_j^*|}\left(\sum_{j=1}^{k_0}(\theta_j-\theta_j^*)^2+\sum_{j>k_0}\theta_j^{*2}\right)\right] \\
&=& \mathbb{E}e^{3\sqrt{2}\sum_{j=1}^{k_0}|Z_j|/\sqrt{n}}\sum_{j=1}^{k_0}Z_j^2/n + \left(\sum_{j>k_0}\theta_j^{*2}\right)\mathbb{E}e^{3\sqrt{2}\sum_{j=1}^{k_0}|Z_j|/\sqrt{n}},
\end{eqnarray}
where the last inequality is by (\ref{eq:l1-bound}). Suppose we can show
\begin{equation}
\mathbb{E}e^{3\sqrt{2}\sum_{j=1}^{k_0}|Z_j|/\sqrt{n}} = O(1),\label{eq:chi-bound}
\end{equation}
and
\begin{equation}
\mathbb{E}Z_1^2e^{3\sqrt{2}\sum_{j=1}^{k_0}|Z_j|/\sqrt{n}} = O(1).\label{eq:chi-bound2}
\end{equation}
Then, up to a constant, (\ref{eq:toshow}) can be bounded by
$$\frac{k_0}{n}+\sum_{j>k_0}\theta_j^{*2}\lesssim\epsilon_n^2,$$
which further implies $QD(P_{\theta^*}\| P_\theta)\lesssim \epsilon_n^2$. 

To complete the proof, we show (\ref{eq:chi-bound}). We have
\begin{eqnarray*}
 \mathbb{E}e^{3\sqrt{2}\sum_{j=1}^{k_0}|Z_j|/\sqrt{n}} &\leq& \mathbb{E}\exp\left(\frac{3\sqrt{2}}{\sqrt{n}}\sum_{j=1}^{k_0}(1+Z_j^2)\right) \\
&=& \exp\left(\frac{3\sqrt{2}k_0}{\sqrt{n}}\right)\mathbb{E}\exp\left(\frac{3\sqrt{2}}{\sqrt{n}}\chi_{k_0}^2\right) \\
&=& \exp\left(\frac{3\sqrt{2}k_0}{\sqrt{n}}\right)\left(1-\frac{6\sqrt{2}}{\sqrt{n}}\right)^{-\frac{k_0}{2}}.
\end{eqnarray*}
Since $\alpha>1/2$, we have $k_0/\sqrt{n}=O(1)$, and thus (\ref{eq:chi-bound}) holds. For (\ref{eq:chi-bound2}), we have
$$\mathbb{E}Z_1^2e^{3\sqrt{2}\sum_{j=1}^{k_0}|Z_j|/\sqrt{n}}=\left(\mathbb{E}Z_1^2e^{3\sqrt{2}|Z_1|/\sqrt{n}}\right)\left(\mathbb{E}e^{3\sqrt{2}\sum_{j=2}^{k_0}|Z_j|/\sqrt{n}}\right).$$
Note that $\mathbb{E}Z_1^2e^{3\sqrt{2}|Z_1|/\sqrt{n}}=O(1)$, and $\mathbb{E}e^{3\sqrt{2}\sum_{j=2}^{k_0}|Z_j|/\sqrt{n}}$ shares the same bound for (\ref{eq:chi-bound}). This implies (\ref{eq:chi-bound2}) also holds.
\end{proof}

\begin{proof}[Proof of Theorem \ref{thm:density1}]
The result is immediately implied by Lemma \ref{lem:density1} and Lemma \ref{lem:density2}
in view of Theorem \ref{thm:general2}.
\end{proof}

\subsection{Proofs of Theorem \ref{thm:PC-MF}, Theorem \ref{thm:PC-MC} and Theorem \ref{thm:PC-solution}}\label{subsec:PC-MF}

\begin{proof}[Proof of Theorem \ref{thm:PC-MF}]
Recall that $\Theta_k$ is the space of piecewise constant vectors with at most $k$ pieces. Then, we have the partition
$$\mathbb{R}^n = \Theta_{n-1} \cup (\Theta_n\backslash\Theta_{n-1}).$$
First of all, we consider $\mathcal{S} = \mathcal{S}_{\rm MF}$. Suppose the measure $Q\in\mathcal{S}_{\rm MF}$ and $D(Q\|\Pi)<\infty$, then the support of $Q$ must be a subset of the support of $\Pi$. Note that the distributions $g_i$'s are all absolutely continuous. That is, for any singleton $x$, $\Pi(\theta_j = x) = 0$, which indicates that $Q(\theta_j = x) = 0$ for any singleton $x$. Thus, $Q$ is continuous in each coordinate and for any $j\in[n-1]$, $Q(\theta_j = \theta_{j+1})=\int Q(\theta_j = \theta_{j+1}=x)dx = \int  Q(\theta_j = x)Q(\theta_{j+1}=x)dx = 0$. Therefore, 
$$Q\left(\Theta_{n-1}\right)=Q\left(\text{there exists a $j\in[n-1]$, such that }\theta_j=\theta_{j+1}\right)=0,$$
because otherwise the independent structure of $Q$ would imply a delta measure for some coordinate, which leads to $D(Q\|\Pi)=\infty$.
This implies that $Q$ is supported on $\Theta_n\backslash\Theta_{n-1}$. Therefore,
$$D(Q\|\Pi)=\int \log\frac{\prod_{i=1}^nq_i(\theta_i)}{\Pi(\Theta_n\backslash\Theta_{n-1})\prod_{i=1}^ng(\theta_i)}dQ(\theta),$$
where $q_i(\theta_i) = \frac{dQ_i(\theta_i)}{d\theta_i}$. Then, by the definition of $\mathcal{S}_{{\rm MF}}$ and the independent structure of $P_{\theta}^{(n)}$, we have
\begin{eqnarray*}
\wh{Q}_{{\rm MF}} &=& \argmin_{Q\in\mathcal{S}_{\rm MF}}\left\{D(Q\|\Pi) + Q D(P_{\theta^*}^{(n)}\|P_{\theta}^{(n)})\right\} \nonumber\\
&=& \argmin_{Q:\frac{dQ(\theta)}{d\theta}=\prod_{i=1}^nq_i(\theta_i)}\left\{Q\sum_{i=1}^n\left(\log\frac{q_i(\theta_i)}{g(\theta_i)}+D\left(N(\theta_i^*,\sigma^2)\|N(\theta_i,\sigma^2)\right)\right)\right\}.
\end{eqnarray*}
This gives
\[\frac{d\wh Q_{\rm MF}(\theta)}{d\theta}\propto \prod_{i=1}^n g(\theta_i)\exp\left(-\frac{(\theta_i-X_i)^2}{2\sigma^2}\right).\]
In other words, the mean-field variational posterior $\wh{Q}_{{\rm MF}}$ is a product measure, and on each coordinate, it equals the posterior distribution induced by the prior $g_i$. Now we give a lower bound for $P_{\theta^*}^{(n)}\wh{Q}_{{\rm MF}}\|\theta-\theta^*\|^2$. Since $\|\theta-\theta^*\|^2=\sum_{i=1}^n(\theta_i-\theta_i^*)^2$, we have
$$P_{\theta^*}^{(n)}\wh{Q}_{{\rm MF}}\|\theta-\theta^*\|^2=\sum_{i=1}^nP_{\theta_i^*}\mathbb{E}\left((\theta_i-\theta_i^*)^2|X_i\right),$$
where we use $\mathbb{E}(\cdot|X_i)$ to stand for the posterior expectation of $\theta_i$ with the prior $\theta_i\sim g_i$. By Jensen's inequality,
$$\mathbb{E}\left((\theta_i-\theta_i^*)^2|X_i\right)\geq \left(\mathbb{E}(\theta_i|X_i)-\theta^*\right)^2.$$
Therefore,
\begin{eqnarray*}
&& \sup_{\theta^*\in\Theta_k(B)}P_{\theta^*}^{(n)}\wh{Q}_{{\rm MF}}\|\theta-\theta^*\|^2 \\ &\geq& \sup_{\theta^*\in\Theta_1(B)}\sum_{i=1}^nP_{\theta_i^*}\left(\mathbb{E}(\theta_i|X_i)-\theta_i^*\right)^2 \\
&\geq& \frac{1}{2}\sum_{i=1}^nP_{\theta_i^*=-B}\left(\mathbb{E}(\theta_i|X_i)-\theta_i^*\right)^2 + \frac{1}{2}\sum_{i=1}^nP_{\theta_i^*=B}\left(\mathbb{E}(\theta_i|X_i)-\theta_i^*\right)^2 \\
&=& \frac{1}{2}\sum_{i=1}^n\left(P_{\theta_i^*=-B}\left(\mathbb{E}(\theta_i|X_i)-\theta_i^*\right)^2+P_{\theta_i^*=B}\left(\mathbb{E}(\theta_i|X_i)-\theta_i^*\right)^2\right)  \\
&\geq& \sum_{i=1}^nB^2\int\min\left(dN(B,\sigma^2),dN(-B,\sigma^2)\right) \\
&\gtrsim& n.
\end{eqnarray*}

Next, we consider $\mathcal{S} = \mathcal{S}_{\rm MF}^{\rm joint}$. As $\wh Q_{\rm MF}^{\rm joint}\in\mathcal{S}_{\rm MF}^{\rm joint}$, we can assume
\[d\wh Q_{\rm MF}^{\rm joint}(w, z, \theta) = d\wh Q^{(w)}(w)\prod_{i=1}^nd\wh Q_i^{(z)}(z)\prod_{i=1}^nd\wh Q_i^{(\theta)}(\theta_i).\]
For the same reason, $\prod_{i=1}^nd\wh Q^{(\theta)}(\theta_i)$ is supported on $\Theta_n\backslash\Theta_{n-1}$.  The joint distribution of prior is written as
\[\frac{\Gamma(\alpha_0+\beta_0)}{\Gamma(\alpha_0)\Gamma(\beta_0)}w^{\alpha_0+\sum_{i=2}^nz_i-1}(1-w)^{\beta_0+n-2-\sum_{i=2}^nz_i}g(\theta_1)\prod_{i=2}^n g(\theta_i)^{z_i}\delta_{\theta_{i-1}}^{1-z_i}.\]
Thus, conditioning on $\theta\in\Theta_n\backslash\Theta_{n-1}$, $z_i = 1$ for all $2\leq i\leq n$. In other words, $\wh Q_i^{(z)}(z_i = 1) =1$ for all $2\leq i\leq n$. Plug it in the definition of $\wh Q_{\rm MF}^{\rm joint}$, we have
\begin{eqnarray*}
&&\left(\wh Q^{(\theta)}, \wh Q^{(w)}\right)\\
& =& \argmin_{\substack{(Q^{(\theta)}, Q^{(w)})\\ dQ^{(\theta)}=\prod_{i=1}^nq_i^{(\theta)}(\theta_i)d\theta\\ dQ^{(w)} = q^{(w)}(w)dw}}\left\{Q^{(w)}\log\frac{q^{(w)}}{\pi(w)w^{n-1}}\right.\\
&&\left.+Q^{(\theta)}\sum_{i=1}^n\left(\log\frac{q_i^{(\theta)}(\theta_i)}{g(\theta_i)}+D\left(N(\theta_i^*,\sigma^2)\|N(\theta_i,\sigma^2)\right)\right)\right\}.
\end{eqnarray*}
This gives $\frac{d\wh Q^{(w)}(w)}{dw}\propto\pi(w)w^{n-1}$ and $\wh Q^{(\theta)}(\theta) = \wh Q_{\rm MF}(\theta)$. It implies that 
\begin{eqnarray*}
 &&\sup_{\theta^*\in\Theta_k(B)}P_{\theta^*}^{(n)}\wh{Q}_{{\rm MF}}^{\rm joint}\|\theta-\theta^*\|^2 =\sup_{\theta^*\in\Theta_k(B)}P_{\theta^*}^{(n)}\wh Q^{(\theta)}\|\theta-\theta^*\|^2\\
 &=&\sup_{\theta^*\in\Theta_k(B)}P_{\theta^*}^{(n)}\wh{Q}_{{\rm MF}}\|\theta-\theta^*\|^2 \gtrsim n
\end{eqnarray*}
The proof is complete.
\end{proof}

\begin{proof}[Proof of Theorem \ref{thm:PC-MC}]
This theorem is a special case of Theorem \ref{thm:PC-MC-misp}, whose proof is given in Section \ref{sec:pf-misp}.
\end{proof}

\begin{proof}[Proof of Theorem \ref{thm:PC-solution}]
If $D\left(Q(w,z,\theta)\|\Pi(w, z, \theta|Y)\right)<\infty$, we will have $\supp(Q)\subseteq\supp(\Pi(\cdot|Y))\subseteq\supp(\Pi)$. For $Q\in\mathcal{S}_{\rm MF}^{\rm joint}$, as $\Pi(z_i = 0, \theta_i\neq\theta_{i-1}) = \Pi(z_i = 1, \theta_i= \theta_{i-1}) = 0$, we can conclude that $Q_i^{(z)}(z_i = 0)Q_i^{(\theta)}(\theta_i\neq\theta_{i-1}|\theta_{i-1}) = 0$ and $Q_i^{(z)}(z_i = 1)Q_i^{(\theta)}(\theta_i = \theta_{i-1}|\theta_{i-1}) = 0$. In other words, the conclusion leads to $Q_i^{(z)}(z_i = 1) = 0$, $Q_i^{(\theta)}(\theta_i\neq\theta_{i-1}|\theta_{i-1}) = 0$ or $Q_i^{(z)}(z_i = 1) = 1$, $Q_i^{(\theta)}(\theta_i=  \theta_{i-1}|\theta_{i-1}) = 0$.

Thus, we can define a set $S\subseteq\{2,3,\cdots, n\}$, such that for $i\not\in S$, $Q_i^{(z)}(z_i = 1) = 0$ and $dQ_i^{(\theta)}(\theta_i|\theta_{i-1}) = \delta_{\theta_{i-1}}(\theta_i)d\theta_i$, whereas for $i\in S$, $Q_i^{(z)}(z_i = 1) = 1$ and $dQ_i^{(\theta)}(\theta_i|\theta_{i-1}) = q_i^{(\theta)}(\theta_i|\theta_{i-1})d\theta_i$, a continuous density function. Then we can write
\[\frac{dQ(w,z ,\theta)}{dwd\theta} = q^{(w)}(w)q_1^{(\theta)}(\theta_1)\prod_{i\in S}\1_{z_i = 1}(z_i)q_i^{(\theta)}(\theta_i|\theta_{i-1})\prod_{i\not\in S}\1_{z_i = 0}(z_i)\delta_{\theta_{i-1}}(\theta_i).\]
Plug it into $D(Q(w,z,\theta)\|\Pi(w,z,\theta|Y))$, and we get
\begin{eqnarray*}
&&D\left(Q(w,z,\theta)\|\Pi(w,z,\theta|Y)\right) \\
&=&\int q^{(w)}(w)\log\frac{q^{(w)}(w)}{\pi(w)w^{|S|}(1-w)^{n-1-|S|}}dw\\
&&+\int_{\Theta(S)}q_1^{(\theta)}(\theta_1)\prod_{i\in S}q_i^{(\theta)}(\theta_i|\theta_{i-1})\prod_{i\not\in S}\delta_{\theta_{i-1}}(\theta_i)\\
&&\times\log\frac{q_1^{(\theta)}(\theta_1)\prod_{i\in S}q_i^{(\theta)}(\theta_i|\theta_{i-1})\prod_{i\not\in S}\delta_{\theta_{i-1}}(\theta_i)}{g(\theta_1)\prod_{i\in S}g(\theta_i)\prod_{i\not\in S}\delta_{\theta_{i-1}}(\theta_i)\exp\left(-\frac{1}{2}\sum_{i=1}^n(Y_i-\theta_i)^2\right)}d\theta,
\end{eqnarray*}
where
\[\Theta(S) = \{\theta|\theta_i = \theta_{i-1}\mbox{ for }i\not\in S\mbox{ and }\theta_i\neq\theta_{i-1}\mbox{ for }i\in S\}.\]
Then
\begin{eqnarray}\label{eq:KL_joint}
\nonumber&&\min_{Q\in\mathcal{S}_{\rm MC}^{\rm joint}}D\left(Q(w,z,\theta)\|\Pi(w,z,\theta|Y)\right)\\
\nonumber&\Leftrightarrow&\min_{S}\left\{\min_{Q^{(\theta)}\in\mathcal{S}_{\rm MC}, Q^{(w)}}\left\{\int q^{(w)}(w)\log\frac{q^{(w)}(w)}{\pi(w)w^{|S|}(1-w)^{n-1-|S|}}dw\right.\right.\\
&&\left.\left.+\int_{\Theta(S)}q^{(\theta)}(\theta)\log\frac{q^{(\theta)}(\theta)}{g(\theta_1)\prod_{i\in S}g(\theta_i)\prod_{i\not\in S}\delta_{\theta_{i-1}}(\theta_i)\exp\left(-\frac{1}{2}\sum_{i=1}^n(Y_i-\theta_i)^2\right)}d\theta\right\}\right\},\nonumber\\
\end{eqnarray}
For a given set $S = \{a_1+1, a_2+1,\cdots, a_{k-1}+1\}$ with $0 = a_0<a_1<\cdots <a_{k-1}<a_k = n$, we first solve the minimization over $q^{(w)}$ and $q^{(\theta)}$. The solutions without constraint that $Q^{(\theta)}\in\mathcal{S}_{\rm MC}$ are given by
\[\wh q^{(w)}(w) = \frac{\Gamma(n-1+\alpha_0+\beta_0)}{\Gamma(k-1+\alpha_0)\Gamma(n-k+\beta_0)}w^{k+\alpha_0-2}(1-w)^{n-k+\beta_0-1},\]
and 
\begin{eqnarray*}
\wh q^{(\theta)}(\theta) &=& \frac{g(\theta_1)\prod_{i\in S}g(\theta_i)\prod_{i\not\in S}\delta_{\theta_{i-1}}(\theta_i)\exp\left(-\frac{1}{2}\sum_{i=1}^n(Y_i-\theta_i)^2\right)}
{\int g(\theta_1)\prod_{i\in S}g(\theta_i)\prod_{i\not\in S}\delta_{\theta_{i-1}}(\theta_i)\exp\left(-\frac{1}{2}\sum_{i=1}^n(Y_i-\theta_i)^2\right)d\theta}\\
&=&\prod_{j = 1}^{k}\frac{g(\theta_{a_{j-1}+1})\exp\left(-\frac{1}{2}\sum_{i = a_{j-1}+1}^{a_j}(Y_i-\theta_{a_{j-1}+1})^2\right)}{\int g(\theta_{a_{j-1}+1})\exp\left(-\frac{1}{2}\sum_{i = a_{j-1}+1}^{a_j}(Y_i-\theta_{a_{j-1}+1})^2\right)d\theta_{a_{j-1}+1}}\prod_{i\not\in S}\delta_{\theta_{i-1}}(\theta_i).
\end{eqnarray*}
As $\wh Q^{(\theta)}$ obtained above is still in the variational set $\mathcal{S}_{\rm MC}$, this is a valid solution to (\ref{eq:KL_joint}) for a specific set $S$, which implies that 
\[\wh Q^{(w)} = {\rm Beta}(k-1+\alpha_0,n-k+\beta_0),\]
and
$$\begin{cases}
d\wh{Q}_1^{(\theta)}(\theta_1)\propto g(\theta_1)\exp\left(-\frac{1}{2}\sum_{i\in(a_0:{a}_1]}(X_i-\theta_1)^2\right)d\theta_1, \\
d\wh{Q}_i^{(\theta)}(\theta_i|\theta_{i-1}) \propto g(\theta_i)\exp\left(-\frac{1}{2}\sum_{l\in({a}_{j-1}:{a}_j]}(X_l-\theta_i)^2\right)d\theta_i, & i={a}_{j-1}+1,j>1,\\
d\wh{Q}_i^{(\theta)}(\theta_i|\theta_{i-1}) = \delta_{\theta_{i-1}}(\theta_i)d\theta_i, & \text{otherwise.}
\end{cases}$$
Now the only thing is to show that $\widehat k$ and $\widehat a_1,\cdots, \widehat a_{k-1}$ are the solution of (\ref{eq:opt_MC}). Plug $\wh Q^{(w)}$ and $\wh Q^{(\theta)}$ into (\ref{eq:KL_joint}), and then
\begin{eqnarray}\label{eq:Q_w}
\nonumber&&\int \wh q^{(w)}(w)\log\frac{\wh q^{(w)}(w)}{\pi(w)w^{|S|}(1-w)^{n-1-|S|}}dw\\
&=& \log\frac{\Gamma(n-1+\alpha_0+\beta_0)}{\Gamma(k-1+\alpha_0)\Gamma(n-k+\beta_0)}-\log\frac{\Gamma(\alpha_0+\beta_0)}{\Gamma(\alpha_0)\Gamma(\beta_0)}
\end{eqnarray}
and
\begin{eqnarray}\label{eq:Q_theta}
&&\int_{\Theta(S)}q^{(\theta)}(\theta)\log\frac{q^{(\theta)}(\theta)}{g(\theta_1)\prod_{i\in S}g(\theta_i)\prod_{i\not\in S}\delta_{\theta_{i-1}}(\theta_i)\exp\left(-\frac{1}{2}\sum_{i=1}^n(Y_i-\theta_i)^2\right)}d\theta\nonumber\\
\nonumber&=&\int_{\Theta(S)}\wh q_1^{(\theta)}(\theta_1)\prod_{i\in S}\wh q_i^{(\theta)}(\theta_i|\theta_{i-1})\prod_{i\not\in S}\delta_{\theta_{i-1}}(\theta_i)\\
\nonumber&&\times\log\frac{\wh q_1^{(\theta)}(\theta_1)\prod_{i\in S}\wh q_i^{(\theta)}(\theta_i|\theta_{i-1})\prod_{i\not\in S}\delta_{\theta_{i-1}}(\theta_i)}{g(\theta_1)\prod_{i\in S}g(\theta_i)\prod_{i\not\in S}\delta_{\theta_{i-1}}(\theta_i)\exp\left(-\frac{1}{2}\sum_{i=1}^n(Y_i-\theta_i)^2\right)}d\theta\\
\nonumber&=&\sum_{j = 1}^{k}\log\left(\int g(\theta_{a_j+1})\exp\left(-\frac{1}{2}\sum_{i = a_{j-1}+1}^{a_j}(Y_i-\theta_{a_{j-1}+1})^2\right)d\theta_{a_{j-1}+1}\right)\\
\nonumber&=&\sum_{j = 1}^{k}\log\left(\int g(\theta)\exp\left(-\frac{1}{2}\sum_{i = a_{j-1}+1}^{a_{j}}(Y_i-\theta)^2\right)d\theta\right).\\
\end{eqnarray}
Plug (\ref{eq:Q_w}) and (\ref{eq:Q_theta}) into (\ref{eq:KL_joint}), and the optimization problem becomes (\ref{eq:opt_MC}). The proof is complete.  
\end{proof}

\subsection{Proofs of Theorem \ref{thm:mixture} and \ref{thm:mixture-latent}}\label{sec:pf-mixture}

To prove Theorem \ref{thm:mixture}, we first establish an upper bound of $P_{\wt f_0}^n\wh{Q}H^2(P_{\wh{k},\theta^{(\wh{k})}},P_{\wt f_0})$ by applying Theorem \ref{thm:general-MF} for a $\wt f_0$ that is constructed to be close to $f_0$. Then, with a change-of-measure argument, we derive a bound for
$P_{f_0}^n\wh{Q}H^2(P_{\wh{k},\theta^{(\wh{k})}},P_{f_0})$. The construction of the surrogate density function $\wt f_0$ is given by the following lemma.
\begin{lemma}\label{lem:mixture2}
Suppose that the true density $f_0$ satisfies conditions (\ref{eq:B1})-(\ref{eq:B3}). For a constant $H_1> 2\alpha$, we define $\wt f_0(x) = \frac{f_0(x)\mathbf{1}_{E_{\sigma_0}}(x)}{\int_{E_{\sigma_0}}f_0(x)dx}$ with  $E_{\sigma_0} = \{x:f_0(x)\geq\sigma_0^{H_1}\}$. For a constant $\xi_4\leq\min\{\xi_3,p\}$ and a sufficiently small $\sigma_0>0$, there exists a finite mixture $p(x|k_{\sigma_0},\theta_{\sigma_0})$ with $k_{\sigma_0} = O(\sigma_0^{-1}|\log\sigma_0|^{p/\xi_4})$ and $\theta_{\sigma_0} = (\mu_{\sigma_0}, w_{\sigma_0},\sigma_0)$,  such that
\begin{equation}\label{eq:C3 mixture}
D_2\left(P_{\wt f_0}\|P_{k_{\sigma_0},\theta_{\sigma_0}}\right) = O(\sigma_0^{2\alpha}).
\end{equation}
Moreover, (\ref{eq:C3 mixture}) holds for all mixtures $p(x|k_{\sigma_0},(\mu,w,\sigma))$ such that $\sigma\in[\sigma_0,\sigma_0+\sigma_0^{H_1+2\alpha+2}]$, $\|\mu-\mu_{\sigma_0}\|_1\leq\sigma_0^{H_1+2\alpha+2}$ and $w\in\Delta_{k_{\sigma_0}}(w_{\sigma_0}, \sigma_0^{H_1+2\alpha+1})$.
\end{lemma}

With the definition of $\wt f_0$ and its property given by Lemma \ref{lem:mixture2}, we can bound $P_{\wt f_0}^n\wh{Q}H^2(P_{\wh{k},\theta^{(\wh{k})}},P_{\wt f_0})$ by checking the conditions (\ref{eq:C1}), (\ref{eq:C2}) and (\ref{eq:C3*}) in Theorem \ref{thm:general-MF}. This argument is split into the next two lemmas.
\begin{lemma}\label{lem:mixture1}
For the prior $\Pi$ that satisfies conditions (\ref{eq:mixture-condition1}), (\ref{eq:mixture-condition2}) and (\ref{eq:mixture-condition3}), the conditions (\ref{eq:C1}) and (\ref{eq:C2}) hold for $L(P^{(n)}, P_0^{(n)}) = nH^2\left(P, P_0\right)$ and all $\epsilon>n^{\delta}$ with some constant $\delta>-1/2$ with respect to $P_0^{(n)}=P_{\wt f_0}^n$ for any $\sigma_0\rightarrow 0$ and $P^{(n)} = P_{k,\theta^{(k)}}^n$.
\end{lemma}

\begin{lemma}\label{lem:mixture3}
Suppose that the true density $f_0$ satisfies conditions (\ref{eq:B1})-(\ref{eq:B3}), and the prior $\Pi$ satisfies conditions (\ref{eq:mixture-condition4}), (\ref{eq:mixture-condition5}), (\ref{eq:mixture-condition6}) and (\ref{eq:mixture-condition7}). Then the condition (\ref{eq:C3*}) holds for Theorem \ref{thm:mixture} with respect to $P_0^{(n)}=P_{\wt f_0}^n$. Here, the density $\wt f_0$ is defined in Lemma \ref{lem:mixture2} with $\sigma_0$ chosen as $n^{-\frac{1}{2\alpha+1}}(\log n)^{\frac{r}{2\alpha+1}}$ and the rate is $\epsilon_n = n^{-\frac{\alpha}{2\alpha+1}}(\log n)^{\frac{\alpha r}{2\alpha+1}}$ with $r$ given in Theorem \ref{thm:mixture}.
\end{lemma}

We first prove Lemma \ref{lem:mixture2}, and then prove Lemma \ref{lem:mixture1} and Lemma \ref{lem:mixture3}. To facilitate the proof of Lemma \ref{lem:mixture2}, we introduce the following lemma, which is analogous to Theorem 1 in in \cite{kruijer2010adaptive}.

\begin{lemma}\label{lem:mixture-approx1}
Let $f_0$ be a density satisfying conditions (\ref{eq:B1})-(\ref{eq:B3}), and let $K_{\sigma_0}$ denote the convolution operator induced by the kernel $\psi_{\sigma_0}$. Then there exists a density $h_\alpha$ such that for a small enough $\sigma_0>0$,
\[\int \frac{f_0^2}{K_{\sigma_0} h_\alpha} = 1+O(\sigma_0^{2\alpha}).\]
\end{lemma}
\begin{proof}
We set 
$G_{\sigma_0} = \{x:f_0(x)\geq\sigma_0^{H_0}\}$
and
\[A_{\sigma_0} = \{x:|l_j(x)|\leq B\sigma_0^{-j}|\log\sigma_0|^{-j/p}, j-1,\cdots, \floor{\alpha}, |L(x)|\leq B\sigma_0^{-\alpha}|\log\sigma_0|^{-\alpha/p}\}.\]
This is the same definition that appears in Lemma 1 of \cite{kruijer2010adaptive}.
Note that $\int \frac{f_0(x)^2}{K_{\sigma_0} h_\alpha(x)}dx-1\geq 0$, and we only need to derive an upper bound for this integral. We first have the following decomposition
\begin{eqnarray*}
&&\int \frac{f_0(x)^2}{K_{\sigma_0} h_\alpha(x)}dx=\int_{A_{\sigma_0}\cap G_{\sigma_0}}\frac{(f_0(x)-K_{\sigma_0} h_\alpha(x))^2}{K_{\sigma_0} h_\alpha(x)}dx\\
&&+\int_{A_{\sigma_0}^c\cup G_{\sigma_0}^c}\frac{f_0(x)^2}{K_{\sigma_0} h_\alpha(x)}dx+\int_{A_{\sigma_0}^c\cup G_{\sigma_0}^c}(K_{\sigma_0} h_\alpha(x)-f_0(x))dx+\int_{A_{\sigma_0}\cap G_{\sigma_0}}f_0(x)dx.
\end{eqnarray*}
The first and third terms can be bounded by $O(\sigma_0^{2\alpha})$ according to the same argument in the proof of Theorem 1 in \cite{kruijer2010adaptive} when $H_0$ is chosen to be large enough. For the second term, according to Remark 1 in \cite{kruijer2010adaptive}, we have $\frac{f_0(x)}{K_{\sigma_0} h_\alpha(x)}\leq M_0$ with some constant $M_0>0$ for all $x$. Then Lemma 2 in \cite{kruijer2010adaptive} implies
\[\int_{A_{\sigma_0}^c\cup G_{\sigma_0}^c}\frac{f_0(x)^2}{K_{\sigma_0} h_\alpha(x)}dx\leq M_0\int_{A_{\sigma_0}^c\cup G_{\sigma_0}^c}f_0(x)dx = O(\sigma_0^{2\alpha}).\]
The last term can be upper bounded by $1$. Summing up all the terms, we obtain the desired conclusion.
\end{proof}

\begin{proof}[Proof of Lemma \ref{lem:mixture2}]
The proof uses a slightly modified argument in the proof of Lemma 4 in \cite{kruijer2010adaptive}.
First of all, according to Lemma \ref{lem:mixture-approx1}, there exists a  density $h_\alpha$ such that
$\int\frac{f_0(x)^2}{K_{\sigma_0} h_\alpha(x)}dx = 1+O(\sigma_0^{2\alpha})$.
Define $E_{\sigma_0}' = \{x:f_0(x)\geq\sigma_0^{H_2}\}$, where $H_2>H_1$ is chosen to be large enough. Set $\tilde h_\alpha(x) = \frac{h_\alpha(x)\mathbf{1}_{E_{\sigma_0}'}(x)}{\int_{E_{\sigma_0}'} h_\alpha(x)dx}$. 
Define the number $a_{\sigma_0} = C_0|\log\sigma_0|^{1/\xi_4}$, with $\xi_4\leq\min\{\xi_3,p\}$ and some constant $C_0>0$. We choose $p(x|k_{\sigma_0},\theta_{\sigma_0})$ with $\theta_{\sigma_0} = (\mu_{\sigma_0}, w_{\sigma_0},\sigma_0)$ to be the finite mixture given by Lemma 12 in \cite{kruijer2010adaptive} that satisfies
\[\|K_{\sigma_0}\tilde h_\alpha-p_{k_{\sigma_0},\theta_{\sigma_0}}\|_\infty\leq\sigma_0^{-1}\exp(-C_0|\log\sigma_0|^{p/\xi_4}),\]
 for $x\in[-a_{\sigma_0}, a_{\sigma_0}]$, where $p_{k_{\sigma_0},\theta_{\sigma_0}}$ is the density of $P_{k_{\sigma_0},\theta_{\sigma_0}}$. We will show that this mixture density satisfies (\ref{eq:C3 mixture}).
We write
\[D_2(P_{\wt f_0}\|P_{k_{\sigma_0}, \theta_{\sigma_0}}) = \int\frac{\wt f_0(x)^2}{p_{k_{\sigma_0},\theta_{\sigma_0}}(x)}dx = \int_{E_{\sigma_0}}\frac{\wt f_0^2}{f_0^2}\frac{ f_0^2}{K_{\sigma_0} h_\alpha}\frac{K_{\sigma_0} h_\alpha}{K_{\sigma_0}\tilde h_\alpha}\frac{K_{\sigma_0}\tilde h_\alpha}{p_{k_{\sigma_0},\theta_{\sigma_0}}}. \]
The four ratios will be bounded separately.
\begin{enumerate}
\item According to (\ref{eq:B2}), we know that $\int f_0(x)^b dx = O(1)$, for any constant $b>0$. Since $H_1>2\alpha$,
\[\int_{E_{\sigma_0}^c}f_0(x)dx\leq (\sigma_0^{H_1})^{\frac{2\alpha}{H_1}}\int_{E_{\sigma_0}^c}f_0(x)^{1-\frac{2\alpha}{H_1}}dx = O(\sigma_0^{2\alpha}).\]
This leads to
\[\left|\frac{\wt f_0^2(x)}{f_0^2(x)}-1\right| = \left|\frac{1}{(1-\int_{E_{\sigma_0}^c}f_0(x)dx)^2}-1\right| \leq C_1\sigma_0^{2\alpha},\]
for a constant $C_1>0$ and all $x\in E_{\sigma_0}$.
\item For the second term, we have
\[\int_{E_{\sigma_0}}\frac{ f_0^2}{K_{\sigma_0} h_\alpha}dx = \int\frac{ f_0^2}{K_{\sigma_0} h_\alpha}dx-\int_{E_{\sigma_0}^c}\frac{f_0^2}{K_{\sigma_0} h_\alpha}dx .\]
Since $\frac{f_0(x)}{K_{\sigma_0} h_\alpha(x)}\leq M_0$ for a constant $M_0$ uniformly over $x$, 
\[\int_{E_{\sigma_0}^c}\frac{f_0^2}{K_{\sigma_0} h_\alpha}dx\leq M_0\int_{E_{\sigma_0}^c} f_0(x)dx = O(\sigma_0^{2\alpha}).\]
Combining with Lemma \ref{lem:mixture-approx1}, we conclude that
\[\left|\int_{E_{\sigma_0}}\frac{ f_0^2}{K_{\sigma_0} h_\alpha}dx - 1\right|\leq C_2\sigma_0^{2\alpha},\]
for a constant $C_2>0$.
\item By the same argument in the proof of Lemma 4 in \cite{kruijer2010adaptive}, we get
\[\left|\frac{K_{\sigma_0} h_\alpha(x)}{K_{\sigma_0}\tilde h_\alpha(x)}- 1\right|\leq C_3\sigma_0^{2\alpha}, \]
for a constant $C_3>0$ and all $x\in E_{\sigma_0}$.
\item According to the proof of Lemma 4 in \cite{kruijer2010adaptive}, we have $E_{\sigma_0}'\subset\{x: f_0(x)\geq c_0\sigma_0^{H_2}\}$ for some constant $c_0$. Because $\xi_4\leq\xi_3$, $E_{\sigma_0}'\subset[-a_{\sigma_0}, a_{\sigma_0}]$. This leads to the inequality $\|K_{\sigma_0}\tilde h_\alpha-p_{k_{\sigma_0},\theta_{\sigma_0}}\|_\infty\leq\sigma_0^{-1}\exp(-C_0|\log\sigma_0|^{p/\xi_4})$. Note that for any $x\in E_{\sigma_0}$, we have $K_{\sigma_0}(x)\tilde h_\alpha(x)\gtrsim K_{\sigma_0} h_\alpha(x)\gtrsim f_0(x)\gtrsim\sigma_0^{H_1}$ uniformly over $x\in E_{\sigma_0}$. Thus, for a sufficiently large $C_0$,
$$\sigma_0^{-1}\exp(-C_0|\log\sigma_0|^{p/\xi_4}) = \sigma^{C_0|\log\sigma_0|^{(p-\xi_4)/\xi_4}-1} = O(\sigma_0^{H_1+2\alpha}),$$
where we have used the condition $\xi_4\leq p$. Then we have
\[\left|\frac{K_{\sigma_0}\tilde h_\alpha(x)}{p(x|k_{\sigma_0},\theta_{\sigma_0})}-1\right| \leq \frac{\|K_{\sigma_0}\tilde h_\alpha-p_{k_{\sigma_0},\theta_{\sigma_0}}\|_\infty}{K_{\sigma_0}\tilde h_\alpha(x)-\|K_{\sigma_0}\tilde h_\alpha-p_{k_{\sigma_0},\theta_{\sigma_0}}\|_\infty}\leq C_4\sigma_0^{2\alpha}.
\]
for all $x\in E_{\sigma_0}$ with some constant $C_4>0$.
\end{enumerate}
Combining the bounds of all terms above, we get
\[\int\frac{\wt f_0^2(x)}{p(x|k_{\sigma_0},\theta_{\sigma_0})}dx = 1+O(\sigma_0^{2\alpha}),\]
which indicates that (\ref{eq:C3 mixture}) holds. 
When $\sigma\in[\sigma_0, \sigma_0^{H_1+2\alpha+2}]$, $\|\mu-\mu_{\sigma_0}\|_1\leq\sigma_0^{H_1+2\alpha+2}$ and $w\in\Delta_{k_{\sigma_0}}(w_{\sigma_0}, \sigma_0^{H_1+2\alpha+1})$, according to Lemma 3 in \cite{kruijer2010adaptive}, we have
\[\|p_{k_{\sigma_0},(\mu,w,\sigma)}-p_{k_{\sigma_0},(\mu_{\sigma_0}, w_{\sigma_0},\sigma_0)}\|_\infty = O(\sigma_0^{H_1+2\alpha}).\]
Then the four points listed above also hold, which means that (\ref{eq:C3 mixture}) is also satisfied for these $(k_{\sigma_0}, (\mu, w, \sigma))$. 
The proof is complete.
\end{proof}

Now we prove Lemma \ref{lem:mixture1} and Lemma \ref{lem:mixture3}.
\begin{proof}[Proof of Lemma \ref{lem:mixture1}]
We consider the set
\begin{equation}\label{eq:outlier-part}
\Theta_n(\epsilon) = \cup_{k=1}^{k_n}\Theta^{(k)}(\epsilon),
\end{equation}
where
\[\Theta^{(k)}(\epsilon) = \left\{P_{k,\theta^{(k)}}:\theta^{(k)}=(\mu,w,\sigma), \mu\in\otimes_{j=1}^k[-b_n, b_n],\sigma\in (m_\sigma, M_\sigma]\right\},\]
with $k_n = \left\lceil\frac{n\epsilon^2}{\log(n\epsilon^2)}\right\rceil$, $b_n = (n\epsilon^2)^{\frac{1}{c_3}}$, $m_\sigma = (n\epsilon^2)^{-\frac{1}{2b_3}}$ and $M_\sigma = \exp\left(\frac{1}{2}n\epsilon^2\right)$.
It's easy to see that
\[\Theta_n(\epsilon)^c\subseteq\left\{P_{k,\theta^{(k)}}:k>k_n\right\}\bigcup\left(\cup_{k=1}^{k_n}\wt\Theta^{(k)}(\epsilon)\right),\]
where $\wt\Theta^{(k)}(\epsilon) = \left\{P_{k,\theta^{(k)}}:\theta^{(k)} = (\mu,w,\sigma), \max_{j}|\mu_j|>B_n\text{ or }\sigma\not\in(m_\sigma, M_\sigma]\right\}$. Thus
\begin{eqnarray*}
\Pi(\Theta_n(\epsilon)^c)&\leq&\sum_{k>k_n}\pi(k)+\sum_{k = 1}^{k_n}\pi(k)\left[\Pi^{(k)}\left(\sigma\not\in(m_\sigma,M_{\sigma}]\right)+\Pi^{(k)}\left(\max_{1\leq j\leq k}|\mu_j|>b_n\right)\right]
\end{eqnarray*}
Now we derive an upper bound for each term.
\begin{enumerate}
\item Set $\tau = \sigma^{-2}$, and then for all $k\in\mathbb{N}_+$,
\begin{eqnarray*}
\Pi^{(k)}(\sigma\not\in(m_\sigma, M_\sigma]) &\leq& \int_0^{\exp\left(-n\epsilon^2\right)}p_\tau(\tau)d\tau + \int_{(n\epsilon^2)^{1/b_3}}^\infty p_\tau(\tau)d\tau\\
&\leq& b_0\exp(-n\epsilon^2)+b_1\exp(-b_2n\epsilon^2),
\end{eqnarray*}
where we have used the condition (\ref{eq:mixture-condition3}).
\item By the condition (\ref{eq:mixture-condition1}), we have
\[\sum_{k>k_n}\pi(k)\leq C_1\exp(-C_2k_n\log(k_n))\leq C_1\exp(-\wt C_2n\epsilon^2).\]
\item According to the conditions (\ref{eq:mixture-condition1}) and (\ref{eq:mixture-condition2}),
\begin{eqnarray*}
&&\sum_{k=1}^{k_n}\pi(k)\Pi^{(k)}(\max_j|\mu_j|>b_n)\leq \sum_{k=1}^{\infty}\pi(k)\Pi^{(k)}(\max_j|\mu_j|>b_n)\\
&\leq&\sum_{k=1}^\infty\pi(k)k\left(\int_{-\infty}^{-b_n}p_\mu(x)dx+\int_{b_n}^\infty p_\mu(x)dx\right)\\
&\leq&c_1\exp(-c_2n\epsilon^2)\sum_{m=1}^\infty\sum_{k=m}^\infty\pi(k)\\
&\leq&c_1\exp(-c_2n\epsilon^2)\sum_{m=1}^\infty C_1\exp(-C_2m\log m)\\
&\leq&\wt c_1\exp(-c_2n\epsilon^2).
\end{eqnarray*}
\end{enumerate}
Summing up the three bounds above, we have $\Pi(\Theta_n(\epsilon)^c)\lesssim\exp(-C_0n\epsilon^2)$
for some constant $C_0>0$. In order that the constant $C_0$ can be arbitrarily large, one can replace $\epsilon$ by $\wt{C}\epsilon$ for a sufficiently large $\wt{C}$ and use the same argument above. We therefore obtain (\ref{eq:C2}).

Now we start to show (\ref{eq:C1}). By Theorem 7.1 in \cite{ghosal2000convergence}, it is sufficient to bound the metric entropy
$$\log N(\epsilon,\Theta_n(\epsilon),H)\lesssim n\epsilon^2.$$
Since $H^2(P_1, P_2)\leq \TV(P_1, P_2)$, we have $N(\epsilon,\Theta_n(\epsilon),H)\leq N(\epsilon^2,\Theta_n(\epsilon),\TV)$.
According to (\ref{eq:outlier-part}),
\[N(\epsilon^2,\Theta_n(\epsilon), \TV)\leq\sum_{k=1}^{k_n}N(\epsilon^2,\Theta^{(k)}(\epsilon), \TV),\]
and thus it is sufficient to bound $N(\epsilon^2,\Theta^{(k)}(\epsilon), \TV)$ for each $k\in[k_n]$.

We use $\psi$ to denote $\psi_{\sigma}$ with $\sigma = 1$ in short. According to Lemma 3 in \cite{kruijer2010adaptive}, for any $P_{k,\theta}$ with $\theta = (\mu,w,\sigma)$ and $P_{k,\wt\theta}$ with $\wt\theta = (\wt\mu,\wt w,\wt\sigma)$ such that $P_{k,\theta}, P_{k,\wt\theta}\in\Theta^{(k)}(\epsilon)$, we have
\begin{eqnarray*}
\TV(P_{k,\theta},P_{k,\wt\theta})&\leq& \|w-\wt w\|_1+2\|\psi\|_\infty\sum_{i=1}^k\frac{w_i\wedge\wt w_i}{\sigma\wedge\wt\sigma}|\mu_i-\wt\mu_i|+\frac{|\sigma-\wt\sigma|}{\sigma\wedge\wt\sigma}\\
&\leq&\|w-\wt w\|_1+2\frac{\|\psi\|_\infty}{m_\sigma}\|\mu-\wt\mu\|_1+\frac{|\sigma-\wt\sigma|}{m_\sigma}.
\end{eqnarray*}
Based on the fact that $N(\epsilon, A\times B, d_1+d_2)\leq N(t\epsilon, A,d_1)\times N((1-t)\epsilon, B, d_2)$, we have
\begin{eqnarray*}
&&N(\epsilon^2, \Theta^{(k)}(\epsilon),\TV)\\
&\leq& N\left(\frac{\epsilon^2}{3},\Delta_k, \|\cdot\|_1\right)N\left(\frac{m_\sigma\epsilon^2}{6\|\psi\|_\infty}, [-b_n, b_n]^k,\|\cdot\|_1\right)N\left(\frac{m_\sigma\epsilon^2}{3}, (m_\sigma, M_\sigma], |\cdot|\right).
\end{eqnarray*}
Then, we use Lemma 5 in \cite{kruijer2010adaptive}, and obtain
\[N\left(\frac{\epsilon^2}{3},\Delta_k, \|\cdot\|_1\right)\leq \exp\left((k-1)\log\frac{15}{\epsilon^2}\right)\leq\exp(C_1k\log(n\epsilon^2)),\]
\[N\left(\frac{m_\sigma\epsilon^2}{6\|\psi\|_\infty}, [-b_n, b_n]^k,\|\cdot\|_1\right)\leq\frac{k!(b_n+\wt\epsilon)^k}{\wt\epsilon^k}\leq \exp(C_2k(\log k+\log n\epsilon^2)),\]
where $\tilde\epsilon = \frac{m_\sigma\epsilon^2}{6\|\psi\|_\infty}$, and
\[N\left(\frac{m_\sigma\epsilon^2}{3}, (m_\sigma, M_\sigma], |\cdot|\right)\leq\frac{M_\sigma}{m_\sigma\epsilon^2/3}\leq \exp(C_3n\epsilon^2),\]
for some constants $C_1, C_2, C_3>0$. Note that we have used the condition $\epsilon>n^\delta$ for some constant $\delta>-1/2$ to derive the above bounds.
Finally, we have
\[N(\epsilon^2, \Theta_n(\epsilon),\TV)\leq k_n\exp\left(C_1k_n\log n\epsilon^2+C_2k_n(\log k_n+\log n\epsilon^2)+C_3n\epsilon^2\right),\]
which leads to
\[\log N(\epsilon^2, \Theta_n(\epsilon), \TV)\lesssim k_n\log(n\epsilon^2) \lesssim n\epsilon^2.\]
The proof is complete.
\end{proof}

\begin{proof}[Proof of Lemma \ref{lem:mixture3}]
According to Lemma \ref{lem:mixture2}, there exist $k_{\sigma_0}$, $\theta_{\sigma_0} = (\mu_{\sigma_0}, w_{\sigma_0}, \sigma_0)$ such that (\ref{eq:C3 mixture}) holds. Then we set $k_0 = k_{\sigma_0}$ in Theorem \ref{thm:general-MF} and $\Theta^{(k_{\sigma_0})} = \Theta_\mu^{(k_{\sigma_0})}\otimes\Theta_w^{(k_{\sigma_0})}\otimes\Theta_{\sigma}^{(k_{\sigma_0})}$, where $\Theta_\mu^{(k_{\sigma_0})} = \otimes_{j=1}^{k_{\sigma_0}}\Theta_{\mu_j}^{(k_{\sigma_0})}$. 
To be specific, let $H_1$ be any fixed constant such that $H_1>2\alpha$, and then we define
\[\Theta_{\mu_j}^{(k_{\sigma_0})} = [\mu_{\sigma_0, j}-k_{\sigma_0}^{-1}\sigma_0^{H_1+2\alpha+2}, \mu_{\sigma_0, j}+k_{\sigma_0}^{-1}\sigma_0^{H_1+2\alpha+2}],\]
\[\Theta_{w}^{(k_{\sigma_0})} = \Delta_{k_{\sigma_0}}(w_{\sigma_0},\sigma_0^{H_1+2\alpha+1}),\]
and
\[\Theta_{\sigma}^{(k_{\sigma_0})} = [\sigma_0, \sigma_0+\sigma_0^{H_1+2\alpha+2}].\]
The conclusion of Lemma \ref{lem:mixture2} implies
\begin{equation}\label{eq:mixture-part1}
\Theta^{(k_{\sigma_0})}\subset\left\{(\mu,w,\sigma): nD_2\left(P_{\wt f_0}\| P_{k_{\sigma_0},\theta_{\sigma_0}}\right)\leq C_2n\sigma_0^{2\alpha}\right\},
\end{equation}
for a constant $C_2>0$. Choose $\sigma_0 =  n^{-\frac{1}{2\alpha+1}}(\log n)^{\frac{r}{2\alpha+1}}$, then $n^{\frac{1}{2\alpha+1}}\leq k_{\sigma_0}\leq n^{\frac{1}{2\alpha+1}+t}$ for any $t>0$ as $n\rightarrow\infty$.
Then the condition (\ref{eq:mixture-condition4}) implies
\begin{equation}\label{eq:mixture-part2}
-\log\pi(k_{\sigma_0}) \lesssim k_{\sigma_0}\log k_{\sigma_0}.
\end{equation} 
We also have
\[\Pi_{\mu_j}^{(k_{\sigma_0})}(\Theta_{\mu_j}^{(k_{\sigma_0})}) \geq \int_{\mu_{\sigma_0,j}-k_{\sigma_0}^{-1}\sigma_0^{H_1+2\alpha+2}}^{\mu_{\sigma_0,j}+k_{\sigma_0}^{-1}\sigma_0^{H_1+2\alpha+2}}p_\mu(x)dx.\]
According to the condition (\ref{eq:mixture-condition5}) and Lemma \ref{lem:mixture2}, we have $|\mu_j|\lesssim |\log\sigma_0|^{1/\xi_4}$ with $\xi_4\leq\min\{\xi_3, p\}$ as in Lemma \ref{lem:mixture2}. Then,
\[-\log\Pi_{\mu_j}^{(k_{\sigma_0})}(\Theta_{\mu_j}^{(k_{\sigma_0})})\lesssim |\log\sigma_0|+|\log\sigma_0|^{c_6/\xi_4} \lesssim |\log\sigma_0|^{\max\{1,c_6/\xi_4\}}. \]
By (\ref{eq:mixture-condition6}), we have
\[-\log\Pi_w^{(k_{\sigma_0})}(\Theta_{w}^{(k_{\sigma_0})})\lesssim k_{\sigma_0}(\log k_{\sigma_0})^{d_3}|\log\sigma_0|.\]
Finally, the condition (\ref{eq:mixture-condition7}) leads to
\[-\log\Pi_\sigma^{(k_{\sigma_0})}(\Theta_{\sigma}^{(k_{\sigma_0})}) \leq -\log\left(\int_{(\sigma_0+\sigma_0^{H_1+2\alpha+2})^{-1}}^{\sigma_0^{-1}}p_\tau(x)dx\right)\lesssim |\log\sigma_0|+\sigma_0^{-b_6}\lesssim\sigma_0^{-1}.\]
With the choice $\xi_4=\min\{p, \xi_3\}$ and $k_{\sigma_0} = O(\sigma_0^{-1}|\log\sigma_0|^{p/\xi_4})$, we have
\begin{eqnarray*}
-\log\pi(k_{\sigma_0})-\sum_{j=1}^{k_{\sigma_0}}\Pi_{\mu_j}(\Theta_{\mu_j}^{(k_{\sigma_0})})-\log\Pi_w^{(k_{\sigma_0})}(\Theta_w^{(k_{\sigma_0})})-\log\Pi_\sigma(\Theta_{\sigma}^{(k_{\sigma_0})})\leq C_3 \sigma_0^{-1}(\log\sigma_0)^r.
\end{eqnarray*}
where $r = \frac{p}{\min\{p,\xi_3\}}+\max\{d_3+1, \frac{c_6}{\min\{p,\xi_3\}}\}$. Plug in $\sigma_0= n^{-\frac{1}{2\alpha+1}}(\log n)^{\frac{r}{2\alpha+1}}$,we obtain (\ref{eq:C3*}) with respect to $\wt f_0$.
\end{proof}

Finally we prove Theorem \ref{thm:mixture}. 

\begin{proof}[Proof of Theorem \ref{thm:mixture}]
We bound $P_{f_0}^n\wh{Q}H^2(P_{\wh k,\theta^{(\wh k)}},P_{f_0})$ by
\begin{eqnarray*}
P_{f_0}^n\wh{Q}H^2(P_{\wh k,\theta^{(\wh k)}},P_{f_0}) &\leq& P_{\wt f_0}^n\wh{Q}H^2(P_{\wh k,\theta^{(\wh k)}},P_{f_0})+\TV(P_{f_0}^n,P_{\wt f_0}^n)\\
&\leq& 2P_{\wt f_0}^n\wh{Q}H^2(P_{\wh k,\theta^{(\wh k)}},P_{\wt f_0})+2H^2(P_{\wt f_0},P_{f_0})+\TV(P_{f_0}^n,P_{\wt f_0}^n).
\end{eqnarray*}
By Lemma \ref{lem:mixture1}, Lemma \ref{lem:mixture3} and Theorem \ref{thm:general-MF}, we have
\[P_{\wt f_0}^n\wh{Q}H^2(P_{\wh k, \theta^{(\wh k)}},P_{\wt f_0})\lesssim \sigma_0^{2\alpha},\]
for $\sigma_0= n^{-\frac{1}{2\alpha+1}}(\log n)^{\frac{r}{2\alpha+1}}$. 
Note that $\wt f_0(x) = \frac{f_0(x)\mathbf{1}_{E_{\sigma_0}}(x)}{\int_{E_{\sigma_0}}f_0(x)dx}$ with $E_{\sigma_0} = \{x: f_0(x)\geq \sigma_0^{H_1}\}$, and $R = \int_{E_{\sigma_0}^c}f_0(x)dx \leq \sigma_0^{H_1/2}\int_{E_{\sigma_0}^{c}}\sqrt{f_0(x)}dx =  O(\sigma_0^{H_1/2})$. Then,
\[H^2(P_{\wt f_0},P_{f_0}) = 1-\int\sqrt{f_0(x)\wt f_0(x)}dx = 1-\sqrt{1-R} = O(\sigma_0^{H_1/2}).\]
Moreover,
\[\TV(P_{f_0}^n,P_{\wt f_0}^n) = 1-(1-R)^n = O(nR) = O(n\sigma_0^{H_1/2}).\]
With the choice $H_1 = 8\alpha+4$, the proof is complete.
\end{proof}

Now we prove Theorem \ref{thm:mixture-latent}. We also use change of measure argument and show the concentration around $P_{\wt f_0}^{(n)}$ at first.

\begin{proof}[Proof of Theorem \ref{thm:mixture-latent}]
We first present a latent variable version of Lemma \ref{lem:general1_model_selection},
\begin{eqnarray}
\nonumber &&P_{\wt f_0}^n\left[\wh Q^{(\wh k)}nH\left(P_{\wh k,\theta^{(\wh k)}}, P_{\wt f_0}\right)^2\right] \\
\nonumber &\leq& \inf_{a>0}\frac{1}{a}\left[\min_{k\in\mathcal{K}}\min_{\bar Q^{(k)}\in\mathcal{\bar S}_{\rm MF}^{(k)}}\left\{D\left(\bar Q^{(k)}\|\bar\Pi^{(k)}\right) \right.\right.\\
\nonumber && \left.+\bar Q^{(k)}D\left(P_{\wt f_0}^n\|P(\cdot|k,z^{(k)},\theta^{(k)})\right)-\log\pi(k)\right\} \\
\label{eq:upper_all} &&\left.+P_{\wt f_0}^n\log \Pi\left(\exp\left(anH\left(P_{k,\theta^{(k)}}, P_{\wt f_0}\right)^2\right)\Big|X^{(n)}\right)\right],
\end{eqnarray}
where $\wt f_0$ is defined in the Lemma \ref{lem:mixture1}. The proof of this inequality follows the same argument as in the proof of Lemma \ref{lem:general1_model_selection} and thus we omit it.
Note that the parametrization of the density $p(x|k,\theta^{(k)})$ in (\ref{eq:mixture}) does not rely on the latent variables. Therefore, when the conditions of Theorem \ref{thm:mixture} are satisfied, for some small constant $a>0$, we have
\[P_{\wt f_0}^n\log \Pi\left(\exp\left(anH\left(P_{k,\theta^{(k)}}, P_{\wt f_0}\right)^2\right)\Big|X^{(n)}\right)\lesssim n^{\frac{1}{2\alpha+1}}(\log n)^{\frac{2\alpha r}{2\alpha+1}},\]
based on the Jensen's Inequality, Lemma \ref{lem:strong convergence} and Lemma \ref{lem:sub-exp}.

Now we need to choose some $k\in\mathcal{K}$ and $\bar{Q}^{(k)}\in\mathcal{S}_{\rm MF}^{(k)}$ to bound the remaining terms of (\ref{eq:upper_all}). We consider $d\bar Q^{(k)}(z^{(k)},\theta^{(k)}) = dQ_z^{(k)}(z^{(k)})dQ_{\theta}^{(k)}(\theta^{(k)})$ and $Q_{z}^{(k)}(z_i^{(k)}=j) = \gamma_{ij}$, where $\sum_{j=1}^k\gamma_{ij} = 1$. We sometimes shorthand $Q_{\theta}^{(k)}$ by $Q^{(k)}$ when the context is clear.
Write $z_{ij}^{(k)} = \1_{\{z_i^{(k)} = j\}}$, and we have
\begin{eqnarray*}
&&D\left(\bar Q^{(k)}\|\bar\Pi^{(k)}\right)+\bar Q^{(k)}D\left(P_{\wt f_0}^{n}\|P(\cdot|k,z^{(k)},\theta^{(k)})\right)-\log\pi(k)\\
&=&P_{\wt f_0}^{n}\sum_{z^{(k)}} \prod_{i=1}^n\prod_{j=1}^k\gamma_{ij}^{z_{ij}^{(k)}}\int\log\frac{\prod_{i=1}^n\prod_{j=1}^k\gamma_{ij}^{z_{ij}^{(k)}}dQ_\theta^{(k)}(\theta^{(k)})}{d\Pi^{(k)}(\theta^{(k)})\prod_{i=1}^n\prod_{j = 1}^k\left(w_{j}\psi_{\sigma}(X_i-\mu_j)\right)^{z_{ij}^{(k)}}}dQ_\theta^{(k)}(\theta^{(k)})\\
&&+P_{\wt f_0}^{n}\log p_{\wt f_0}^{n}(X^{(n)})-\log\pi(k)\\
&=&\sum_{i=1}^n\sum_{j=1}^k\gamma_{ij}\log\gamma_{ij}+D\left(Q^{(k)}\|\Pi^{(k)}\right)+P_{\wt f_0}^{n}\log p_{\wt f_0}^{n}(X^{(n)})-\log\pi(k)\\
&&-\sum_{z^{(k)}}\sum_{i,j}z_{ij}^{(k)}\prod_{i=1}^n\prod_{j=1}^k\gamma_{ij}^{z_{ij}^{(k)}}\int \log w_j\psi_{\sigma}(X_i-\mu_j)dQ^{(k)}(\theta^{(k)})\\
&=&\sum_{i=1}^n\sum_{j=1}^k\gamma_{ij}\log\frac{\gamma_{ij}}{\exp\left(\int \log w_j\psi_{\sigma}(X_i-\mu_j)dQ^{(k)}(\theta^{(k)}\right))}\\
&&+D\left(Q^{(k)}\|\Pi^{(k)}\right)+P_{\wt f_0}^{n}\log p_{\wt f_0}^{n}(X^{(n)})-\log\pi(k).
\end{eqnarray*}
Thus, the optimal choice of $\gamma_{ij}$ is that
\[\gamma_{ij} = \frac{\exp\left(\int \log w_j\psi_{\sigma}(X_i-\mu_j)dQ^{(k)}(\theta^{(k)})\right)}
{\sum_{r = 1}^k\exp\left(\int \log w_r\psi_{\sigma}(X_i-\mu_r)dQ^{(k)}(\theta^{(k)})\right)},\]
and we fix this choice as our $Q_z^{(k)}$.
We then have
\begin{eqnarray}\label{eq:upper1}
&&D\left(\bar Q^{(k)}\|\bar \Pi^{(k)}\right)+\bar{Q}^{(k)}D\left(P_{\wt f_0}^{n}\|P(\cdot|k,z^{(k)},\theta^{(k)})\right)-\log\pi(k)\\
&=&-\sum_{i=1}^n\log\left[\sum_{r = 1}^k\exp\left(\int\log w_r\psi_{\sigma}(X_i-\mu_r)dQ^{(k)}(\theta^{(k)})\right)\right]\nonumber\\
&&+D\left(Q^{(k)}\|\Pi^{(k)}\right)+P_{\wt f_0}^{n}\log p_{\wt f_0}^{n}(X^{(n)})-\log\pi(k).\nonumber
\end{eqnarray}
We now specify the choice of $k\in\mathcal{K}$ and $Q^{(k)}=Q_{\theta}^{(k)}$ in (\ref{eq:upper1}). According to Lemma \ref{lem:mixture2}, for $k  = k_{\sigma_0} = O(\sigma_0^{-1}|\log\sigma_0|^{p/\xi_4})$, when $\theta^{(k)} = (\mu, w,\sigma)$ such that
\[\sigma \in [\sigma_0, \sigma_0+\sigma_0^{H_1+2\alpha+2}],\qquad\|\mu-\mu_{\sigma_0}\|_1\leq\sigma_0^{H_1+2\alpha+2},\qquad w\in\Delta_{k_{\sigma_0}}(w_{\sigma_0}, \sigma_0^{H_1+2\alpha+1}),\]
we have
\begin{equation}\label{eq:D2}
D_2\left(P_{\wt f_0}\|P_{k,\theta^{(k)}}\right)\lesssim\sigma_0^{2\alpha}.
\end{equation}
Suppose $i_0 = \argmax_i w_{\sigma_0, i}$, then $w_{\sigma_0, i_0}\geq k_{\sigma_0}^{-1}$ and $w_{\sigma_0, i_0}-\frac{k_{\sigma_0}-1}{2k_{\sigma_0}}\sigma_0^{H_1+2\alpha+1}>\frac{1}{2k_{\sigma_0}}\sigma_0^{H_1+2\alpha+1}$ when $\sigma_0\rightarrow 0$. Then consider $w_{\sigma_0,j}^* = w_{\sigma_0, j}+\1_{\{j\neq i_0\}}\frac{1}{2k_{\sigma_0}}\sigma_0^{H_1+2\alpha+1}-\1_{\{j = i_0\}}\frac{k_{\sigma_0}-1}{2k_{\sigma_0}}\sigma_0^{H_1+2\alpha+1}$. Obviously, $w_{\sigma}^*\in\Delta_{k_{\sigma_0}}(w_{\sigma_0},\sigma_0^{H_1+2\alpha+1})$ and $w_{\sigma_0, i}^*\geq\frac{1}{2k_{\sigma_0}}\sigma_0^{H_1+2\alpha+1}$ for all $1\leq i\leq k_{\sigma_0}$.
Set 
\[\wt\Theta_{\mu_j}^{(k_{\sigma_0})} = [\mu_{\sigma_0,j}-k_{\sigma_0}^{-1}\sigma_0^{H_1+2\alpha+2}, \mu_{\sigma_0,j}+k_{\sigma_0}^{-1}\sigma_0^{H_1+2\alpha+2}],\]
\[\wt\Theta_{w}^{(k_{\sigma_0})} = \Delta_{k_{\sigma_0}}\left(w_{\sigma_0}^*,\frac{\nu}{2k_{\sigma_0}}\sigma_0^{H_1+2\alpha+1}\right).\]
\[\wt\Theta_{\sigma}^{(k_{\sigma_0})}= \left[\wt\sigma_0,\wt\sigma_0+\frac{1}{2}\wt\sigma_0^{H_1+2\alpha+2}\right],\]
where $\wt\sigma_0 = (1+\epsilon)\sigma_0$ with $\epsilon>0$ and $0<\nu<1$ to be determined later. Choose $k = k_{\sigma_0}$ and $dQ^{(k_{\sigma_0})}(\theta^{(k_{\sigma_0})})=dQ_w^{(k_{\sigma_0})}(w)dQ_\tau^{(k_{\sigma_0})}(\tau)\prod_{j=1}^{k_{\sigma_0}}dQ_{\mu_j}^{(k_{\sigma_0})}(\mu_j)$ in (\ref{eq:upper1}), where
\[dQ_w^{(k_{\sigma_0})}(w) =  \frac{p_w^{(k_{\sigma_0})}(w)\1_{\{w\in\wt\Theta_w^{(k_{\sigma_0})}\}}dw}{\int_{\wt\Theta_w^{(k_{\sigma_0})}}p_w^{(k_{\sigma_0})}(w)dw},\] 
\[dQ_{\mu_j}(\mu_j) = \frac{p_{\mu}(\mu_j)\1_{\{\mu_j\in\wt\Theta_{\mu_j}^{(k_{\sigma_0})}\}}d\mu_j}{\int_{\wt\Theta_{\mu_j}^{(k_{\sigma_0})}} p_{\mu}(\mu_j)d\mu_j},\qquad\text{for $1\leq j\leq k_{\sigma_0}$,}\]
\[dQ_\tau(\tau) =  \frac{p_\tau(\tau)\1_{\{\tau^{-1/2}\in\wt\Theta_{\sigma}^{(k_{\sigma_0})}\}}d\tau}{\int_{\tau^{-1/2}\in\wt\Theta_{\sigma}^{(k_{\sigma_0})}}p_\tau(\tau)d\tau}.\]
Now, we build an upper bound for 
\begin{equation}\label{eq:log-exp-term}
-\sum_{i=1}^n\log\left[\sum_{r = 1}^{k_{\sigma_0}}\exp\left(\int\log w_r\psi_{\sigma}(X_i-\mu_r)dQ^{(k_{\sigma_0})}(\theta^{(k_{\sigma_0})})\right)\right],
\end{equation}
 or equivalently, construct lower bounds for $\int\log w_r\psi_{\sigma}(X_i-\mu_r)dQ^{(k_{\sigma_0})}(\theta^{(k_{\sigma_0})})$ for all $1\leq i\leq n$ and $1\leq r\leq k_{\sigma_0}$.
Set $\tau_{\min}^{1/2} = (\wt\sigma_0+\frac{1}{2}\wt\sigma_0^{H_1+2\alpha+2})^{-1}$ and $\tau_{\max}^{1/2} = \wt\sigma_0^{-1}$, and then for any $w\in\wt\Theta_w^{(k_{\sigma_0})}$, $\mu_r\in\wt\Theta_{\mu_r}^{(k_{\sigma_0})}$, $\tau^{-1/2} = \sigma\in\wt\Theta_{\sigma_0}^{(k_{\sigma_0})}$, we have
\[w_r\geq w_{\sigma_0, r}^*-\frac{\nu}{2k_{\sigma_0}}\sigma_0^{H_1+2\alpha+1}\geq (1-\nu)w_{\sigma_0,r}^*,\]
and
\begin{eqnarray*}
\psi_{\sigma}(X_i-\mu_r) &=&\frac{\tau^{1/2}}{2\Gamma\left(1+1/p\right)}\exp\left(-\tau^{p/2}|X_i-\mu_r|^p\right)\\
&\geq&\frac{\tau_{\min}^{1/2}}{2\Gamma(1+1/p)}\exp\left(-\tau_{\max}^{p/2}\left[|X_i-\mu_{\sigma_0, r}|+{k_{\sigma_0}^{-1}}\sigma_0^{H_1+2\alpha+2}\right]^p\right).
\end{eqnarray*}
Now we build the upper bounds of $\left[|X_i-\mu_{\sigma_0, r}|+{k_{\sigma_0}^{-1}}\sigma_0^{H_1+2\alpha+2}\right]^p$ in two cases:
\begin{itemize}
\item If $k_{\sigma_0}^{-1}\sigma_0^{H_1+2\alpha+2}\leq\epsilon|X_i-\mu_{\sigma_0, r}|$, then
\[\left[|X_i-\mu_{\sigma_0, r}|+{k_{\sigma_0}^{-1}}\sigma_0^{H_1+2\alpha+2}\right]^p\leq(1+\epsilon)^p|X_i-\mu_{\sigma_0, r}|^p.\]
\item If $k_{\sigma_0}^{-1}\sigma_0^{H_1+2\alpha+2}>\epsilon|X_i-\mu_{\sigma_0, r}|$, then
\[\left[|X_i-\mu_{\sigma_0, r}|+{k_{\sigma_0}^{-1}}\sigma_0^{H_1+2\alpha+2}\right]^p\leq\left((1+\epsilon^{-1})k_{\sigma_0}^{-1}\sigma_0^{H_1+2\alpha+2}\right)^p.\]
\end{itemize}
Therefore, 
\[\left[|X_i-\mu_{\sigma_0, r}|+{k_{\sigma_0}^{-1}}\sigma_0^{H_1+2\alpha+2}\right]^p\leq(1+\epsilon)^p|X_i-\mu_{\sigma_0, r}|^p+\left((1+\epsilon^{-1})k_{\sigma_0}^{-1}\sigma_0^{H_1+2\alpha+2}\right)^p,\]
and then
\begin{eqnarray*}
&&\psi_{\sigma}(X_i-\mu_r)\\
&\geq&\exp\left(-\left(\tau_{\max}^{1/2}(1+\epsilon^{-1})k_{\sigma_0}^{-1}\sigma_0^{H_1+2\alpha+2}\right)^p\right)\\
&&\times\frac{\tau_{\min}^{1/2}}{2\Gamma(1+1/p)}\exp\left(-((1+\epsilon)^2\tau_{\max})^{p/2}|X_i-\mu_{\sigma_0, r}|^p\right)\\
&=&\frac{\tau_{\min}^{1/2}}{(1+\epsilon)\tau_{\max}^{1/2}}\exp\left(-\left(\tau_{\max}^{1/2}(1+\epsilon^{-1})k_{\sigma_0}^{-1}\sigma_0^{H_1+2\alpha+2}\right)^p\right)\psi_{\sigma_0}(|X_i-\mu_{k_{\sigma_0},r}|).
\end{eqnarray*}
where the last step we apply the fact that $(1+\epsilon)^{-1}\tau_{\max}^{-1/2} = (1+\epsilon)^{-1}\wt\sigma_0 = \sigma_0$. Thus, we have
\begin{eqnarray*}
&&\int\log w_r\psi_{\sigma}(X_i-\mu_r)dQ^{(k_{\sigma_0})}(\theta^{(k_{\sigma_0})})\\
&\geq&\log\left[(1-\nu)\frac{\tau_{\min}^{1/2}}{(1+\epsilon)\tau_{\max}^{1/2}}\exp\left(-\left(\tau_{\max}^{1/2}(1+\epsilon^{-1})k_{\sigma_0}^{-1}\sigma_0^{H_1+2\alpha+2}\right)^p\right)\right]+\log\psi_{\sigma_0}(|X_i-\mu_r|),
\end{eqnarray*}
so 
\begin{eqnarray*}
&&-\sum_{i=1}^n\log\left[\sum_{r = 1}^{k_{\sigma_0}}\exp\left(\int \log w_r\psi_{\sigma}(X_i-\mu_r)dQ^{(k_{\sigma_0})}(\theta^{({k_{\sigma_0}})})\right)\right]\\
&\leq&-n\log\left[(1-\nu)\frac{\tau_{\min}^{1/2}}{(1+\epsilon)\tau_{\max}^{1/2}}\exp\left(-\left(\tau_{\max}^{1/2}(1+\epsilon^{-1})k_{\sigma_0}^{-1}\sigma_0^{H_1+2\alpha+2}\right)^p\right)\right]\\
&&-\log p_{k_{\sigma_0},\theta^{*(k_{\sigma_0})}}^n(X^{(n)})
\end{eqnarray*}
where $\theta^{*(k_{\sigma_0})} = (\mu_{\sigma_0}, w_{\sigma_0}^*,\sigma_0)$. 

We plug the above upper bound into (\ref{eq:upper1}), and then
\begin{eqnarray*}
&&D\left(\bar Q^{(k_{\sigma_0})}\|\bar\Pi^{(k_{\sigma_0})}\right)+\bar Q^{(k_{\sigma_0})}D\left(P_{\wt f_0}^{n}\|P(\cdot|k_{\sigma_0},z^{(k_{\sigma_0})},\theta^{(k_{\sigma_0})})\right)-\log\pi(k_{\sigma_0})\\
&\leq&n\left[\log\frac{1+\epsilon}{1-\nu}+\frac{1}{2}\log\frac{(1+\epsilon)\tau_{\max}}{\tau_{\min}}+\left(\tau_{\max}^{1/2}(1+\epsilon^{-1}){k_{\sigma_0}^{-1}}\sigma_0^{H_1+2\alpha+2}\right)^p\right]\\
&&+D\left(P_{\wt f_0}^{n}\|P_{k_{\sigma_0}, \theta^{*(k_{\sigma_0})}}^{n}\right)+D\left(Q^{(k_{\sigma_0})}\|\Pi^{(k_{\sigma_0})}\right).
\end{eqnarray*}
Now we build the upper bound for each term in the right hand side above. For the first term, when $\nu\leq 1/2$, $\frac{1}{1-\nu}\leq 1+2\nu$, we have
\begin{eqnarray*}
&&\log\frac{1+\epsilon}{1-\nu}+\frac{1}{2}\log\frac{(1+\epsilon)\tau_{\max}}{\tau_{\min}}+\left(\tau_{\max}^{1/2}(1+\epsilon^{-1})k_{\sigma_0}^{-1}\sigma_0^{H_1+2\alpha+2}\right)^p\\
&\leq&
\log(1+\epsilon)(1+2\nu)+\frac{1}{2}\log(1+\epsilon)(1+\wt\sigma_0^{H_1+2\alpha+1})+\epsilon^{-p}|\log\sigma_0|^{-p^2/\xi_4}\sigma_0^{(H_1+2\alpha+3)p}\\
&\leq&2\epsilon+2\nu+\epsilon+\frac{1}{2}(1+\epsilon)^{H_1+2\alpha+1}\sigma_0^{H_1+2\alpha+1}+\epsilon^{-p}\sigma_0^{(H_1+2\alpha+3)p}|\log\sigma_0|^{-p^2/\xi_4}.
\end{eqnarray*}
Choose $\epsilon = \sigma_0^{2\alpha}$ and $\nu = \sigma_0^{2\alpha}$. When $H_1>2\alpha$ and $\sigma_0\rightarrow 0$, we have
\[\log\frac{1+\epsilon}{1-\nu}+\frac{1}{2}\log\frac{(1+\epsilon)\tau_{\max}}{\tau_{\min}}+\left(\tau_{\max}^{1/2}(1+\epsilon^{-1})k_{\sigma_0}^{-1}\sigma_0^{H_1+2\alpha+2}\right)^p\lesssim\sigma_0^{2\alpha}.\]
Next, for the second term, as $w_{\sigma_0}^*\in\Delta_{k_{\sigma_0}}(w_{\sigma_0},\sigma_0^{H_1+2\alpha+1})$, by (\ref{eq:D2}), we have
\[D\left(P_{\wt f_0}^{n}\|P_{k_{\sigma_0}, \theta^{*(k_{\sigma_0})}}^{n}\right) = nD\left(P_{\wt f_0}\|P_{k_{\sigma_0}, \theta^{*(k_{\sigma_0})}}\right)\leq nD_2\left(P_{\wt f_0}\|P_{k_{\sigma_0}, \theta^{*(k_{\sigma_0})}}\right)\lesssim n\sigma_0^{2\alpha}.\]
Finally, for the last term,
\begin{eqnarray*}
&&D\left(Q^{(k_{\sigma_0})}\|\Pi^{(k_{\sigma_0})}\right)\\
 &=& -\log\pi(k_{\sigma_0})-\log\Pi_w^{(k_{\sigma_0})}(\wt\Theta_w^{(k_{\sigma_0})})-\sum_{j=1}^{k_{\sigma_0}}\log\Pi_{\mu_j}(\wt\Theta_{\mu_j}^{(k_{\sigma_0})})-\log\Pi_{\sigma}(\wt\Theta_{\sigma}^{(k_{\sigma_0})})
\end{eqnarray*}
Then, by the same arguments as in the proof of Lemma \ref{lem:mixture3}, when $\sigma_0 = n^{-\frac{1}{2\alpha+1}}(\log n)^{\frac{2\alpha r}{2\alpha+1}}$, we can further obtain that
\[-\log\Pi_{\mu_j}\left(\wt\Theta_{\mu_j}^{(k_{\sigma_0})}\right)\lesssim|\log\sigma_0|^{\max\{1,c_6/\xi_4\}},\]
\[-\log\Pi_{w}^{(k_{\sigma_0})}(\wt\Theta_w^{(k_{\sigma_0})})\lesssim k_{\sigma_0}(\log k_{\sigma_0})^{d_3}|\log\frac{\sigma_0^{2\alpha}}{2k_{\sigma_0}}\sigma_0^{H_1+2\alpha+1}|\asymp k_{\sigma_0}(\log k_{\sigma_0})^{d_3}|\log\sigma_0|,\]
\[-\log\Pi_{\sigma}(\wt\Theta_{\sigma_0}^{(k_{\sigma_0})})\lesssim |\log\wt\sigma_0|+\wt\sigma_0^{-b_6}\lesssim\sigma_0^{-1},\]
\[-\log\pi(k_{\sigma_0})\lesssim k_{\sigma_0}\log k_{\sigma_0}.\]
Therefore, with the choice of $\xi_4 = \min\{p,\xi_3\}$, we have
\[D\left(Q^{(k_{\sigma_0})}\|\Pi^{(k_{\sigma_0})}\right)\lesssim\sigma_0^{-1}(\log\sigma_0)^r,\]
where $r$ is the same defined in Theorem \ref{thm:mixture}. Combining all the bounds above, we have
\begin{eqnarray*}
&&D\left(\bar Q^{(k_{\sigma_0})}\|\bar\Pi^{(k_{\sigma_0})}\right)+\bar Q^{(k_{\sigma_0})}D\left(P_{\wt f_0}^{n}\|P(\cdot|k_{\sigma_0},z^{(k_{\sigma_0})},\theta^{(k_{\sigma_0})})\right)-\log\pi(k_{\sigma_0})\\
&\lesssim& n\sigma_0^{2\alpha}+\sigma_0^{-1}(\log\sigma_0)^{r}\asymp n^{\frac{1}{2\alpha+1}}(\log n)^{\frac{2\alpha r}{2\alpha+1}}.
\end{eqnarray*}
Finally, we have
\[P_{\wt f_0}^{n}\left[\wh Q^{(\wh k)}nH\left(P_{\wh k,\theta^{(\wh k)}}, P_{\wt f_0}\right)^2\right]\lesssim n^{\frac{1}{2\alpha+1}}(\log n)^{\frac{2\alpha r}{2\alpha+1}}.\]
With the same change of measure argument in the proof of Theorem \ref{thm:mixture}, the proof is complete.
\end{proof}

\subsection{Proofs of Theorem \ref{thm:VB-sieve-general}, Corollary \ref{cor:VB-sieve-factorize} and Theorem \ref{thm:EB as VB}}\label{sec:pf-EB}

We first show the following lemma to assist the proof of Theorem \ref{thm:VB-sieve-general}
\begin{lemma}\label{lem:support}
The variational posterior $\wh{Q}$ with respect to the set $\mathcal{S}_{\rm MF}$ is a product measure, with the density for each coordinate in the form of
\begin{equation}\label{eq:combine}
q_j(\theta_j) = \left\{\begin{array}{ll} g_{j1}(\theta_j),&j<k,\\ (1- p) g_{j1}(\theta_j)+pg_{j2}(\theta_j),&j =  k,\\ g_{j2}(\theta_j),&j>k,\end{array}\right.
\end{equation}
where $g_{j1}\in\mathcal{G}_{j1} = \left\{g:\int_{\Theta_{j1}}g(\theta_j)d\theta_j = 1\right\}$ and $g_{j2}\in\mathcal{G}_{j2} = \left\{g:\int_{\Theta_{j2}}g(\theta_j)d\theta_j = 1\right\}$ for all j, $k$ is some integer, and $p\in [0,1)$.
\end{lemma}

\begin{proof}
In order that $D(\wh{Q}\|\Pi(\cdot|X^{(n)})) < \infty$, we must have
\[\supp(\wh Q)\subseteq\supp(\Pi(\cdot|X^{(n)}))\subseteq\supp(\Pi).\]
In other words, for any set $B$ such that $\Pi(B) = 1$, we must have $\wh Q(B) = 1$. For each coordinate, we can assume that $q_j(\theta_j) = p_jg_{j1}(\theta_j)+(1-p_j)g_{j2}(\theta_j)$, where $g_{j1}\in\mathcal{G}_{j1}$ and $g_{j2}\in\mathcal{G}_{j2}$. For each $k$, define
$$B_k=\left\{\theta=(\theta_j)_{j=1}^{\infty}: \theta_j\in \Theta_{j1}\text{ for }j\leq k\text{ and }\theta_j\in\Theta_{j2}\text{ for }j>k\right\}.$$
Obverse that for $j\neq l$, $B_j\cap B_l=\varnothing$. Then, we can define the set $B=\cup_{k=0}^{\infty}B_k$. According to the sampling process of $\Pi$, $\Pi(B ) = 1$, which implies that $\wh Q(B) = 1$.
Note that for each $k$,
$$\wh Q(B_k)=\prod_{j\leq k}p_j\prod_{j>k}(1-p_j),$$
and then
\begin{equation}
1=\wh Q(B)=\sum_{k=0}^{\infty}\prod_{j\leq k}p_j\prod_{j>k}(1-p_j).\label{eq:constraint}
\end{equation}
For any $0< k<s$,
\begin{eqnarray*}
1 &=& (1-p_k+p_k)(1-p_s+p_s) = (1-p_k)p_s+[(1-p_s)p_k+p_kp_s+\\
&&(1-p_k)(1-p_s)]\prod_{l\neq k,s}(1-p_l+p_l)\\
&\geq&(1-p_k)p_s+\sum_{k=0}^{\infty}\prod_{j\leq k}p_j\prod_{j>k}(1-p_j)\\
& = &(1-p_k)p_s+1.
\end{eqnarray*}
Therefore, $(1-p_k)p_s = 0$ for all $0<k<s$, and there are three possible cases:
\begin{itemize}
\item $p_j = 0$ for all $j$.
\item $p_j = 1$ for all $j$.
\item $p_j = 0$ for $j<k$, $p_j = 1$ for $j>k$, and $p_k\in[0,1)$ for some $k\in\mathbb{N}$.
\end{itemize} 
However, the first two cases do not satisfy the constraint (\ref{eq:constraint}). Thus, the variational posterior $\wh{Q}$ is limited to the form (\ref{eq:combine}), which completes the proof.
\end{proof}

\begin{proof}[Proof of Theorem \ref{thm:VB-sieve-general}]
By Lemma \ref{lem:support}, the variational posterior has the form
$$p\prod_{j<k}g_{j1}(\theta_j)\prod_{j\geq k}g_{j2}(\theta_j)+(1-p)\prod_{j\leq k}g_{j1}(\theta_j)\prod_{j>k}g_{j2}(\theta_j).$$
Now we need to determine $k$, $p$, $g_{j1}$ for $j\leq k$ and $g_{j2}$ for $j\geq k$. We denote the above distribution by $Q_k$. Then, it is easy to see that $Q_k(B_{k-1}\cup B_k)=1$. This implies
\begin{eqnarray}\label{eq:KL_Q_k}
&&D(Q_k\|\Pi(\cdot|X^{(n)}))\nonumber\\
 &=& \int_{B_{k-1}} p\prod_{j<k}g_{j1}(\theta_j)\prod_{j\geq k}g_{j2}(\theta_j)\log\frac{p\prod_{j<k}g_{j1}(\theta_j)\prod_{j\geq k}g_{j2}(\theta_j)}{\pi(k-1)p(X^{(n)}|\theta)\prod_{j< k}f_{j1}(\theta_j)\prod_{j\geq k}f_{j2}(\theta_j)}d\theta \nonumber\\
&& + \int_{B_{k}} (1-p)\prod_{j\leq k}g_{j1}(\theta_j)\prod_{j> k}g_{j2}(\theta_j)\log\frac{(1-p)\prod_{j\leq k}g_{j1}(\theta_j)\prod_{j> k}g_{j2}(\theta_j)}{\pi(k)p(X^{(n)}|\theta)\prod_{j\leq k}f_{j1}(\theta_j)\prod_{j>k}f_{j2}(\theta_j)}d\theta\nonumber\\
&& +\log p_\Pi(X^{(n)})\nonumber\\
&=& p\log\frac{p}{\pi(k-1)\exp\left(m_{k-1}\left(X^{(n)}, (g_{j1})_{j = 1}^{k-1}\cup(g_{j2})_{j=  k}^{\infty}\right)\right)}\nonumber\\
&&+(1-p)\log\frac{1-p}{\pi(k)\exp\left(m_{k}\left(X^{(n)}, (g_{j1})_{j = 1}^{k}\cup(g_{j2})_{j=  k+1}^{\infty}\right)\right)}\nonumber\\
&&+\log p_{\Pi}(X^{(n)}),\nonumber\\
\end{eqnarray}
where $p_\Pi(X^{(n)}) = \int p(X^{(n)}|\theta)d\Pi(\theta)$. Minimizing $D(Q_k\|\Pi(\cdot|X^{(n)}))$ over $p$ leads to 
\[\wt p = \frac{\pi(k-1)e^{m_{k-1}\left(X^{(n)}, (g_{j1})_{j = 1}^{k-1}\cup(g_{j2})_{j=  k}^{\infty}\right)}}{\pi(k-1)e^{m_{k-1}\left(X^{(n)}, (g_{j1})_{j = 1}^{k-1}\cup(g_{j2})_{j=  k}^{\infty}\right)}+\pi(k)e^{m_{k}\left(X^{(n)}, (g_{j1})_{j = 1}^{k}\cup(g_{j2})_{j=  k+1}^{\infty}\right)}}.\]
Plugging $\wt p$ into (\ref{eq:KL_Q_k}), we have
\begin{eqnarray*}
&&D(Q_k\|\Pi(\cdot|X^{(n)}))\\
 &=&-\log\left[\pi(k-1)e^{m_{k-1}\left(X^{(n)}, (g_{j1})_{j = 1}^{k-1}\cup(g_{j2})_{j=  k}^{\infty}\right)}+\pi(k)e^{m_{k}\left(X^{(n)}, (g_{j1})_{j = 1}^{k}\cup(g_{j2})_{j=  k+1}^{\infty}\right)}\right]\\
 &&+\log p_{\Pi}(X^{(n)}).
\end{eqnarray*}
Therefore, $\wt k$, $\wt g_{j1}^{(\wt k)}$, $\wt g_{j2}^{(\wt k)}$ are the solution to maximizxing the objective function
$$\pi(k-1)e^{m_{k-1}\left(X^{(n)},(g_{j1})_{j=1}^{k-1}\cup (g_{j2})_{j=k}^{\infty}\right)}+\pi(k)e^{m_k\left(X^{(n)},(g_{j1})_{j=1}^{k}\cup (g_{j2})_{j=k+1}^{\infty}\right)},$$
under the constraints that $g_{j1}\in\mathcal{G}_{j1}$ and $g_{j2}\in\mathcal{G}_{j2}$ for all $j$. The proof is complete.
 \end{proof}

\begin{proof}[Proof of Corollary \ref{cor:VB-sieve-factorize}]
If $p(X^{(n)}|\theta) = \prod_{j=1}^{\infty}p(X_j^{(n)}|\theta_j)$, then
\begin{eqnarray*}
&&m_{k-1}(X^{(n)}, (g_{j1})_{j=1}^{k-1}\cup(g_{j2})_{j = k}^{\infty})\\
 &=& \sum_{j = 1}^{k-1}\int g_{j1}(\theta_j)\log\frac{f_{j1}(\theta_j)p(X_j^{(n)}|\theta_j)}{g_{j1}(\theta_j)}d\theta_j+\sum_{j = k}^{\infty}\int g_{j2}(\theta_j)\log\frac{f_{j2}(\theta_j)p(X_j^{(n)}|\theta_j)}{g_{j2}(\theta_j)}d\theta_j\\
&\leq&\sum_{j = 1}^{k-1}\log\int_{\Theta_{j1}} f_{j1}(\theta_j)p(X_j^{(n)}|\theta_j)d\theta_j+\sum_{j = k}^{\infty}\log\int_{\Theta_{j2}} f_{j2}(\theta_j)p(X_j^{(n)}|\theta_j)d\theta_j,
\end{eqnarray*}
and
\begin{eqnarray*}
&&m_{k}(X^{(n)}, (g_{j1})_{j=1}^{k}\cup(g_{j2})_{j = k+1}^{\infty})\\
 &=& \sum_{j = 1}^{k}\int g_{j1}(\theta_j)\log\frac{f_{j1}(\theta_j)p(X_j^{(n)}|\theta_j)}{g_{j1}(\theta_j)}d\theta_j+\sum_{j = k+1}^{\infty}\int g_{j2}(\theta_j)\log\frac{f_{j2}(\theta_j)p(X_j^{(n)}|\theta_j)}{g_{j2}(\theta_j)}d\theta_j\\
&\leq&\sum_{j = 1}^{k}\log\int_{\Theta_{j1}} f_{j1}(\theta_j)p(X_j^{(n)}|\theta_j)d\theta_j+\sum_{j = k+1}^{\infty}\log\int_{\Theta_{j2}} f_{j2}(\theta_j)p(X_j^{(n)}|\theta_j)d\theta_j.
\end{eqnarray*}
The equalities above hold when $g_{j1}(\theta_j)\propto f_{j1}(\theta_j)p(X_j^{(n)}|\theta_j)\1_{\theta_j\in\Theta_{j1}}$ for $j\leq k$ and $g_{j2}(\theta_j)\propto f_{j2}(\theta_j)p(X_j^{(n)}|\theta_j)\1_{\theta_j\in\Theta_{j2}}$ for $j\geq k$. Plug these choices into the objective function, and then the objective function becomes
\begin{eqnarray*}
&&\pi(k-1)\prod_{j=1}^{k-1}\int_{\Theta_{j1}} f_{j1}(\theta_j)p(X_j^{(n)}|\theta_j)d\theta_j\prod_{j=k}^{\infty}\int_{\Theta_{j2}} f_{j2}(\theta_j)p(X_j^{(n)}|\theta_j)d\theta_j\\
&&+\pi(k)\prod_{j=1}^{k}\int_{\Theta_{j1}} f_{j1}(\theta_j)p(X_j^{(n)}|\theta_j)d\theta_j\prod_{j=k+1}^{\infty}\int_{\Theta_{j2}} f_{j2}(\theta_j)p(X_j^{(n)}|\theta_j)d\theta_j.
\end{eqnarray*}
This implies that $\wt{k}$ maximizes $\pi(k-1|X^{(n)})+\pi(k|X^{(n)})$, where
\[\pi(k|X^{(n)})\propto\pi(k)\prod_{j=1}^{k}\int_{\Theta_{j1}} f_{j1}(\theta_j)p(X_j^{(n)}|\theta_j)d\theta_j\prod_{j=k+1}^{\infty}\int_{\Theta_{j2}} f_{j2}(\theta_j)p(X_j^{(n)}|\theta_j)d\theta_j.\]
Therefore, $\wt p$ is given by 
\[\wt p = \frac{\pi(k-1|X^{(n)})}{\pi(k-1|X^{(n)})+\pi(k|X^{(n)})}.\]
The proof is complete.
\end{proof}

\begin{proof}[Proof of Theorem \ref{thm:EB as VB}]
Assume $Q\in\mathcal{S}_{\rm EB}$, according to the definition, there exists a $k$ such that 
\[Q\left(B_k\right) = 1,\]
where $B_k = (\otimes_{j\leq k}\Theta_{j1})\bigotimes(\otimes_{j>k}\Theta_{j2})$. Then
\begin{eqnarray}\label{eq:Q_EB_VB}
&&D\left(Q\|\Pi(\cdot|X^{(n)})\right) = \int_{B_k}dQ(\theta)\log\frac{dQ(\theta)p_\Pi^{(n)}(X^{(n)})}{d\Pi(\theta)p(X^{(n)}|\theta)}\nonumber\\
&=&\int_{B_k}dQ(\theta)\log\frac{dQ(\theta)/d\theta}{\pi(k)\prod_{j\leq k}f_{j1}(\theta_j)\prod_{j>k}f_{j2}(\theta_j)p(X^{(n)}|\theta)}+\log p_{\Pi}^{(n)}(X^{(n)}),\nonumber\\
\end{eqnarray}
where $p_{\Pi}^{(n)}(X^{(n)}) = \int p^{(n)}(X^{(n)}|\theta)d\Pi(\theta)$.

Therefore, for a specific $k$, $Q$ is chosen as
\[dQ^{(k)}(\theta)\propto\prod_{j\leq k}f_{j1}(\theta_j)\prod_{j>k}f_{j2}(\theta_j)p(X^{(n)}|\theta)d\theta,\]
to minimize $D\left(Q\|\Pi(\cdot|X^{(n)})\right)$ under the constraint that $Q\left(B_k\right) = 1$. Plug this form into the right hand side of (\ref{eq:Q_EB_VB}), we can get $k = \wh k$ selected as
\[\wh k = \argmax_{k}\pi(k)\int\prod_{j\leq k}f_{j1}(\theta_j)\prod_{j>k}f_{j2}(\theta_j)p(X^{(n)}|\theta)d\theta.\]
And therefore, $\wh Q_{\rm EB} = \wh Q^{(\wh k)}$ is the variational posterior with the variational class $\mathcal{S}_{\rm EB}$. 
\end{proof}

\subsection{Proof of Theorem \ref{thm:UB for S_k}}

\begin{proof}[Proof of Theorem \ref{thm:UB for S_k}]
Recall that
$$d\wh{Q}_{[k]}=\prod_{j\leq k}dN\left(\frac{n}{n+j^{2\beta+1}}Y_j,\frac{1}{n+j^{2\beta+1}}\right)\prod_{j=k+1}^ndN(0,e^{-jn})\prod_{j>n}\delta_0.$$
Then, we can decompose the risk into
\begin{eqnarray*}
&&P_{\theta^*}^{(n)}\wh{Q}_{[k]}\|{\theta}-\theta^*\|^2 = P_{\theta^*}^{(n)}\wh{Q}_{[k]}\sum_{j\leq k}(\theta_j-\theta_j^*)^2 + P_{\theta^*}^{(n)}\wh{Q}_{[k]}\sum_{j> k}(\theta_j-\theta_j^*)^2 \\
&=& P_{\theta^*}^{(n)}\sum_{j\leq k}\left(\frac{n}{n+j^{2\beta+1}}Y_j-\theta_j^*\right)^2 + \sum_{j>k}\theta_{j}^{*2}  +\sum_{j\leq k}\frac{1}{n+j^{2\beta+1}} + \sum_{j=k+1}^ne^{-jn} \\
&=& \sum_{j\leq k}\left(\frac{j^{2\beta+1}}{n+j^{2\beta+1}}\right)^2\theta_j^{*2} + \sum_{j>k}\theta_{j}^{*2}  + \sum_{j\leq k}\frac{n}{(n+j^{2\beta+1})^2}\\
&& + \sum_{j\leq k}\frac{1}{n+j^{2\beta+1}} + \sum_{j=k+1}^ne^{-jn}.
\end{eqnarray*}
For the upper bound, we have
$$P_{\theta^*}^{(n)}\wh{Q}_{[k]}\|{\theta}-\theta^*\|^2\leq \sum_{j\leq k}\left(\frac{j^{2\beta+1}}{n+j^{2\beta+1}}\right)^2\theta_j^{*2} + \sum_{j>k}\theta_{j}^{*2}+2\sum_{j\leq k}\frac{1}{n+j^{2\beta+1}}+2e^{-kn}.$$

Now we discuss in the two cases:

\begin{itemize}
\item When $k\leq n^{\frac{1}{2\beta+1}}$, we have
\[\sum_{j>k}\theta_{j}^{*2}\leq k^{-2\alpha}B^2,\qquad \sum_{j\leq k}\frac{1}{n+j^{2\beta+1}}\leq \frac{k}{n},\]
and
$$\sum_{j\leq k}\left(\frac{j^{2\beta+1}}{n+j^{2\beta+1}}\right)^2\theta_j^{*2}\leq \sum_{j\leq k}\frac{j^{4\beta+2-2\alpha}}{n^2}j^{2\alpha}\theta_j^{*2}\leq \frac{1+k^{4\beta+2-2\alpha}}{n^2}B^2.$$
Therefore,
$$P_{\theta^*}^{(n)}\wh{Q}_{[k]}\|{\theta}-\theta^*\|^2\lesssim k^{-2\alpha}+\frac{k}{n}.$$
\item When $k>n^{\frac{1}{2\beta+1}}$, we have
$$\sum_{j\leq k}\frac{1}{n+j^{2\beta+1}}\leq \frac{n^{\frac{1}{2\beta+1}}}{n}+\sum_{j>n^{\frac{1}{2\beta+1}}}j^{-2\beta-1}\lesssim n^{-\frac{2\beta}{2\beta+1}},$$
and
$$\sum_{j\leq k}\left(\frac{j^{2\beta+1}}{n+j^{2\beta+1}}\right)^2\theta_j^{*2}\leq \sum_{j\leq n^{\frac{1}{2\beta+1}}}\frac{j^{4\beta+2-2\alpha}}{n^2}j^{2\alpha}\theta_j^{*2}+\sum_{j>n^{\frac{1}{2\beta+1}}}\theta_j^{*2}\lesssim n^{-\frac{2\alpha}{2\beta+1}}.$$
Thus, we have
$$P_{\theta^*}^{(n)}\wh{Q}_{[k]}\|{\theta}-\theta^*\|^2\lesssim n^{-\frac{2(\alpha\wedge \beta)}{2\beta+1}}.$$
\end{itemize}
Now we prove the lower bound. According to the risk decomposition, we have
$$P_{\theta^*}^{(n)}\wh{Q}_{[k]}\|{\theta}-\theta^*\|^2\geq \sum_{j\leq k}\left(\frac{j^{2\beta+1}}{n+j^{2\beta+1}}\right)^2\theta_j^{*2} + \sum_{j>k}\theta_{j}^{*2}+\sum_{j\leq k}\frac{1}{n+j^{2\beta+1}}.$$

\begin{itemize}
\item
When $k\leq n^{\frac{1}{2\beta+1}}$, we consider a $\theta^*$ with every coordinate $0$ except that $\theta_{k+1}^*=(k+1)^{-\alpha}B$. It is easy to check that $\theta^*\in\Theta_{\alpha}(B)$. Then, we have
$\sum_{j\leq k}\frac{1}{n+j^{2\beta+1}}\geq \frac{k}{2n}$ and $\sum_{j>k}\theta_{j}^{*2}\geq B^2(k+1)^{-2\alpha}$. Therefore,
$$\sup_{\theta^*\in\Theta_{\alpha}(B)}P_{\theta^*}^{(n)}\wh{Q}_{[k]}\|{\theta}-\theta^*\|^2\gtrsim k^{-2\alpha}+\frac{k}{n}.$$
\item
When $k>n^{\frac{1}{2\beta+1}}$, we consider a $\theta^*$ with every coordinate $0$ except that $\theta^*_{\ceil{n^{\frac{1}{2\beta+1}}}}=\left(\ceil{n^{\frac{1}{2\beta+1}}}\right)^{-\alpha}B$, and it is easy to check that $\theta^*\in\Theta_{\alpha}(B)$. Then, we have
$$\sum_{j\leq k}\left(\frac{j^{2\beta+1}}{n+j^{2\beta+1}}\right)^2\theta_j^{*2} \geq \frac{1}{4}\theta^{*2}_{\ceil{n^{\frac{1}{2\beta+1}}}}\gtrsim n^{-\frac{2\alpha}{2\beta+1}},$$
and
$$\sum_{j\leq k}\frac{1}{n+j^{2\beta+1}}\gtrsim \sum_{j\leq n^{\frac{1}{2\beta+1}}}\frac{1}{n+j^{2\beta+1}}\gtrsim n^{-\frac{2\beta}{2\beta+1}}.$$
This leads to the lower bound
$$\sup_{\theta^*\in\Theta_{\alpha}(B)}P_{\theta^*}^{(n)}\wh{Q}_{[k]}\|{\theta}-\theta^*\|^2\gtrsim n^{-\frac{2(\alpha\wedge \beta)}{2\beta+1}}.$$
\end{itemize}
Now the proof is complete.
\end{proof}

\subsection{Proofs of Theorem \ref{thm:lasso} and Theorem \ref{thm:VB-LASSO}}

\begin{proof}[Proof of Theorem \ref{thm:lasso}]
Define $Q = N(\beta_0,\tau^2 I_p)$, and then
\[Q\log(dQ) = -\frac{p}{2}\log(2\pi\tau^2e),\]
which is a constant with respect to $\beta_0$. Thus,
\begin{eqnarray*}
\wh Q_{\tau^2} &=& \argmin_{Q\in\mathcal{S}_{\tau^2}}D\left(Q\|\Pi(\cdot|y)\right)\\
&=&\argmin_{Q\in\mathcal{S}_{\tau^2}}\{Q\log(dQ)-Q\log(d\Pi)-Q\log p_{\beta}(y)\}\\
&=&\argmin_{Q\in\mathcal{S}_{\tau^2}}\left\{\frac{1}{2}Q\|X\beta-y\|^2+\lambda \sum_{i=1}^pQ|\beta_i|\right\}.
\end{eqnarray*}
where $p_\beta(y) = \frac{1}{(2\pi)^{n/2}}\exp\left(-\frac{1}{2}\|y-X\beta\|^2\right)$. With $Q = N(\beta_0,\tau^2I_p)$, we have $\beta_j$'s independently drawn from $N(\beta_{0j},\tau^2)$, and therefore,
\[Q\|X\beta-y\|^2 = \|X\beta_0-y\|^2+\tau^2{\rm tr}(X^TX),\]
\[Q|\beta_i| = \tau Q|\tau^{-1}\beta_{0i}| = \tau h(\tau^{-1}\beta_{0i}),\]
where
\begin{eqnarray*}
h(x) &=& \int_{-\infty}^{\infty}|t|\phi(t-x)dt = \int_{-\infty}^{\infty}|t+x|\phi(t)dt\\
&=&\int_{-\infty}^{-x}-(t+x)\phi(t)dt+\int_{-x}^{\infty}(t+x)\phi(t)dt\\
&=&\phi(t)\Big|_{-\infty}^{-x}-\phi(t)\Big|_{-x}^{\infty}-x\Phi(-x)+x\Phi(x)\\
&=&2\phi(x)+x[\Phi(x)-\Phi(-x)].
\end{eqnarray*}
The proof is complete.
\end{proof}

\begin{lemma}\label{lem:bound}
\[0\leq h(x)-|x|\leq \sqrt{\frac{2}{\pi}}\qquad\text{for all $x$.}\]
\end{lemma}

\begin{proof}[Proof of Lemma \ref{lem:bound}]
It is not hard to see that $h(-x) = h(x)$. Thus, we only need to show the inequality for $x\geq 0$. For $x\geq 0$, set $d(x) = h(x)-x$, and then 
\[d'(x) = \Phi(x)-\Phi(-x)-1\leq 0.\]
Thus, $d(x)$ is monotonically decreasing when $x>0$ and $d(x)\leq d(0) = 2\phi(0) = \sqrt{\frac{2}{\pi}}$. For the left part of inequality, notice that $\frac{1-\Phi(x)}{\phi(x)}\leq\frac{1}{x}$ for all $x>0$ in \cite{pinelis2002monotonicity}, and then we can directly obtain that $d(x)\geq 0$.
\end{proof}

\begin{proof}[Proof of Theorem \ref{thm:VB-LASSO}]
We use the notation $H_{\tau}(\beta)=\sum_{j=1}^p\tau h(\beta_j/\tau)$. By Lemma \ref{lem:bound}, we have
$|H_{\tau}(\beta)-\|\beta\|_1| \leq \tau\sum_{j = 1}^p|h(\beta_j/\tau)-\beta_j/\tau|\leq p\tau\sqrt{2/\pi}$.
By rearranging the basic inequality $\|y-X\wh{\beta}\|^2 + 2\lambda H_\tau(\wh{\beta})\leq \|y-X\beta^*\|^2 + 2\lambda H_\tau(\beta^*)$, we have
$$\|X(\wh{\beta}-\beta^*)\|^2 \leq  2\left|\iprod{X(\wh{\beta}-\beta^*)}{y-X\beta^*}\right| + 2\lambda H_\tau(\beta^*) - 2\lambda H_\tau(\wh{\beta}).$$
For the terms on the right hand side of the above inequality, we have $\left|\iprod{X(\wh{\beta}-\beta^*)}{y-X\beta^*}\right|\leq \|X^T(y-X\beta^*)\|_{\infty}\|\wh{\beta}-\beta^*\|_1$ and $H_\tau(\beta^*) -  H_\tau(\wh{\beta}) \leq \|\beta^*\|_1-\|\wh{\beta}\|_1 + 2p\tau\sqrt{2/\pi}$. Therefore, with the notation $\Delta=\wh{\beta}-\beta^*$, we have
\begin{equation}
\|X\Delta\|^2 \leq \lambda \|\Delta\|_1 + 2\lambda\|\beta^*\|_1 - 2\lambda\|\beta^*+\Delta\|_1 + 2\lambda p\tau\sqrt{2/\pi}, \label{eq:Lasso-step}
\end{equation}
as long as $\lambda\geq 2\|X^T(y-X\beta^*)\|_{\infty}$. Note that the choice $\lambda=C\sqrt{n\log p}$ implies that the condition $\lambda\geq 2\|X^T(y-X\beta^*)\|_{\infty}$ holds with probability at least $1-p^{-C'}$ by a union bound argument in \cite{bickel2009simultaneous}. With the decompositions $\|\Delta\|_1=\|\Delta_S\|_1+\|\Delta_{S^c}\|_1$, $\|\beta^*\|_1=\|\beta_S^*\|_1$ and $\|\beta^*+\Delta\|_1=\|\beta^*_S+\Delta_S\|_1 + \|\Delta_{S^c}\|_1$, the inequality (\ref{eq:Lasso-step}) becomes
\begin{equation}
\|X\Delta\|^2 \leq \lambda\left(3\|\Delta_S\|_1-\|\Delta_{S^c}\|_1+2p\tau\sqrt{2/\pi}\right).\label{eq:Lasso-mid}
\end{equation}
The inequality (\ref{eq:Lasso-mid}) immediately implies what is known as the generalized cone condition defined in \cite{gao2017sparse},
\begin{equation}
\|\Delta_{S^c}\|_1\leq 3\|\Delta_S\|_1 + 2p\tau\sqrt{2/\pi}. \label{eq:g-cone}
\end{equation}
Another consequence of (\ref{eq:Lasso-mid}) is the error bound
\begin{equation}
\|X\Delta\|^2 \leq \lambda\left(3\sqrt{s}\|\Delta\|+2p\tau\sqrt{2/\pi}\right). \label{eq:g-error}
\end{equation}
For the $\Delta$ that satisfies (\ref{eq:g-cone}), we define
\begin{eqnarray*}
\Delta^{(1)} &=& \frac{3\|\Delta_S\|_1+2p\tau\sqrt{2/\pi}}{3\|\Delta_S\|_1}\Delta_S + \frac{3\|\Delta_S\|_1}{3\|\Delta_S\|_1+2p\tau\sqrt{2/\pi}}\Delta_{S^c}, \\
\Delta^{(2)} &=& -\frac{2p\tau\sqrt{2/\pi}}{3\|\Delta_S\|_1}\Delta_S + \frac{2p\tau\sqrt{2/\pi}}{3\|\Delta_S\|_1+2p\tau\sqrt{2/\pi}}\Delta_{S^c}.
\end{eqnarray*}
It is easy to see that $\Delta=\Delta^{(1)}+\Delta^{(2)}$. Since
\begin{eqnarray*}
&&\|\Delta^{(1)}_{S^c}\|_1=\frac{3\|\Delta_S\|_1}{3\|\Delta_S\|_1+2p\tau\sqrt{2/\pi}}\|\Delta_{S^c}\|_1\\
&\leq& \frac{3\|\Delta_S\|_1}{3\|\Delta_S\|_1+2p\tau\sqrt{2/\pi}}(3\|\Delta_{S}\|_1+2p\tau\sqrt{2/\pi})=3\|\Delta_{S}\|_1\leq 3\|\Delta_S^{(1)}\|_1,
\end{eqnarray*}
we have $\frac{1}{\sqrt{n}}\|X\Delta^{(1)}\|\geq \kappa \|\Delta^{(1)}\|$ by the definition of $\kappa$ in (\ref{eq:RE}). We also bound $\|\Delta^{(2)}\|$ by
$$\|\Delta^{(2)}\|\leq \|\Delta^{(2)}\|_1\leq \frac{2p\tau\sqrt{2/\pi}}{3} + \frac{2p\tau\sqrt{2/\pi}}{3\|\Delta_S\|_1+2p\tau\sqrt{2/\pi}}\|\Delta_{S^c}\|_1\leq \frac{8p\tau\sqrt{2/\pi}}{3},$$
where the last inequality is by (\ref{eq:g-cone}). Therefore,
$$\|\Delta\|\leq \|\Delta^{(1)}\|+\|\Delta^{(2)}\|\leq \frac{\|X\Delta^{(1)}\|}{\kappa\sqrt{n}}+\frac{8p\tau\sqrt{2/\pi}}{3}.$$
Since $\|X\Delta^{(2)}\|\leq\sqrt{n\max_{i,j}|X_{ij}|^2\|\Delta^{(2)}\|_1^2}\leq\sqrt{L^2n^2\|\Delta^{(2)}\|_1^2}\leq nL\frac{8p\tau\sqrt{2/\pi}}{3}$, we have
$$\|\Delta\|\leq \frac{\|X\Delta\|}{\kappa\sqrt{n}}+\frac{8p\tau\sqrt{2/\pi}}{3}+\frac{\|X\Delta^{(2)}\|}{\kappa\sqrt{n}}\leq \frac{\|X\Delta\|}{\kappa\sqrt{n}}+\left(1+\frac{\sqrt{n}L}{\kappa}\right)\frac{8p\tau\sqrt{2/\pi}}{3}.$$
Combining the above inequality and (\ref{eq:g-error}), we have
$$\|\Delta\|^2\lesssim \frac{\|X\Delta\|^2}{n\kappa^2}+\left(1+\frac{n}{\kappa^2}\right)p^2\tau^2\lesssim
\left(\frac{\lambda\sqrt{s}}{n\kappa^2}\|\Delta\| +\frac{p\tau\lambda}{n\kappa^2}\right)+ \left(1+\frac{n}{\kappa^2}\right)p^2\tau^2,$$
which further leads to
$$\|\Delta\|^2\lesssim \frac{\lambda^2s}{n^2\kappa^4} + \frac{p\tau\lambda}{n\kappa^2} + \left(1+\frac{n}{\kappa^2}\right)p^2\tau^2.$$
With $\lambda=C\sqrt{n\log p}$ and $\tau=O\left(\frac{1}{np}\right)$, we have $\|\Delta\|^2\lesssim \frac{s\log p}{n\kappa^4}$, which completes the proof.
\end{proof}

\subsection{Proofs of Theorem \ref{thm:misspecification}, Theorem \ref{thm:mis_MF} and Theorem \ref{thm:PC-MC-misp}}\label{sec:pf-misp}

To show Theorem \ref{thm:misspecification}, we need the following three lemmas.

\begin{lemma}\label{lem:exp_convergence}
If the conditions (\ref{eq:C1})-(\ref{eq:C3}) in Theorem \ref{thm:convergence} are satisfied, then there exists a constant $M_0>0$ large enough such that when $n\epsilon^2\geq n\epsilon_n^2\geq M_0$ and $0<a\leq\frac{m}{2C_1}$, we have
\[P_0^{(n)}\left(\log\Pi\left[\exp\left(aL(P_\theta^{(n)},P_{0}^{(n)})\right)\Big|X^{(n)}\right]\geq aC_1n\epsilon^2+\log 2\right)\leq\exp\left(-\frac{m}{2}n\epsilon^2\right),\]
where $m = \frac{1}{2}\min\{1,\rho-1\}$.
\end{lemma}

\begin{lemma}\label{lem:VB-exp}
Under the conditions (\ref{eq:C1})-(\ref{eq:C3}) in Theorem \ref{thm:convergence}, there exist some constants $M_0>1$, $M>0$ and $c>0$ such that when $n\epsilon^2\geq n\epsilon_n^2\geq M_0$,
 \[P_0^{(n)}\left(\wh QL(P_\theta^{(n)},P_{0}^{(n)})> M\left(D\left(\wh Q\|\Pi(\cdot|X^{(n)})\right)+n\epsilon^2\right)\right)\leq\exp\left(-cn\epsilon^2\right),\]
where $\wh Q$ is the variational posterior distribution defined in (\ref{eq:VB}).
\end{lemma}

\begin{lemma}\label{lem:mis}
Suppose the conditions (\ref{eq:C1})-(\ref{eq:C3}) in Theorem \ref{thm:convergence} hold for $P_0^{(n)} = P_{\theta_0}^{(n)}$, when $n\epsilon^2\geq n\epsilon_n^2\geq M_0$, with any $p>1$ as a constant
\begin{eqnarray*}
&&P_*^{(n)}\left(\wh QL(P_\theta^{(n)},P_{\theta_0}^{(n)})> M\left(D\left(\wh Q\|\Pi(\cdot|X^{(n)})\right)+c^{-1}D_p\left(P_*^{(n)}\|P_{\theta_0}^{(n)}\right)+n\epsilon^2\right)\right)\\
&\leq&\exp\left(-\frac{(p-1)c}{p}n\epsilon^2\right),
\end{eqnarray*}
where $M_0$, $M$ and $c$ are the same constants in Lemma \ref{lem:VB-exp} and $\wh Q$ is the variational posterior distribution defined in (\ref{eq:VB}).
\end{lemma}

\begin{proof}[Proof of Lemma \ref{lem:exp_convergence}]
According to Lemma \ref{lem:strong convergence}
\[P_0^{(n)}\Pi\left(L(P_{\theta}^{(n)},P_0^{(n)})>C_1n\epsilon^2\Big|X^{(n)}\right)\leq 4\exp\left(-2mn\epsilon^2\right),\]
with $m = \frac{1}{2}\min\{1,\rho-1\}$ for any $\epsilon\geq\epsilon_n$. Then by Markov inequality,
\[P_0^{(n)}\left[\Pi\left(L(P_\theta^{(n)}, P_0^{(n)})>C_1n\epsilon^2\Big|X^{(n)}\right)>\exp\left(-mn\epsilon^2\right)\right]\leq 4\exp\left(-mn\epsilon^2\right).\]
Denote $B_j = \left\{X^{(n)}\Big|\Pi\left(L(P_\theta^{(n)}, P_0^{(n)})>C_1jn\epsilon^2\Big|X^{(n)}\right)\leq\exp\left(-mjn\epsilon^2\right)\right\}$ and $B = \cap_{j=1}^\infty B_j$. Then,
\[P_0^{(n)}(B^c) \leq \sum_{j=1}^\infty P_0^{(n)}(B_j^c)\leq 4\sum_{j=1}^\infty\exp(-mjn\epsilon^2)\leq\frac{4}{1-\exp(-mn\epsilon^2)}\exp(-mn\epsilon^2).\]
When $M_0 = \frac{2\log 8}{m}$ and $n\epsilon^2\geq M_0$, it is easy to check that 
\[P_0^{(n)}(B^c)\leq \exp\left(-\frac{m}{2}n\epsilon^2\right).\]
Under the event $B$,
\begin{eqnarray*}
&&\Pi\left[\exp\left(aL(P_\theta^{(n)},P_{0}^{(n)})\right)\Big|X^{(n)}\right]\\
&\leq&\exp\left(aC_1n\epsilon^2\right)+\int_{\exp\left(aC_1n\epsilon^2\right)}^\infty\Pi\left[\exp\left(aL(P_\theta^{(n)},P_{0}^{(n)})\right)\geq t\Big|X^{(n)}\right]dt\\
&\leq&\exp\left(aC_1n\epsilon^2\right)+\sum_{j=1}^\infty\left[\exp\left((j+1)aC_1n\epsilon^2\right)-\exp\left(jaC_1n\epsilon^2\right)\right]\Pi\left[L(P_\theta^{(n)},P_{0}^{(n)})\geq jC_1n\epsilon^2\Big|X^{(n)}\right]\\
&\leq&\exp\left(aC_1n\epsilon^2\right)+\exp\left(aC_1n\epsilon^2\right)\sum_{j=1}^\infty\exp\left((aC_1-m)jn\epsilon^2\right)\\
&\leq&\exp\left(aC_1n\epsilon^2\right)\left[1+\sum_{j=1}^\infty\exp\left(-\frac{m}{2}jn\epsilon^2\right)\right]\\
&=&\exp\left(aC_1n\epsilon^2\right)\frac{1}{1-\exp\left(-\frac{m}{2}n\epsilon^2\right)}\leq 2\exp\left(aC_1n\epsilon^2\right),
\end{eqnarray*}
where we have used the condition that $0<a\leq\frac{m}{2C_1}$. The conclusion of the lemma directly follows the result above.
\end{proof}

\begin{proof}[Proof of Lemma \ref{lem:VB-exp}]
According to Lemma \ref{lem:basic ineq}, for any $a>0$, we have
\[a\wh QL(P_\theta^{(n)},P_{0}^{(n)})\leq D\left(\wh Q\big\|\Pi(\cdot|X^{(n)})\right)+\log\Pi\left[\exp\left(aL(P_\theta^{(n)},P_{0}^{(n)})\right)\Big|X^{(n)}\right].\]
Choose $a = \frac{\min\{1,\rho-1\}}{4C_1}$, and then according to Lemma \ref{lem:exp_convergence}, under the event $B$ (defined in the proof of Lemma \ref{lem:exp_convergence}), for $n\epsilon^2>M_0>1$, we have
\[\wh QL(P_\theta^{(n)},P_{0}^{(n)})\leq \frac{1}{a}\left(D\left(\wh Q\big\|\Pi(\cdot|X^{(n)})\right)+\log 2\right)+C_1n\epsilon^2\leq M\left(D\left(\wh Q\|\Pi(\cdot|X^{(n)})\right)+n\epsilon^2\right),\]
with $M = \max\left\{\frac{\log 2}{a}+C_1,\frac{1}{a}\right\}$. Therefore, the conclusion that
\[P_0^{(n)}\left(\wh QL(P_\theta^{(n)},P_{0}^{(n)})> M\left(D\left(\wh Q\|\Pi(\cdot|X^{(n)})\right)+n\epsilon^2\right)\right)\leq\exp\left(-cn\epsilon^2\right),\]
is implied by $P_0^{(n)}(B^c)\leq \exp\left(-cn\epsilon^2\right)$,
where $c = \frac{\min\{1,\rho-1\}}{4}$.
\end{proof}

\begin{proof}[Proof of Lemma \ref{lem:mis}]
By H\"older's inequality,
\begin{eqnarray*}
&&P_*^{(n)}\left(\wh QL(P_\theta^{(n)},P_{\theta_0}^{(n)})> M\left[D\left(\wh Q\|\Pi(\cdot|X^{(n)})\right)+n\epsilon^2\right]\right)\\
&\leq&\left(P_{\theta_0}^{(n)}\left(\frac{dP_*^{(n)}}{dP_{\theta_0}^{(n)}}\right)^p\right)^{1/p}\left(P_{\theta_0}^{(n)}\left(\wh QL(P_\theta^{(n)},P_{\theta_0}^{(n)})> M\left[D\left(\wh Q\|\Pi(\cdot|X^{(n)})\right)+n\epsilon^2\right]\right)\right)^{1-1/p}\\
&\leq&\exp\left(-\frac{p-1}{p}\left(cn\epsilon^2-D_p\left(P_*^{(n)}\|P_{\theta_0}^{(n)}\right)\right)\right)\\
\end{eqnarray*}
where $q = \frac{p}{p-1}$. Replace $n\epsilon^2$ by $n\epsilon^2+c^{-1}D_p\left(P_*^{(n)}\|P_{\theta_0}^{(n)}\right)$, and we have
\begin{eqnarray*}
&&P_*^{(n)}\left(\wh QL(P_\theta^{(n)},P_{0}^{(n)})> M\left(D\left(\wh Q\|\Pi(\cdot|X^{(n)})\right)+c^{-1}D_p\left(P_*^{(n)}\|P_{\theta_0}^{(n)}\right)+n\epsilon^2\right)\right)\\
&\leq&\exp\left(-\frac{(p-1)c}{p}n\epsilon^2\right),
\end{eqnarray*}
where $M_0$, $M$ and $c$ are the same constants in Lemma \ref{lem:VB-exp}.
\end{proof}

Now we can show Theorem \ref{thm:misspecification}.
\begin{proof}[Proof of Theorem \ref{thm:misspecification}]
Define $$Y = M^{-1}\wh QL(P_\theta^{(n)},P_{\theta_0}^{(n)})-\left[D\left(\wh Q\|\Pi(\cdot|X^{(n)})\right)+c^{-1}D_p\left(P_*^{(n)}\|P_{\theta_0}^{(n)}\right)+n\epsilon_n^2\right].$$
Then Lemma \ref{lem:mis} implies that $$P_*^{(n)}(Y\geq t)\leq \exp\left(-\frac{(p-1)c}{p}(n\epsilon_n^2+t)\right),$$
for all $t\geq 0$. Note that
\begin{eqnarray*}
&&P_*^{(n)}\left(\wh QL\left(P_{\theta}^{(n)}, P_{\theta_0}^{(n)}\right)\right)\\
&=&MP_*^{(n)}D\left(\wh Q\|\Pi(\cdot|X^{(n)})\right)+Mc^{-1}D_p\left(P_*^{(n)}\|P_{\theta_0}^{(n)}\right)+Mn\epsilon_n^2+MP_*^{(n)}Y.
\end{eqnarray*}
For the first term on the right hand side of the above equality, we have 
\begin{eqnarray*}
P_*^{(n)}D\left(\wh Q\|\Pi(\cdot|X^{(n)})\right) &=& P_*^{(n)}\inf_{Q\in\mathcal{S}}D\left(Q\|\Pi(\cdot|X^{(n)})\right) \\
&\leq& \inf_{Q\in\mathcal{S}}P_*^{(n)}D\left(Q\|\Pi(\cdot|X^{(n)})\right)=n\gamma_n^2.
\end{eqnarray*}
The term $P_*^{(n)}Y$ can be bounded by
\begin{eqnarray*}
P_*^{(n)}Y &\leq& P_*^{(n)}Y\1_{\{Y\geq 0\}}\leq\int_0^{\infty}P_*^{(n)}(Y\geq t)dt\\
&\leq&\int_0^{\infty}\exp\left(-\frac{(p-1)c}{p}(n\epsilon_n^2+t)\right)dt\leq \frac{p}{(p-1)c}\exp\left(-\frac{(p-1)c}{p}n\epsilon_n^2\right)\\
&\lesssim&n\epsilon_n^2
\end{eqnarray*}
The proof is complete by choosing $p=2$.
\end{proof}

\begin{proof}[Proof of Theorem \ref{thm:mis_MF}]
Theorem \ref{thm:mis_MF} uses the same arguments in the proof of Theorem \ref{thm:convergence2} with $\epsilon_n^2$ replaced by $\epsilon_n^2+\frac{1}{n}D_2\left(P_*^{(n)}\|P_{\theta_0}^{(n)}\right)$. Therefore, we omit the details here.
\end{proof}
In the end of this part, we will show Theorem \ref{thm:PC-MC-misp}, which directly implies Theorem \ref{thm:PC-MC}. 
We want to check conditions (\ref{eq:C1})-(\ref{eq:C3}) for $\theta_0\in\Theta_{k_0}(B)$. For this aim, we establish the following lemmas.

\begin{lemma}\label{lem:marginal_prior}
The marginal sampling process of $\theta$ in the prior for piecewise constant model can be regarded as following procedure:
\begin{itemize}
\item Sample $k\sim\pi(k)$ with 
\begin{equation}\label{eq:pi_k}
\pi(k)= \frac{\Gamma(k-1+\alpha_0)\Gamma(n-k+\beta_0)\Gamma(\alpha_0+\beta_0)(n-1)!}{\Gamma(n-1+\alpha_0+\beta_0)\Gamma(\alpha_0)\Gamma(\beta_0)(k-1)!(n-k)!};
\end{equation}
\item Conditioning on $k$, sample $k-1$ change points uniformly from $\{2,3,\cdots, n\}$. In the other words, we uniformly sample a subset $S\subseteq\{2,3,\cdots, n\}$ of size $k-1$ with probability ${n-1\choose k-1}^{-1}$;
\item Conditioning on $S$, sample $\theta_i$ according to $\theta_i\sim g_i$ for all $i\in S$ and $\theta_i = \theta_{i-1}$ for all $i\not\in S$.
\end{itemize}
Moreover, when (\ref{eq:MC-condition1}) is satisfied, 
\[n^{-(C_2+1)(k-1)-1}\leq\pi(k)\leq n^{-(C_1-1)(k-1)}.\]
\end{lemma}

\begin{proof}
First of all, the density of marginal prior on $\theta$ can be written as
\begin{eqnarray*}
\frac{d\Pi(\theta)}{d\theta} &=& \int\sum_{z}\frac{\Gamma(\alpha_0+\beta_0)}{\Gamma(\alpha_0)\Gamma(\beta_0)}w^{\alpha_0+\sum_{i = 2}^nz_i-1}(1-w)^{\beta_0+n-1-\sum_{i = 2}^nz_i-1}\\
&&\times g(\theta_1)\prod_{z_i = 1,i>1}g(\theta_i)\prod_{z_i = 0,i>1}\delta_{\theta_{i-1}}(\theta_i)dw\\
&=&\sum_{k = 1}^n\pi(k){n-1\choose k-1}^{-1}\sum_{|S| = k-1}g(\theta_1)\prod_{i\in S}g(\theta_i)\prod_{i\not\in S}\delta_{\theta_{i-1}}(\theta_i),
\end{eqnarray*}
where $S$ above is the set of label $2\leq i\leq n$ such that $z_i = 1$ and $\pi(k)$ is defined in (\ref{eq:pi_k}), which implies the marginal sampling process of $\theta$ can be written as the procedure above.

Then the condition (\ref{eq:MC-condition1}) indicates that
\[\frac{\pi(k+1)}{\pi(k)} = \frac{k-1+\alpha_0}{n-k+\beta_0-1}\frac{n-k}{k}\leq\frac{n(\alpha_0+n-1)}{\beta_0}\leq n^{1-C_1},\]
\[\frac{\pi(k+1)}{\pi(k)} = \frac{k-1+\alpha_0}{n-k+\beta_0-1}\frac{n-k}{k}\geq\frac{\alpha_0}{(\beta_0+n)n}\geq n^{-C_2-1},\]
which implies that
\[n^{-(C_2+1)(k-1)}\pi(1)\leq\pi(k)\leq n^{-(C_1-1)(k-1)}\pi(1).\]
When $C_1>1$, $C_2>0$, it is easy to see that $1/n<\pi(1)<1$ as $\pi(k)$ is decreasing with respect to $k$. Hence, we have
\[n^{-(C_2+1)(k-1)-1}\leq\pi(k)\leq n^{-(C_1-1)(k-1)}.\]
\end{proof}

\begin{lemma}\label{lem:PC-estimator}
Suppose $\theta_0\in\Theta_{k_0}$. For some integer $m\geq k_0$, define
\begin{equation}\label{eq:PC-estimator}
\widehat\theta_m = \argmin_{\theta\in\Theta_m}\|\theta-X\|^2.
\end{equation}
Then for any $t\geq 24\sigma^2 r\log\frac{en}{r}$ with $r = \min\{n,m+k_0\}$,  we have
\[P_{\theta^*}^{(n)}(\|\widehat\theta_m-\theta^*\|^2>t)\leq\exp\left(-\frac{t}{16\sigma^2}\right).\]
\end{lemma}

\begin{proof}
According to the definition,
\[\|\widehat\theta_m-X\|^2\leq \|\theta^*-X\|^2.\]
Using the identity $\|\widehat\theta_m-X\|^2 = \|\widehat\theta_m-\theta^*\|^2+\|\theta^*-X\|^2+2\iprod{\widehat\theta_m-\theta^*}{\theta^*-X}$, we get
\[\|\widehat\theta_m-\theta^*\|\leq 2\left|\iprod{\frac{\widehat\theta_m-\theta^*}{\|\widehat\theta_m-\theta^*\|}}{X-\theta^*}\right|.\]
Since $\frac{\widehat\theta_m-\theta^*}{\|\widehat\theta_m-\theta^*\|}\in \Theta_r$, we have
\[\|\widehat\theta_m-\theta^*\|^2\leq 4\sigma^2\sup_{\|u\|=1:u\in\Theta_{r}}\left|\sum_{i=1}^nu_iZ_i\right|^2,\]
where $Z_i\sim N(0,1)$.
Then,
\begin{eqnarray*}
&&P_{\theta^*}^{(n)}(\|\widehat\theta_m-\theta^*\|^2>t)\leq \mathbb{P}\left(\sup_{\|u\|=1:u\in\Theta_{r}}\left|\sum_{i=1}^nu_iZ_i\right|^2\geq \frac{t}{4\sigma^2}\right)\\
&\leq&\sum_{x_1+x_2+\cdots+x_r = n}\mathbb{P}\left(\sup_{\sum_{i=1}^rx_i\wt u_i^2 =1}\left|\sum_{i=1}^r\sqrt{x_i}\wt u_i \wt Z_i\right|^2\geq \frac{t}{4\sigma^2}\right)\\
& = &\sum_{x_1+x_2+\cdots+x_r = n}\mathbb{P}\left(\|\wt Z\|^2\geq\frac{t}{4\sigma^2}\right),
\end{eqnarray*}
where $r= \min\{m+k_0, n\}$, $\wt Z = (\wt Z_1, \wt Z_2,\cdots, \wt Z_r)^T\sim N(0, I_r)$. Then
a standard chi-squared bound gives
\begin{eqnarray*}
&&P_{\theta^*}^{(n)}(\|\widehat\theta_m-\theta^*\|^2>t)\leq\exp\left(r\log\frac{en}{r}+\frac{r}{2}\log 2\right)\exp\left(-\frac{t}{8\sigma^2}\right)\\
&\leq&\exp\left(\frac{t}{16\sigma^2}\right)\exp\left(-\frac{t}{8\sigma^2}\right) = \exp\left(-\frac{t}{16\sigma^2}\right).
\end{eqnarray*}
The proof is complete.
\end{proof}

We want check conditions (\ref{eq:C1})-(\ref{eq:C3}) with respect to $\theta_0$. This step can be split into the following two lemmas.

\begin{lemma}\label{lem:PC-MC1}
Assume $\theta_0\in\Theta_{k_0}(B)$. For the prior $\Pi$ that satisfies (\ref{eq:MC-condition1}), the conditions (\ref{eq:C1}) and (\ref{eq:C2}) hold for all $\epsilon>\sqrt{\frac{k_0\log n}{n}}$.
\end{lemma}

\begin{lemma}\label{lem:PC-MC2}
Assume $\theta_0\in\Theta_{k_0}(B)$. For the prior $\Pi$ that satisfies (\ref{eq:MC-condition1}) and (\ref{eq:MC-condition2}), the conditions (\ref{eq:C3}) and (\ref{eq:C4**}) hold for $\epsilon_n = \sigma\sqrt{\frac{k_0\log n}{n}}$ with both $\mathcal{S} = \mathcal{S}_{\rm MC}$ and $\mathcal{S} = \mathcal{S}_{\rm MC}^{\rm joint}$.
\end{lemma}

Now we start to prove Lemma \ref{lem:PC-MC1} and Lemma \ref{lem:PC-MC2}.

\begin{proof}[Proof of Lemma \ref{lem:PC-MC1}]
For any $\epsilon>\sqrt{\frac{k_0\log n}{n}}$, we set $m = \lceil \frac{C_0n\epsilon^2}{2\log n}\rceil$. Choose a sufficiently large $C_0$ so that $m\geq 2k_0\geq 2$. 
We consider $\Theta_n(\epsilon) = \Theta_r$ with $r = \min\{k_0+m, n\}$. Then by the condition (\ref{eq:MC-condition1}) and Lemma \ref{lem:marginal_prior}, we have 
\[\Pi(\Theta_n(\epsilon)^c) = \sum_{j = r+1}^n\pi(j)\leq n^{-(C_1-1)r}\sum_{j = 1}^{n-r}\pi(j)\leq n^{-(C_1-1)r}\leq\exp\left(-(C_1-1)C_0n\epsilon^2\right),\] 
which implies (\ref{eq:C2}).

To show (\ref{eq:C1}),
we consider the testing function $\phi_n = \mathbb{I}\{\|\widehat\theta_m-\theta_0\|\geq 5\sigma\sqrt{(C_0+1)n\epsilon^2}\}$, where $\widehat\theta_m$ is defined in (\ref{eq:PC-estimator}). Note that  
\begin{eqnarray*}
&&(5\sigma\epsilon\sqrt{(C_0+1)n})^2\geq 24(C_0+1)\sigma^2n\epsilon^2\\
&\geq& 24(C_0n\epsilon^2+k_0\log n)\sigma^2\geq 24(2m+k_0)\sigma^2\log n,
\end{eqnarray*}
and
we apply Lemma \ref{lem:PC-estimator} to obtain
\[P_{\theta_0}^{(n)}\phi_n =P_{\theta_0}^{(n)}(\|\widehat\theta_m-\theta_0\|^2\geq 25(C_0+1)\sigma^2n\epsilon^2)\leq\exp\left(-\frac{25}{16}(C_0+1)n\epsilon^2\right).\]
Moreover, for any $\theta\in\Theta_n(\epsilon)$ and $\|\theta-\theta_0\|^2\geq 10\sigma\epsilon\sqrt{(C_0+1)n}$, we have
\[(5\sigma\epsilon\sqrt{(C_0+1)n})^2\geq 24(2m+k_0)\sigma^2\log n \geq 24(m+r)\sigma^2\log n,\]
and then,
\begin{eqnarray*}
P_\theta^{(n)}(1-\phi_n) &=& P_{\theta}^{(n)}(\|\widehat\theta_m-\theta_0\|\leq 5\sigma\epsilon\sqrt{(C_0+1)n})\\
&\leq& P_\theta^{(n)}(\|\widehat\theta_m-\theta\|\geq 5\sigma\epsilon\sqrt{(C_0+1)n})\\
&\leq& \exp\left(-\frac{25}{16}(C_0+1)n\epsilon^2\right).
\end{eqnarray*}
Therefore, (\ref{eq:C1}) is satisfied with a sufficiently large $C_0$.
\end{proof}

\begin{proof}[Proof of Lemma \ref{lem:PC-MC2}]
We first verify condition (\ref{eq:C3}). Note that for any $\rho>0$,
\[D_\rho\left(P_{\theta_0}^{(n)}\|P_\theta^{(n)}\right) = \frac{\rho}{2\sigma^2}\|\theta-\theta_0\|^2.\]
Consider the set $\Theta = \cup_{i = 1}^n[\theta_{0i}-n^{-1/2}, \theta_{0i}+n^{-1/2}]$, then for $n>1$,
\[\Theta\subseteq\left\{\theta: D_\rho\left(P_{\theta_0}^{(n)}\|P_{\theta}^{(n)}\right)\leq \frac{\rho}{\sigma^2}\leq \frac{\rho}{\sigma^2\log 2}k_0\log n\right\},\]
and
\begin{eqnarray*}
\Pi(\Theta) &\geq&\pi(k_0)\prod_{i\in S(\theta_0)}\int_{\theta_{0i}-n^{-1/2}}^{\theta_{0i}+n^{-1/2}}g(\theta_j)d\theta_j\geq n^{-(C_2+1)(k_0-1)-1}\left(2cn^{-1/2}\right)^{k_0}\\
&\geq&\exp\left(-\left(\frac{5}{2}+C_2-\log(2c)\right)n\epsilon_n^2\right),
\end{eqnarray*}
where $S(\theta_0) = \{i: i>1, \theta_{0i}\neq\theta_{0(i-1)}\}\cup\{1\}$ and $C_2$ is given in (\ref{eq:MC-condition1}). 
Therefore, condition (\ref{eq:C3}) is satisfied.

Now we check condition (\ref{eq:C4**}) for both $\mathcal{S} = \mathcal{S}_{\rm MC}$ and $\mathcal{S} = \mathcal{S}_{\rm MC}^{\rm joint}$. 
When $\mathcal{S} = \mathcal{S}_{\rm MC}$, assume $|S(\theta_0)| = \wt k_0$ and $S(\theta_0) = \{a_0+1,a_1+1,\cdots, a_{\wt k_0-1}+1\}$ with $0 = a_0< a_1<\cdots< a_{\wt k_0} = n$. Since $\theta_0\in\Theta_{k_0}(B)$, we must have $\wt k_0\leq k_0$. Define
\[\Theta_i = \left(\theta_{0i}-n^{-1/2},\theta_{0i}+n^{-1/2}\right),\mbox{ for }i \in S(\theta_0).\]
Then we define 
\[\Theta = \{\theta: \theta_i\in\Theta_i,\mbox{ for }i\in S(\theta_0), \theta_i = \theta_{i-1},\mbox{ for }i\not\in S(\theta_0)\}.\]
Then choose $dQ(\theta) = \frac{d\Pi(\theta)\1_{\Theta}(\theta)}{\Pi(\Theta)}$. As
\[dQ(\theta) = \prod_{i\in S(\theta_0)}\frac{g(\theta_i)\1_{\Theta_i}(\theta_i)}{\int_{\Theta_i}g(\theta_i)d\theta_i}\prod_{i\not\in S(\theta_0)}\delta_{\theta_{i-1}}(\theta_i)d\theta,\]
we have $Q\in\mathcal{S}_{\rm MC}$. For any $\theta\in\supp(Q) = \Theta$, \begin{eqnarray*}
D\left(P_{\theta^*}^{(n)}\|P_{\theta}^{(n)}\right) &=& \frac{1}{2\sigma^2}\|\theta^*-\theta\|^2\leq \frac{1}{\sigma^2}\|\theta^*-\theta_0\|^2+\frac{1}{\sigma^2}\|\theta_0-\theta\|^2\\
&\leq& \frac{1}{\sigma^2}\|\theta^*-\theta_0\|^2+\frac{1}{\sigma^2}
\leq D_2\left(P_{\theta^*}^{(n)}\|P_{\theta}^{(n)}\right)+\sigma^{-2}k_0\log n.
\end{eqnarray*}
Moreover, 
\begin{eqnarray*}
D\left(Q\|\Pi\right) &=& -\log\Pi(\Theta) = -\log\pi(\wt k_0)-\sum_{i\in S(\theta_0)}\log\left(\int_{\theta_{0i}-n^{-1/2}}^{\theta_{0i}+n^{-1/2}}g(\theta_i)d\theta_i\right)\\
&\lesssim& \wt k_0\log n\leq k_0\log n.
\end{eqnarray*}
Thus, condition (\ref{eq:C4**}) is satisfied for $\mathcal{S} = \mathcal{S}_{\rm MC}$.

When $\mathcal{S}= \mathcal{S}_{\rm MC}^{\rm joint}$. Choose $dQ^{\rm joint}(w,z,\theta) = dQ^{(w)}(w)\prod_{i = 2}^ndQ_i^{(z)}(z_i)dQ^{(\theta)}(\theta)$, where
\[Q^{(w)} = {\rm Beta}(\wt k_0-1+\alpha_0, n-\wt k_0+\beta_0),\]
 \[Q_{i}^{(z)}(z_i = 1) =\left\{\begin{array}{ll}0,&i\not\in S(\theta_0),\\ 1,&i\in S(\theta_0),\end{array}\right.\qquad\mbox{ for all $i>1$},\]
\[dQ^{(\theta)}(\theta) = \prod_{i\in S(\theta_0)}\frac{g(\theta_i)\1_{\Theta_i}(\theta_i)}{\int_{\Theta_i}g(\theta_i)d\theta_i}\prod_{i\not\in S(\theta_0)}\delta_{\theta_{i-1}}(\theta_i)d\theta,\]
Obviously, we have $Q^{\rm joint}\in\mathcal{S}_{\rm MC}^{\rm joint}$ and for any $\theta\in\supp(Q^{(\theta)})$, we have shown that 
\[D\left(P_{\theta^*}^{(n)}\|P_{\theta}^{(n)}\right)\lesssim D_2\left(P_{\theta^*}^{(n)}\|P_{\theta_0}^{(n)}\right)+k_0\log n.\]
On the other hand, suppose $dQ^{(\theta)}(\theta) = q^{(\theta)}(\theta)d\theta$ and $dQ^{(w)}(w) = q^{(w)}(w)dw$, we have
\begin{eqnarray*}
&&D\left(Q^{\rm joint}(w,z,\theta)\|\Pi(w,z,\theta)\right)\\
 &=& \iint q^{(w)}(w)q^{(\theta)}(\theta)\log\frac{q^{(w)}(w)q^{(\theta)}(\theta)}{\pi(w)w^{\wt k_0-1}(1-w)^{n-\wt k_0}\prod_{i\in S(\theta_0)}g(\theta_i)d\theta_i\prod_{i\not\in S(\theta_0)}\delta_{\theta_{i-1}}(\theta_i)}d\theta dw\\
 &=&-\log\pi(\wt k_0)-\sum_{i\in S(\theta_0)}\log\left(\int_{\theta_{0i}-n^{-1/2}}^{\theta_{0i}+n^{-1/2}}g(\theta_i)d\theta_i\right)\\
 &\lesssim&\wt k_0\log  n\leq k_0\log n.
 \end{eqnarray*}
Thus, condition (\ref{eq:C4**}) is satisfied for $\mathcal{S} = \mathcal{S}_{\rm MC}^{\rm joint}$. The proof is complete.

\end{proof}

\begin{proof}[Proof of Theorem \ref{thm:PC-MC-misp}]
By Lemma \ref{lem:PC-MC1} and Lemma \ref{lem:PC-MC2}, together with Theorem \ref{thm:misspecification} and Theorem \ref{thm:mis_MF}, we have
 \begin{eqnarray*}
 P_{\theta^*}^{(n)}\wh Q\|\theta-\theta_0\|^2\lesssim k_0\log n+\|\theta^*-\theta_0\|^2,
 \end{eqnarray*}
 for both $\wh{Q}=\wh{Q}_{\rm MC}$ and $\wh{Q}=\wh{Q}_{\rm MC}^{\rm joint}$.
 Then for every $1\leq k_0\leq n$ and $\theta_0\in\Theta_{k_0}(B)$, we have
 \[P_{\theta^*}^{(n)}\wh Q\|\theta-\theta^*\|^2\lesssim P_{\theta^*}^{(n)}\wh Q\|\theta-\theta_0\|^2+\|\theta_0-\theta^*\|^2\lesssim k_0\log n+\|\theta^*-\theta_0\|^2.\]
 Therefore, by taking minimum over $k_0\in[n]$ and $\theta_0\in\Theta_{k_0}(B)$, we can get
 \[P_{\theta^*}^{(n)}\wh Q\|\theta-\theta^*\|^2\lesssim\min_{1\leq k\leq n}\left\{\inf_{\theta_0\in\Theta_k(B)}\|\theta^*-\theta_0\|^2+k\log n\right\}.\]
The proof is complete. 
\end{proof}

\subsection{Proof of Theorem \ref{thm:convergence-general}}

\begin{proof}[Proof of Theorem \ref{thm:convergence-general}]
By  Lemma \ref{lem:basic ineq}, we have
\[a\wh Q_*L(P_\theta^{(n)}, P_0^{(n)})\leq D(\wh Q_*\|\Pi(\cdot|X^{(n)}))+\log\Pi(\exp(aL(P_\theta^{(n)}, P_0^{(n)}))|X^{(n)}).\]
Then, under the conditions of Theorem \ref{thm:convergence-general}, we have
$$D(\wh Q_*\|\Pi(\cdot|X^{(n)}))\leq D_*(\wh Q_*\|\Pi(\cdot|X^{(n)}))\leq D_*(Q\|\Pi(\cdot|X^{(n)})),$$
for all $Q\in\mathcal{S}$. Then, following the same argument in the proof of Theorem \ref{thm:convergence}, we complete the proof.
\end{proof}

\begin{small}
\bibliographystyle{plainnat}
\bibliography{reference}
\end{small}




\end{document}